\documentclass[10pt, reqno,amsmath,amsthm,amssymb,amscd]{amsart}
\usepackage{amsmath}
\usepackage{mathrsfs,amssymb, amscd,amsmath,amsthm}
\usepackage[enableskew,vcentermath]{youngtab}
\usepackage{multicol}\multicolsep=0pt
\usepackage{tikz}
\newcommand{\END}{\mathcal{E}nd}
\newcommand{\clr}{rgb:black,1;blue,4;red,1}

\newcommand{\bdot}{ node[circle, draw, fill=\clr, thick, inner sep=0pt, minimum width=4pt]{}}
\newcommand{\rdot}{ node[circle, draw=darkred, fill=darkred, thick, inner sep=0pt, minimum width=4pt]{}}

\newcommand{\ob}[1]{\mathsf{#1}}

\newcommand{\down}{\downarrow}

\newcommand{\AOB}{\mathcal{AOB}}

\newcommand{\OB}{\mathcal{OB}}

\newcommand{\lcap}{
\begin{tikzpicture}[baseline = 3pt, scale=0.5, color=\clr]
        \draw[-,thick] (1,0) to[out=up, in=right] (0.53,0.5) to[out=left, in=right] (0.47,0.5);
        \draw[->,thick] (0.49,0.5) to[out=left,in=up] (0,0);
\end{tikzpicture}
}
\newcommand{\lcup}{
\begin{tikzpicture}[baseline = 6pt, scale=0.5, color=\clr]
        \draw[-,thick] (1,1) to[out=down, in=right] (0.53,0.5) to[out=left, in=right] (0.47,0.5);
        \draw[->,thick] (0.49,0.5) to[out=left,in=down] (0,1);
\end{tikzpicture}
}
\newcommand{\lcapl}{
\begin{tikzpicture}[baseline = 3pt, scale=0.5, color=\clr]
        \draw[<-,thick] (1,0) to[out=up, in=right] (0.53,0.5);
         \draw[-,thick]  (0.53,0.5) to[out=left, in=right] (0.47,0.5);
        \draw[-,thick] (0.49,0.5) to[out=left,in=up] (0,0);
\end{tikzpicture}
}
\newcommand{\lcupl}{
\begin{tikzpicture}[baseline = 6pt, scale=0.5, color=\clr]
        \draw[<-,thick] (1,1) to[out=down, in=right] (0.53,0.5);
        \draw[-,thick](0.53,0.5) to[out=left, in=right] (0.47,0.5);
        \draw[-,thick] (0.49,0.5) to[out=left,in=down] (0,1);
\end{tikzpicture}
}

\newcommand{\rcap}{
\begin{tikzpicture}[baseline = 3pt, scale=0.5, color=\clr]
        \draw[<-,thick] (1,0) to[out=up, in=right] (0.53,0.5) to[out=left, in=right] (0.47,0.5);
        \draw[-,thick] (0.49,0.5) to[out=left,in=up] (0,0);
\end{tikzpicture}
}

\newcommand{\xd}{
\begin{tikzpicture}[baseline = 3pt, scale=0.5, color=\clr]

\draw[->,thick] (0,0) to[out=up, in=down] (0,1);
\draw(0,0.5) \bdot;
\end{tikzpicture}
}
\newcommand{\xdx}{
\begin{tikzpicture}[baseline = 3pt, scale=0.5, color=\clr]
\draw[<-,thick] (0,0) to[out=up, in=down] (0,1);
\draw(0,0.5) \bdot;\end{tikzpicture}
}

\newcommand{\xli}{
\begin{tikzpicture}[baseline = 3pt, scale=0.5, color=\clr]
        \draw[<-,thick] (0,0) to[out=up, in=down] (0,1);
\end{tikzpicture}
}
\newcommand{\sli}{
\begin{tikzpicture}[baseline = 3pt, scale=0.5, color=\clr]
        \draw[->,thick] (0,0) to[out=up, in=down] (0,1);
\end{tikzpicture}
}

 \providecommand{\og}{``}
\providecommand{\fg}{''} \providecommand{\smfandname}{and}

\usepackage{amssymb}
\usepackage{mathrsfs,amssymb}
\usepackage{multicol}\multicolsep=0pt
\usepackage{pstricks,pst-node}

\usepackage[enableskew,vcentermath]{youngtab}
\usepackage[sort]{cite}
\usepackage{xcolor,graphicx}

\def\crulefill{\leavevmode\leaders\hrule height 1pt\hfill\kern 0pt}
\long\def\QUERY#1{%
\leavevmode\newline%
\noindent$\star\star\star$\thinspace\textsf{Comment/Query}\crulefill\newline%
   \space #1\newline\hbox to 120mm{\crulefill}$\star\star\star$\newline}
\newtheorem{Theorem}{Theorem}[section]
\newtheorem{Lemma}[Theorem]{Lemma}
\newtheorem{Cor}[Theorem]{Corollary}
\newtheorem{Prop}[Theorem]{Proposition}

 \theoremstyle{definition}

\newtheorem{Defn}[Theorem]{Definition}

\newtheorem{rem}[Theorem]{Remark}

\numberwithin{equation}{section}
\theoremstyle{definition}

\makeatletter
\def\enumerate{\begingroup\ifnum\@enumdepth>3\@toodeep\else
      \advance\@enumdepth\@ne
      \edef\@enumctr{enum\romannumeral\the\@enumdepth}%
      \topsep\z@\parskip\z@
      \list{\csname label\@enumctr\endcsname}
        {\@nmbrlisttrue\let\@listctr\@enumctr
         \parsep\z@\itemsep\z@\topsep\z@
         \setcounter{\@enumctr}{0}
         \def\makelabel##1{\hss\llap{\rm ##1}}
       }\fi}

\makeatother

\let\bar=\overline
\let\epsilon=\varepsilon
\def\({\big(}
\def\){\big)}

\def\0{\underline{0}}

\DeclareMathOperator{\End}{End}

\def\Hom{\text{Hom}}

{\catcode`\|=\active
  \gdef\set#1{\mathinner{\lbrace\,{\mathcode`\|"8000%
                                   \let|\midvert #1}\,\rbrace}}
  \gdef\seT#1{\mathinner{\Big\lbrace\,{\mathcode`\|"8000%
                                   \let|\midverT #1}\,\Big\rbrace}}
}
\def\midvert{\egroup\mid\bgroup}
\def\midverT{\egroup\,\Big|\,\bgroup}

\def\Set[#1]#2|#3|{\Big\{\ #2\ \Big| \
           \vcenter{\hsize #1mm\centering #3}\Big\}}

\def\qed{\hfill\mbox{$\Box$}}

\def\Hom{{\rm Hom}}

\def\Set{{\rm Set}}

\def\Hom{\text{Hom}}%
\def\textsf#1{{\textit{#1}}}%

\definecolor{white}{HTML}{FFFFFF}
\definecolor{darkblue}{HTML}{111199}
\definecolor{darkgreen}{HTML}{336633}
\definecolor{darkred}{HTML}{993333}

\definecolor{darkpurple}{HTML}{995599}

\begin{document}
\title[{\tiny  Representations of cyclotomic oriented Brauer categories}]{Representations of cyclotomic oriented Brauer categories}
\author{Mengmeng Gao, Hebing Rui, Linliang Song}
\address{M.G.  School of Mathematical Science, Tongji University,  Shanghai, 200092, China}\email{g19920119@163.com}
\address{H.R.  School of Mathematical Science, Tongji University,  Shanghai, 200092, China}\email{hbrui@tongji.edu.cn}
\address{L.S.  School of Mathematical Science, Tongji University,  Shanghai, 200092, China}\email{llsong@tongji.edu.cn}

\thanks{ M. Gao and H. Rui is supported  partially by NSFC (grant No.  11971351).  L. Song is supported  partially by NSFC (grant No.  11501368). }
\sloppy
\maketitle\begin{abstract} Let $A$ be  the locally unital  algebra associated to a cyclotomic oriented Brauer category over an arbitrary algebraically closed field $\Bbbk$ of   characteristic $p\ge 0$. The category of locally finite dimensional representations of  $A $ is  used to   give   the  tensor product categorification (in the general sense of Losev and Webster) for an integrable  lowest  weight with  an    integrable  highest weight representation of the same level  for the Lie algebra $\mathfrak g$,  where $\mathfrak g$ is a direct sum of copies of  $   \mathfrak {sl}_\infty$ (resp.,    $  \hat{\mathfrak {sl}}_p$ ) if $p=0$ (resp., $p>0$). Such a result  was  expected  in \cite{BCNR} when $\Bbbk=\mathbb C$ and  proved previously in \cite{Br} when the level is $1$.
\end{abstract}
\maketitle

\section{Introduction}
Throughout this paper,  $\Bbbk$ is an arbitrary algebraically closed field of characteristic $p\ge 0$. Unless otherwise stated, all algebras,  categories and functors  are assumed to be $\Bbbk$-linear.

 Fix $\ell\in\mathbb Z^{>0}$, $\mathbf u, \mathbf u'\in \Bbbk^\ell$ and set \begin{equation}\label{mathbfu} \mathbf u=(u_1,\ldots, u_\ell), \quad  \mathbf u'=(u_1',\ldots, u_\ell').\end{equation} In \cite{BCNR}, Brundan et al. introduced the notion of the cyclotomic oriented Brauer category $\OB(\mathbf u, \mathbf u')$.  When $\ell=1$,  $\OB(\mathbf u, \mathbf u')$,  called  the oriented Brauer category,
 is monoidally equivalent to the  category    in \cite[Example~1.26]{DM}. The aim of this paper is to give a tensor product categorification of any integrable lowest weight with  any integrable  highest weight module of the same level  for  $\mathfrak{g}$   by using locally finite dimensional representations of cyclotomic oriented Brauer categories, where $\mathfrak{g}$ is
a direct sum of copies of $\mathfrak{sl}_\infty$ (resp., $\widehat{\mathfrak{sl}}_p$)  if $p=0$ (resp., $p>0$).

 Let $A_\ell$ be   the locally unital algebra associated to $\OB(\mathbf u, \mathbf u')$.  Reynolds proved that $A_1$ admits a triangular decomposition~\cite{Re}.
   Brundan  further showed that the locally unital algebra $B$  associated to  the oriented skein category (which  was introduced by   Turaev  and was called  HOMFLY-PT skein category\cite{Tur}) also admits a triangular decomposition. For any locally unital and locally finite dimensional algebra $C$,
 let $C$-lfdmod be the category of all  locally finite dimensional left $C$-modules. The (upper finite) triangular decomposition property is the key in the proof  that   both $A_1$-lfdmod and $B$-lfdmod   are  upper finite fully stratified categories in the sense of Brundan and Stroppel~\cite{Br, BS}. Brundan also constructed  certain  endofunctors of $A_1$-lfdmod and  $B$-lfdmod and its associated graded categories so as  to give a tensor product categorification of an integrable  lowest weight with  an integrable highest weight module of the same level $1$ for $\mathfrak{g}$.
    However, it is unclear to us whether $A_\ell$ admits an upper finite triangular decomposition in general. So, one could  not use arguments on upper finite triangular decompositions to show that  $A_\ell$-lfdmod is an upper finite fully stratified category.

   In \cite{GRS3}, we introduced and studied the upper finite  weakly triangular categories.
   It is easy to see that any category whose associated locally unital algebra  admits an upper finite  triangular decomposition is an  upper finite weakly triangular category. Moreover,
  cyclotomic oriented Brauer categories, cyclotomic Brauer categories~\cite{RS3} and  cyclotomic Kauffman categories~\cite{GRS2} are  upper finite weakly triangular categories~\cite{GRS3}.  We  established in \cite{GRS3} that $C$-lfdmod is an upper finite fully stratified category if $C$ is the  locally  unital algebra  associated to an upper finite weakly triangular category. In particular, $A_\ell$-lfdmod is an upper finite fully stratified category. 

Set \begin{equation}\label{iu} \mathbb I=\mathbb I_\mathbf u \bigcup \mathbb I_{\mathbf u'}, \end{equation} where   $\mathbb I_\mathbf u=\{u_j+n| 1\leq j\leq \ell, n\in\mathbb Z\}$ and  $   \mathbb I_{\mathbf u'}=\{u'_j+n| 1\leq j\leq \ell, n\in\mathbb Z\}$, and  $u_1, \ldots, u_\ell$ and $u_1', \ldots, u_\ell'$ are given in \eqref{mathbfu}.
Let $\mathfrak g$ be the complex Kac-Moody Lie algebra   associated to   Cartan matrix $(a_{i,j})_{i,j\in \mathbb I}$, where  \begin{equation}\label{lie} a_{i,j}=\begin{cases}
    2, & \text{if $i=j$;} \\
    -1, & \text{if $i=j\pm 1$ and $p\neq 2$;} \\
     -2, & \text{if $i=j-1$ and $p= 2$;} \\
    0, & \hbox{ otherwise.}
  \end{cases}
  \end{equation}
  Then $\mathfrak g$ is isomorphic to  a direct sum of copies of  $\mathfrak {sl}_\infty$ (resp., $\hat{\mathfrak {sl}}_p$) if $p=0$ (resp., $p>0$) depending on the number of  $\mathbb Z$-orbits of $\{u_1, \ldots, u_\ell, u_1', \ldots, u_\ell'\}$ in the sense that, for any $x, y\in \{u_1, \ldots, u_\ell, u_1', \ldots, u_\ell'\}$,  $x$ and $y$ are in the same $\mathbb Z$-orbit if and only if  $x-y\in \mathbb Z1_\Bbbk$.
 Let  $V(\omega_{\mathbf u})$  be the integrable highest weight $\mathfrak g$-module of weight
 \begin{equation}\label{param}  \omega_{\mathbf u} =\sum_{i=1}^\ell \omega_{u_i},\end{equation}  where $\omega_{u_i}$'s are fundamental weights and $\ell$ is said to be the level.
Similarly, let    $\tilde V(-\omega_{\mathbf u'})$  be the
integrable lowest weight $\mathfrak g$-module of weight $-\omega_{\mathbf u'}$.  The following is the main result  of this paper. It was obtained previously in \cite{Br} when $\ell=1$.

  \begin{Theorem}\label{main} Let $A$ be the locally unital  $\Bbbk$-algebra associated to the cyclotomic oriented Brauer category $\OB(\mathbf u, \mathbf u')$. Then  $A$-lfdmod   admits the structure of a tensor product categorification of the $\mathfrak g$-module
$\tilde V(-\omega_{\mathbf u'})\otimes V(\omega_\mathbf u)$ in the general sense of Losev and Webster (e.g., Definition~\ref{defioftpc}).\end{Theorem}

If $\Bbbk=
\mathbb C$ and  $\mathbf u\bigcup\mathbf u'$ consists of a unique  orbit, then  the category $A$-lfdmod  admits the structure of a tensor product categorification of $\mathfrak {sl}_\infty$-module
$\tilde V(-\omega_{\mathbf u'})\otimes V(\omega_\mathbf u)$. Such a result was expected in \cite{Br}.

 In order to prove Theorem~\ref{main},   we study representations of cyclotomic oriented Brauer categories. Our results include  a classification of simple $A$-modules, a criterion on the complete reducibility of the category of left $A$-modules, certain partial results on blocks, and certain endofunctors of $A$-lfdmod and its associated graded categories and so on.

After posting the paper on ArXiv we were informed that the category $\OB(\mathbf u, \mathbf u')$ is a  generalized cyclotomic quotient (GCQ) of the Heisenberg category (of central charge zero) \cite{Br1}.
Brundan, Savage and Webster proved that GCQs of the Heisenberg category are isomorphic to GCQs of a corresponding Kac-Moody 2-category\cite{BSW}.  Webster showed in \cite{We1}  that  for all GCQs of all Kac-Moody 2-categories, the category of locally finite-dimensional modules is a tensor product categorification, in particular, this category is upper finite fully stratified.  So Theorem ~\ref{main}  could  also be obtained by applying the isomorphism between GCQS of Heisenberg and Kac-Moody 2-category and tensor product categorification for GCQS of Kac-Moody 2-category. We are also told recently by Webster  that  even if one  knows that two results are equivalent under a theorem that translates between two contexts, it can still make more sense to prove them in the other context rather than trying to translate when the translation is sufficiently complicated. Moreover, as pointed in \cite{BS}, GCQs of the Heisenberg category of central charge different from zero should possess  upper finite weakly triangular decompositions and hence   Theorem~\ref{main} can be generalized to those more general GCQs of Heisenberg category by using the theory of upper finite weakly triangular category in \cite{GRS3}. However, it seems hard to prove the expected basis theorem for a GCQ of Heisenberg category with central charge different from zero.

 We organize  this paper as follows. In section~2, we recall some elementary results on
cyclotomic oriented Brauer categories in \cite{BCNR, GRS3}. We study representations of cyclotomic oriented Brauer categories  in section~3 and prove  Theorem~\ref{main}  in section~4.

\textbf{Acknowledgements.}  We would like to thanks Jonathan Brundan and Ben Webster for helpful comments and detailed  explanation  on the results in \cite{BSW,We1}.

\section{Cyclotomic oriented Brauer categories}
In this section, we recall some  results on cyclotomic oriented Brauer categories.
In  the usual string calculus for strict monoidal categories,  the composition $g\circ h$  of two morphisms $g$ and $h$   is given by vertical   stacking and the tensor product $g\otimes h $ is given by   horizontal concatenation. More explicitly, \begin{equation}\label{com1}
      g\circ h= \begin{tikzpicture}[baseline = 19pt,scale=0.5,color=\clr,inner sep=0pt, minimum width=11pt]
        \draw[-,thick] (0,0) to (0,3);
        \draw (0,2.2) node[circle,draw,thick,fill=white]{$g$};
        \draw (0,0.8) node[circle,draw,thick,fill=white]{$h$};
    \end{tikzpicture}
    ~,~ \ \ \ \ \ g\otimes h=\begin{tikzpicture}[baseline = 19pt,scale=0.5,color=\clr,inner sep=0pt, minimum width=11pt]
        \draw[-,thick] (0,0) to (0,3);
        \draw[-,thick] (2,0) to (2,3);
        \draw (0,1.5) node[circle,draw,thick,fill=white]{$g$};
        \draw (2,1.5) node[circle,draw,thick,fill=white]{$h$};
    \end{tikzpicture}~.
\end{equation}

\subsection{The category $\OB$} \textsf{The oriented Brauer category} $\OB$ was introduced in \cite{BCNR}. It is  the strict monoidal category generated by two generating objects  $\uparrow, \down$ and four elementary morphisms
\begin{equation} \label{gen}
 \lcup:\emptyset\rightarrow \uparrow\otimes \downarrow, \ \ \  \lcap: \downarrow\otimes \uparrow\rightarrow\emptyset, \ \ \ \text{ \begin{tikzpicture}[baseline = 2.5mm]
	\draw[->,thick,darkblue] (0.28,0) to[out=90,in=-90] (-0.28,.6);
	\draw[->,thick,darkblue] (-0.28,0) to[out=90,in=-90] (0.28,.6);
\end{tikzpicture}}: \uparrow\otimes\uparrow\rightarrow\uparrow\otimes\uparrow, \ \ \ \text{ \begin{tikzpicture}[baseline = 2.5mm]
	\draw[<-,thick,darkblue] (0.28,0) to[out=90,in=-90] (-0.28,.6);
	\draw[->,thick,darkblue] (-0.28,0) to[out=90,in=-90] (0.28,.6);
\end{tikzpicture} }:\uparrow\otimes\downarrow\rightarrow\downarrow\otimes\uparrow,  \end{equation}
which  satisfy the  relations~\cite[(1.4)-(1.8)]{BCNR} :

\begin{equation}\label{relation 1}\begin{tikzpicture}[baseline = -0.5mm]
\draw[->,thick,darkblue] (0,0) to (0,0.3);
\draw[-,thick,darkblue] (0,0) to[out=down,in=left] (0.25,-0.25) to[out=right,in=down] (0.5,0);
\draw[-,thick,darkblue] (0.5,0) to[out=up,in=left] (0.75,0.25) to[out=right,in=up] (1,0);
\draw[-,thick,darkblue] (0,0) to (0,0.3);
\draw[-,thick,darkblue] (1,0) to (1,-0.3);
  \end{tikzpicture}
  ~=~
  \begin{tikzpicture}[baseline = -0.5mm]
  \draw[->,thick,darkblue] (0,-0.25) to (0,0.25);
  \end{tikzpicture},
  \end{equation}

  \begin{equation} \label{relation 2}
   \begin{tikzpicture}[baseline = -0.5mm]
\draw[-,thick,darkblue] (0.5,0) to[out=down,in=left] (0.75,-0.25) to[out=right,in=down] (1,0);
\draw[-,thick,darkblue] (0,0) to[out=up,in=left] (0.25,0.25) to[out=right,in=up] (0.5,0);
\draw[-,thick,darkblue] (1,0) to (1,0.3);
\draw[->,thick,darkblue] (0,0) to (0,-0.3);
  \end{tikzpicture}~=~\begin{tikzpicture}[baseline = -0.5mm]
  \draw[<-,thick,darkblue] (0,-0.25) to (0,0.25);
  \end{tikzpicture},\end{equation}
 \begin{equation}\label{relation 3}\mathord{
\begin{tikzpicture}[baseline = -1mm]
	\draw[-,thick,darkblue] (0.28,-.6) to[out=90,in=-90] (-0.28,0);
	\draw[->,thick,darkblue] (-0.28,0) to[out=90,in=-90] (0.28,.6);
	\draw[-,thick,darkblue] (-0.28,-.6) to[out=90,in=-90] (0.28,0);
	\draw[->,thick,darkblue] (0.28,0) to[out=90,in=-90] (-0.28,.6);
\end{tikzpicture}
}=
\mathord{
\begin{tikzpicture}[baseline = -1mm]
	\draw[->,thick,darkblue] (0.18,-.6) to (0.18,.6);
	\draw[->,thick,darkblue] (-0.18,-.6) to (-0.18,.6);
\end{tikzpicture},
}\end{equation}

\begin{equation} \label{relation 4}
\mathord{
\begin{tikzpicture}[baseline = -1mm]
	\draw[->,thick,darkblue] (0.45,-.6) to (-0.45,.6);
        \draw[-,thick,darkblue] (0,-.6) to[out=90,in=-90] (-.45,0);
        \draw[->,thick,darkblue] (-0.45,0) to[out=90,in=-90] (0,0.6);
	\draw[<-,thick,darkblue] (0.45,.6) to (-0.45,-.6);
\end{tikzpicture}
}=
\mathord{
\begin{tikzpicture}[baseline = -1mm]
	\draw[->,thick,darkblue] (0.45,-.6) to (-0.45,.6);
        \draw[-,thick,darkblue] (0,-.6) to[out=90,in=-90] (.45,0);
        \draw[->,thick,darkblue] (0.45,0) to[out=90,in=-90] (0,0.6);
	\draw[<-,thick,darkblue] (0.45,.6) to (-0.45,-.6);
\end{tikzpicture}},\end{equation}
\begin{equation}\label{relation 5} \text{\begin{tikzpicture}[baseline = 2.5mm]
	\draw[<-,thick,darkblue] (0.28,0) to[out=90,in=-90] (-0.28,.6);
	\draw[->,thick,darkblue] (-0.28,0) to[out=90,in=-90] (0.28,.6);
\end{tikzpicture} is the two-sided  inverse to \begin{tikzpicture}[baseline = -0.5mm]
\draw[-,thick,darkblue]  (0.75,-0.25) to[out=right,in=down] (1,0);
\draw[-,thick,darkblue] (0,0) to[out=up,in=left] (0.25,0.25) ;
\draw[-,thick,darkblue] (1,0) to (1,0.25);
\draw[->,thick,darkblue] (0,0) to (0,-0.25);
\draw[->,thick,darkblue] (0.25,-0.25) to[out=90,in=-90] (0.75,0.25);
\draw[-,thick,darkblue] (0.25,0.25) to[out=0,in=180] (0.75,-0.25);
  \end{tikzpicture}.}\end{equation}
Here,  \begin{tikzpicture}[baseline = -0.5mm]
  \draw[<-,thick,darkblue] (0,-0.25) to (0,0.25);
  \end{tikzpicture} is  the identity morphism from \begin{tikzpicture}[baseline = -0.5mm]
  \draw[<-,thick,darkblue] (0,-0.25) to (0,0.25);
  \end{tikzpicture} to \begin{tikzpicture}[baseline = -0.5mm]
  \draw[<-,thick,darkblue] (0,-0.25) to (0,0.25);
  \end{tikzpicture}, and
\begin{tikzpicture}[baseline = -0.5mm]
  \draw[->,thick,darkblue] (0,-0.25) to (0,0.25);
  \end{tikzpicture} is  the identity morphism from \begin{tikzpicture}[baseline = -0.5mm]
  \draw[->,thick,darkblue] (0,-0.25) to (0,0.25);
  \end{tikzpicture} to \begin{tikzpicture}[baseline = -0.5mm]
  \draw[->,thick,darkblue] (0,-0.25) to (0,0.25);
  \end{tikzpicture}.
In the following, the inverse of \begin{tikzpicture}[baseline = 2.5mm]
	\draw[<-,thick,darkblue] (0.28,0) to[out=90,in=-90] (-0.28,.6);
	\draw[->,thick,darkblue] (-0.28,0) to[out=90,in=-90] (0.28,.6);
\end{tikzpicture}  in \eqref{relation 5}  will be denoted by \begin{tikzpicture}[baseline = 2.5mm]
	\draw[->,thick,darkblue] (0.28,0) to[out=90,in=-90] (-0.28,.6);
	\draw[<-,thick,darkblue] (-0.28,0) to[out=90,in=-90] (0.28,.6);
\end{tikzpicture}.
For any $\delta_1\in\Bbbk$, let $\OB(\delta_1)$ be the category obtained from $\OB$ by adding the  additional relation
\begin{equation}\label{deltaa}  \begin{tikzpicture}[baseline = -0.5mm]
\draw[-,thick,darkblue] (0,0) to[out=down,in=left] (0.25,-0.25) to[out=right,in=down] (0.5,0);
  \draw[-,thick,darkblue] (0,0) to[out=up,in=left] (0.25,0.25) to[out=right,in=up] (0.5,0);
  \draw[->,thick,darkblue]  (0.25,0.25) to[out=right,in=up] (0.5,0);
  \end{tikzpicture}=\delta_{1}.\end{equation}
\subsection{The category $\AOB$}\label{affK} \textsf{The affine oriented Brauer category} $\AOB$ was also introduced in \cite{BCNR}. It is the strict  monoidal category generated by two generating objects  $\uparrow, \down$ and
the morphism \begin{tikzpicture}[baseline=-.5mm]
\draw[->,thick,darkblue] (0,-.3) to (0,.3);
      \draw (0,0) \bdot;
\end{tikzpicture} together with four morphisms in \eqref{gen} subject to the relations in \eqref{relation 1}--\eqref{relation 5} and one extra relation:
\begin{equation}\label{relation 6}
\begin{tikzpicture}[baseline = -1mm]
	\draw[->,thick,darkblue] (0.28,-0.3) to[out=90,in=-90] (-0.28,.3);
	\draw[->,thick,darkblue] (-0.28,-.3) to[out=90,in=-90] (0.28,.3);
 \draw (0.22,.11) \bdot;
\end{tikzpicture}~=~
\begin{tikzpicture}[baseline = -7mm]

\draw(-0.25,-.75) \bdot;
\draw[->,thick,darkblue] (0.28,-0.9) to[out=90,in=-90] (-0.28,-.3);
	\draw[->,thick,darkblue] (-0.28,-.9) to[out=90,in=-90] (0.28,-.3);
  \end{tikzpicture}
  ~+~\begin{tikzpicture}[baseline = -1mm]
 \draw[<-,thick,darkblue] (0,0.3) to (0,-.3);
 \draw[<-,thick,darkblue] (0.56,0.3) to (0.56,-.3);
  \end{tikzpicture}.
\end{equation}
Thanks to  \cite[(3.3)-(3.6), (3.8)]{BCNR},
\begin{equation}\label{relation 7}\begin{tikzpicture}[baseline = -0.5mm]
\draw[-,thick,darkblue] (0,0) to[out=down,in=left] (0.25,-0.25) to[out=right,in=down] (0.5,0);
\draw[-,thick,darkblue] (0.5,0) to[out=up,in=left] (0.75,0.25) to[out=right,in=up] (1,0);
\draw[-,thick,darkblue] (0,0) to (0,0.3);
\draw[->,thick,darkblue] (1,0) to (1,-0.3);
  \end{tikzpicture}
  ~=~
  \begin{tikzpicture}[baseline = -0.5mm]
  \draw[<-,thick,darkblue] (0,-0.25) to (0,0.25);
  \end{tikzpicture},\end{equation}

  \begin{equation} \label{relation 8} \begin{tikzpicture}[baseline = -0.5mm]
\draw[-,thick,darkblue] (0.5,0) to[out=down,in=left] (0.75,-0.25) to[out=right,in=down] (1,0);
\draw[-,thick,darkblue] (0,0) to[out=up,in=left] (0.25,0.25) to[out=right,in=up] (0.5,0);
\draw[->,thick,darkblue] (1,0) to (1,0.3);
\draw[-,thick,darkblue] (0,0) to (0,-0.3);
  \end{tikzpicture}~=~\begin{tikzpicture}[baseline = -0.5mm]
  \draw[->,thick,darkblue] (0,-0.25) to (0,0.25);
  \end{tikzpicture},\end{equation}

   \begin{equation}\label{relation 9}\mathord{
\begin{tikzpicture}[baseline = -1mm]
	\draw[<-,thick,darkblue] (0.28,-.6) to[out=90,in=-90] (-0.28,0);
	\draw[-,thick,darkblue] (-0.28,0) to[out=90,in=-90] (0.28,.6);
	\draw[<-,thick,darkblue] (-0.28,-.6) to[out=90,in=-90] (0.28,0);
	\draw[-,thick,darkblue] (0.28,0) to[out=90,in=-90] (-0.28,.6);
\end{tikzpicture}
}=
\mathord{
\begin{tikzpicture}[baseline = -1mm]
	\draw[<-,thick,darkblue] (0.18,-.6) to (0.18,.6);
	\draw[<-,thick,darkblue] (-0.18,-.6) to (-0.18,.6);
\end{tikzpicture},}\end{equation}

\begin{equation} \label{relation 10}  \mathord{
\begin{tikzpicture}[baseline = -1mm]
	\draw[<-,thick,darkblue] (0.45,-.6) to (-0.45,.6);
        \draw[<-,thick,darkblue] (0,-.6) to[out=90,in=-90] (-.45,0);
        \draw[-,thick,darkblue] (-0.45,0) to[out=90,in=-90] (0,0.6);
	\draw[->,thick,darkblue] (0.45,.6) to (-0.45,-.6);
\end{tikzpicture}
}=
\mathord{
\begin{tikzpicture}[baseline = -1mm]
	\draw[<-,thick,darkblue] (0.45,-.6) to (-0.45,.6);
        \draw[<-,thick,darkblue] (0,-.6) to[out=90,in=-90] (.45,0);
        \draw[-,thick,darkblue] (0.45,0) to[out=90,in=-90] (0,0.6);
	\draw[->,thick,darkblue] (0.45,.6) to (-0.45,-.6);
\end{tikzpicture}},\end{equation}

\begin{equation}\label{relation 11}
\begin{tikzpicture}[baseline = -7mm]
\draw (0.25,-.45) \bdot;
\draw[<-,thick,darkblue] (0.28,-0.9) to[out=90,in=-90] (-0.28,-.3);
	\draw[<-,thick,darkblue] (-0.28,-.9) to[out=90,in=-90] (0.28,-.3);
  \end{tikzpicture}~=~\begin{tikzpicture}[baseline = -1mm]
	\draw[<-,thick,darkblue] (0.28,-0.3) to[out=90,in=-90] (-0.28,.3);
	\draw[<-,thick,darkblue] (-0.28,-.3) to[out=90,in=-90] (0.28,.3);
\draw(-0.24,-.13) \bdot;
\end{tikzpicture}
~-~\begin{tikzpicture}[baseline = -1mm]
 \draw[->,thick,darkblue] (0,0.3) to (0,-.3);
 \draw[->,thick,darkblue] (0.56,0.3) to (0.56,-.3);
  \end{tikzpicture},
\end{equation}where
\begin{equation}\label{gen1} \text{ $\lcupl:=\begin{tikzpicture}[baseline = 2.5mm]
	\draw[-,thick,darkblue] (0.28,0) to[out=90,in=-90] (-0.28,.6);
	\draw[->,thick,darkblue] (-0.28,0) to[out=90,in=-90] (0.28,.6);
\draw[-,thick,darkblue] (-0.28,0) to[out=down,in=left] (0,-0.25) to[out=right,in=down] (0.28,0);
\end{tikzpicture}, \quad\quad  \lcapl:=\begin{tikzpicture}[baseline = 2.5mm]
	\draw[<-,thick,darkblue] (0.28,0) to[out=90,in=-90] (-0.28,.6);
\draw[-,thick,darkblue] (-0.28,0.6) to[out=up,in=left] (0,0.85) to[out=right,in=up] (0.28,0.6);
	\draw[-,thick,darkblue] (-0.28,0) to[out=90,in=-90] (0.28,.6);
\end{tikzpicture},\quad\quad \begin{tikzpicture}[baseline=-.5mm]
\draw[<-,thick,darkblue] (0,-.3) to (0,.3);
     \draw (0,0) \bdot;
\end{tikzpicture}:=\begin{tikzpicture}[baseline = -0.5mm]
\draw[-,thick,darkblue] (0.5,0) to[out=down,in=left] (0.75,-0.25) to[out=right,in=down] (1,0);
 \draw (0.5,0)\bdot;
\draw[-,thick,darkblue] (0,0) to[out=up,in=left] (0.25,0.25) to[out=right,in=up] (0.5,0);
\draw[-,thick,darkblue] (1,0) to (1,0.3);
\draw[->,thick,darkblue] (0,0) to (0,-0.3);
  \end{tikzpicture},\quad\quad   \begin{tikzpicture}[baseline = -7mm]
\draw[<-,thick,darkblue] (0.28,-0.9) to[out=90,in=-90] (-0.28,-.3);
	\draw[<-,thick,darkblue] (-0.28,-.9) to[out=90,in=-90] (0.28,-.3);
  \end{tikzpicture}:=\begin{tikzpicture}[baseline = -0.5mm]
\draw[-,thick,darkblue]  (0.75,-0.25) to[out=right,in=down] (1,0);
\draw[-,thick,darkblue] (0,0) to[out=up,in=left] (0.25,0.25) ;
\draw[-,thick,darkblue] (1,0) to (1,0.25);
\draw[-,thick,darkblue] (-0.25,0.25) to[out=up,in=left] (0.25,0.5) to[out=right,in=up] (0.75,0.25);
\draw[-,thick,darkblue] (0.25,-0.25) to[out=down,in=left] (0.75,-0.5) to[out=right,in=down] (1.25,-0.25);
\draw[->,thick,darkblue] (0,0) to (0,-0.25);
\draw[->,thick,darkblue] (-.25,0.25) to (-.25,-0.25);
\draw[-,thick,darkblue] (1.25,-0.25) to (1.25,0.25);
\draw[-,thick,darkblue] (0.25,0.25) to[out=0,in=180] (0.75,-0.25);
\draw[-,thick,darkblue] (0.25,-0.25) to[out=90,in=-90] (0.75,0.25);
  \end{tikzpicture}$. }\end{equation}

   \begin{Lemma}\label{sigma} There is a $\Bbbk$-linear  monoidal contravariant functor $\tau: \AOB\rightarrow \AOB$, which fixes generating objects and  both \begin{tikzpicture}[baseline = 2.5mm]
	\draw[->,thick,darkblue] (0.28,0) to[out=90,in=-90] (-0.28,.6);
	\draw[->,thick,darkblue] (-0.28,0) to[out=90,in=-90] (0.28,.6);
\end{tikzpicture} and \begin{tikzpicture}[baseline=-.5mm]
\draw[->,thick,darkblue] (0,-.3) to (0,.3);
      \node at (0,0) {$\color{darkblue}\scriptstyle\bullet$};
\end{tikzpicture} and switches  $\lcup$ (resp., $\lcap$, resp.,  \begin{tikzpicture}[baseline = 2.5mm]
	\draw[<-,thick,darkblue] (0.28,0) to[out=90,in=-90] (-0.28,.6);
	\draw[->,thick,darkblue] (-0.28,0) to[out=90,in=-90] (0.28,.6);
\end{tikzpicture}) to $\lcapl$
(resp., $\lcupl$, resp.,\begin{tikzpicture}[baseline = 2.5mm]
	\draw[->,thick,darkblue] (0.28,0) to[out=90,in=-90] (-0.28,.6);
	\draw[<-,thick,darkblue] (-0.28,0) to[out=90,in=-90] (0.28,.6);
\end{tikzpicture}).
Furthermore,  $\tau^2=\text{Id}$.
\end{Lemma}
\begin{proof}The  result follows immediately from the defining relations of $\AOB$ above.\end{proof}
\begin{Lemma}\label{dots} As morphisms  in $\AOB$, we have
\begin{multicols}{2}\item[(1)]$
\begin{tikzpicture}[baseline = 2.5mm]
	\draw[->,thick,darkblue] (0.28,0) to[out=90,in=-90] (-0.28,.6);
	\draw[<-,thick,darkblue] (-0.28,0) to[out=90,in=-90] (0.28,.6);
   \draw (-0.21,0.18) \bdot;
\end{tikzpicture}
~=~
\begin{tikzpicture}[baseline = 2.5mm]
	\draw[<-,thick,darkblue] (-0.28,-.0) to[out=90,in=-90] (0.28,0.6);
	\draw[->,thick,darkblue] (0.28,-.0) to[out=90,in=-90] (-0.28,0.6);
\draw(0.24,0.45) \bdot;
\end{tikzpicture}~-~\begin{tikzpicture}[baseline = 2.5mm]\draw[<-,thick,darkblue] (0,0.6) to[out=down,in=left] (0.28,0.35) to[out=right,in=down] (0.56,0.6);
  \draw[<-,thick,darkblue] (0,0) to[out=up,in=left] (0.28,0.25) to[out=right,in=up] (0.56,0);
         \end{tikzpicture}.$
         \item[(2)]$
\begin{tikzpicture}[baseline = 2.5mm]
	\draw[->,thick,darkblue] (0.28,0) to[out=90,in=-90] (-0.28,.6);
	\draw[<-,thick,darkblue] (-0.28,0) to[out=90,in=-90] (0.28,.6);
   \draw (0.24,0.16) \bdot;
\end{tikzpicture}
~=~
\begin{tikzpicture}[baseline = 2.5mm]
	\draw[<-,thick,darkblue] (-0.28,-.0) to[out=90,in=-90] (0.28,0.6);
	\draw[->,thick,darkblue] (0.28,-.0) to[out=90,in=-90] (-0.28,0.6);
\draw (-0.2,0.42) \bdot;
\end{tikzpicture}
~-~\begin{tikzpicture}[baseline = 2.5mm]\draw[<-,thick,darkblue] (0,0.6) to[out=down,in=left] (0.28,0.35) to[out=right,in=down] (0.56,0.6);
  \draw[<-,thick,darkblue] (0,0) to[out=up,in=left] (0.28,0.25) to[out=right,in=up] (0.56,0);
         \end{tikzpicture}$.
\item[(3)]
$
\begin{tikzpicture}[baseline = 2.5mm]
	\draw[<-,thick,darkblue] (0.28,0) to[out=90,in=-90] (-0.28,.6);
	\draw[->,thick,darkblue] (-0.28,0) to[out=90,in=-90] (0.28,.6);
   \draw (-0.21,0.18) \bdot;
\end{tikzpicture}
~=~
\begin{tikzpicture}[baseline = 2.5mm]
	\draw[->,thick,darkblue] (-0.28,-.0) to[out=90,in=-90] (0.28,0.6);
	\draw[<-,thick,darkblue] (0.28,-.0) to[out=90,in=-90] (-0.28,0.6);
\draw(0.23,0.42) \bdot;
\end{tikzpicture}~+~\begin{tikzpicture}[baseline = 2.5mm]\draw[->,thick,darkblue] (0,0.6) to[out=down,in=left] (0.28,0.35) to[out=right,in=down] (0.56,0.6);
  \draw[->,thick,darkblue] (0,0) to[out=up,in=left] (0.28,0.25) to[out=right,in=up] (0.56,0);
         \end{tikzpicture}.$
         \item[(4)]$
\begin{tikzpicture}[baseline = 2.5mm]
	\draw[<-,thick,darkblue] (0.28,0) to[out=90,in=-90] (-0.28,.6);
	\draw[->,thick,darkblue] (-0.28,0) to[out=90,in=-90] (0.28,.6);
   \draw(0.22,0.18)\bdot;
\end{tikzpicture}
~=~
\begin{tikzpicture}[baseline = 2.5mm]
	\draw[->,thick,darkblue] (-0.28,-.0) to[out=90,in=-90] (0.28,0.6);
	\draw[<-,thick,darkblue] (0.28,-.0) to[out=90,in=-90] (-0.28,0.6);
\draw(-0.22,0.45) \bdot;
\end{tikzpicture}
~+~\begin{tikzpicture}[baseline = 2.5mm]\draw[->,thick,darkblue] (0,0.6) to[out=down,in=left] (0.28,0.35) to[out=right,in=down] (0.56,0.6);
  \draw[->,thick,darkblue] (0,0) to[out=up,in=left] (0.28,0.25) to[out=right,in=up] (0.56,0);
         \end{tikzpicture}$.
  \item[(5)]\begin{tikzpicture}[baseline = -0.5mm]
 \draw[-,thick,darkblue] (0,0) to[out=down,in=left] (0.28,-0.28) to[out=right,in=down] (0.56,0);
\draw (0,0) \bdot;
  \draw[<-,thick,darkblue] (0,0.2) to (0,0);
   \draw[-,thick,darkblue] (0.56,0) to (0.56,.2);
 \end{tikzpicture}
 ~=~\:
 \begin{tikzpicture}[baseline = -0.5mm]
 \draw[-,thick,darkblue] (0,0) to[out=down,in=left] (0.28,-0.28) to[out=right,in=down] (0.56,0);
 \draw (0.56,0) \bdot;
  \draw[<-,thick,darkblue] (0,0.2) to (0,0);
   \draw[-,thick,darkblue] (0.56,0) to (0.56,.2);
 \end{tikzpicture}.
 \item[(6)] \begin{tikzpicture}[baseline = -0.5mm]
  \draw[-,thick,darkblue] (0,0) to[out=up,in=left] (0.28,0.28) to[out=right,in=up] (0.56,0);
  \draw[<-,thick,darkblue] (0,-0.2) to (0,-0);
 \draw (0,0)\bdot;
 \draw[-,thick,darkblue] (0.56,0) to (0.56,-.2);
  \end{tikzpicture}
  ~=~\:\begin{tikzpicture}[baseline = -0.5mm]
\draw[-,thick,darkblue] (0,0) to[out=up,in=left] (0.28,0.28) to[out=right,in=up] (0.56,0);
 \draw[-,thick,darkblue] (0.56,0) to (0.56,-.2);
 \draw[<-,thick,darkblue] (0,-0.2) to (0,-0);
  \draw (0.56,0)\bdot;
  \end{tikzpicture}.
    \item[(7)]\begin{tikzpicture}[baseline = -0.5mm]
 \draw[-,thick,darkblue] (0,0) to[out=down,in=left] (0.28,-0.28) to[out=right,in=down] (0.56,0);
\draw (0,0) \bdot;
  \draw[-,thick,darkblue] (0,0) to (0,.2);
   \draw[->,thick,darkblue] (0.56,0) to (0.56,.2);
 \end{tikzpicture}
 ~=~\:
 \begin{tikzpicture}[baseline = -0.5mm]
 \draw[-,thick,darkblue] (0,0) to[out=down,in=left] (0.28,-0.28) to[out=right,in=down] (0.56,0);
 \draw (0.56,0) \bdot;
  \draw[-,thick,darkblue] (0,0) to (0,.2);
   \draw[->,thick,darkblue] (0.56,0) to (0.56,.2);
 \end{tikzpicture}.
 \item[(8)] \begin{tikzpicture}[baseline = -0.5mm]
  \draw[-,thick,darkblue] (0,0) to[out=up,in=left] (0.28,0.28) to[out=right,in=up] (0.56,0);
  \draw[-,thick,darkblue] (0,0) to (0,-.2);
 \draw (0,0)\bdot;
 \draw[->,thick,darkblue] (0.56,0) to (0.56,-.2);
  \end{tikzpicture}
  ~=~\:\begin{tikzpicture}[baseline = -0.5mm]
\draw[-,thick,darkblue] (0,0) to[out=up,in=left] (0.28,0.28) to[out=right,in=up] (0.56,0);
 \draw[->,thick,darkblue] (0.56,0) to (0.56,-.2);
 \draw[-,thick,darkblue] (0,0) to (0,-.2);
  \draw (0.56,0)\bdot;
  \end{tikzpicture}.
\end{multicols}
\end{Lemma}
\begin{proof}Easy exercise. \end{proof}

\subsection{The category  $\OB(\mathbf u, \mathbf u')$ }\label{OB}
Recall $\mathbf u,  \mathbf u'$ in \eqref{mathbfu} and define
  \begin{equation} \label{ff} f(u)=\prod_{i=1}^\ell (u-u_i), \quad
f'(u)=\prod_{i=1}^\ell (u-u'_i). \end{equation}
Thanks to    \cite[(1.13)]{BCNR}, there are scalars $\delta_i \in\Bbbk$, $i\in \mathbb N\setminus \{0\}$, such that
 \begin{equation}\label{jidjedd}
 1+\sum_{i\geq 1}\delta_iu^{-i}=f'(u)/f(u)\in \Bbbk[[u^{-1}]] .
 \end{equation}
 In \cite{BCNR}, Brundan et.~al consider  the right tensor ideal $K$ of $\AOB$ generated by $f(\begin{tikzpicture}[baseline=-.5mm]
 \draw[->,thick,darkblue] (0,-.3) to (0,.3);
      \draw (0,0) \bdot;
\end{tikzpicture})$ together with  $\begin{tikzpicture}[baseline = -0.5mm]
\draw[-,thick,darkblue] (0,0) to[out=down,in=left] (0.25,-0.25) to[out=right,in=down] (0.5,0);
  \draw[-,thick,darkblue] (0,0) to[out=up,in=left] (0.25,0.25) to[out=right,in=up] (0.5,0);
  \draw[->,thick,darkblue]  (0.25,0.25) to[out=right,in=up] (0.5,0);
  \draw (0,0) \bdot;
  \node at (-0.2,0) {$\color{darkblue}\scriptstyle i$};
  \end{tikzpicture}-\delta_{i+1}$ for all  $i\in \mathbb N$, where $\begin{tikzpicture}[baseline=-.5mm]
 \draw[->,thick,darkblue] (0,-.3) to (0,.3);
     \draw (0,0) \bdot;
      \node at (-0.2,0) {$\color{darkblue}\scriptstyle i$};
\end{tikzpicture}$ is the $i$th power of   $\begin{tikzpicture}[baseline=-.5mm]
 \draw[->,thick,darkblue] (0,-.3) to (0,.3);
      \draw (0,0) \bdot;
\end{tikzpicture}$. \textsf{The cyclotomic oriented Brauer category}~\cite{BCNR} is defined to be the quotient category
\begin{equation}\label{cobc} \OB(\mathbf u, \mathbf u' )=\AOB/K,\end{equation}.
\begin{rem}\label{anothre}
By \cite[(1.14)]{BCNR}, there are
$\delta_j'$ $\in \Bbbk$, $j\in \mathbb N\setminus \{0\}$,   such that
$$(1+\sum_{i\geq 1}\delta_iu^{-i}) (1-\sum_{j\geq 1}\delta_j'u^{-j})   =1. $$
It is proved in  \cite[Remark~1.6]{BCNR} that the previous $K$ is also generated by  $f'(\begin{tikzpicture}[baseline=-.5mm]
 \draw[<-,thick,darkblue] (0,-.3) to (0,.3);
     \draw (0,0) \bdot;
\end{tikzpicture})$ together with  $\begin{tikzpicture}[baseline = -0.5mm]
\draw[-,thick,darkblue] (0,0) to[out=down,in=left] (0.25,-0.25) to[out=right,in=down] (0.5,0);
  \draw[<-,thick,darkblue] (0,0) to[out=up,in=left] (0.25,0.25) to[out=right,in=up] (0.5,0);
  \draw[-,thick,darkblue]  (0.25,0.25) to[out=right,in=up] (0.5,0);
 \draw (0.5,0) \bdot;
  \node at (0.7,0) {$\color{darkblue}\scriptstyle i$};
  \end{tikzpicture}-\delta_{i+1}'$ for all  $i\in \mathbb N$.
  \end{rem}
 By \eqref{jidjedd}, $\delta_1=u_1-u_1'$ and $\OB(\mathbf u, \mathbf u' )$ is $\OB(\delta_1)$ if $\ell=1$.
  We are going to discuss  $\OB(\mathbf u, \mathbf u' ) $  no matter whether $\ell=1$ or not. In any case, the set of  objects in  $\OB(\mathbf u, \mathbf u' ) $
 is
 \begin{equation}\label{oobj} J=\langle \uparrow, \down\rangle, \end{equation}  the set of   finite sequence of the symbols $\uparrow, \down$, including the empty word $\emptyset$. Let
\begin{equation}\label {Ba1} A=\bigoplus_{ a,  b\in J}\text{Hom}_{ \OB(\mathbf u, \mathbf u' ) }( a, b). \end{equation}
Since the contravariant functor $\tau$ in Lemma~\ref{sigma}
stabilizes the right tensor ideal $K$, it induces an anti-involution $\tau_A$ on $A$.

 For any subspace $B\subseteq A$ and any $a, b\in J$,  let $B_{a, b}=1_a B 1_b$, where $1_a$ is the identity morphism from $a$ to $a$.
 When $b=a$,  $B_{a, b}$  is also denoted  by $B_a$. Then $A_{a, b} =\Hom_{ \OB(\mathbf u, \mathbf u' ) }( b, a)$ and \begin{equation}\label {Bal2} A= \bigoplus_{ a,  b \in J} A_{a, b}.\end{equation}
 So,  $A$ is  a \textsf{locally unital   $\Bbbk$-algebra} and  the set  $\{1_{ a} \mid  a \in J\}$  serves as the system of  mutually orthogonal idempotents of $A$. In this paper $C$-mod is the category of left $C$-modules  $M$ such that  $M=\bigoplus_{a\in H} 1_{a} M$   for any locally unital  algebra $C=\bigoplus_{a,b\in H}1_a C1_b $. Let
$C$-lfdmod be the subcategory of $C$-mod consisting  of modules  $M$ such that any $1_{ a} M$
is  finite dimensional.
Let  $C$-fdmod (resp., $C$-pmod) be the category of finite dimensional (resp., finitely generated projective) left $C$-modules.

   We are going to  recall the notion of  \textsf{normally ordered oriented Brauer diagrams} in \cite{BCNR}. For any $a, b \in \mathbb N$,   an   $(a, b)$-\textsf{Brauer diagram}  is a diagram on which $a+b$ points are placed on  two parallel horizontal lines, and $a$  points on the lower line and $b$ points on the upper line, and each point joins precisely to one other point.
  If two points at the upper (resp., lower) line join each other, then this strand  is called a cup (resp.,
cap). Otherwise, it is called a vertical strand.

Any $(a, b)$-Brauer diagram can also  be considered as a partitioning of the set $\{1,2,\ldots, a+b\}$ into disjoint union of pairs. If $a+b$ is odd, then  there is no $(a, b)$-Brauer diagram. Two $(a, b)$-Brauer diagrams are said to be equivalent if they give the same partitioning of
$\{1,2,\ldots, a+b\}$ into disjoint union of pairs.

 An \textsf{oriented Brauer diagram} is   obtained by adding consistent orientation to each strand in  a    Brauer diagram as above.
 Given any oriented Brauer diagram $d$, let $a$ (resp., $b$) be the element in $J$ which is indicated
 from the orientation of the endpoints of the lower line (resp., upper line) and $d$ is called of type $a\rightarrow b$.
 For example, the following diagram is of type  $\uparrow\uparrow\down \down \rightarrow  \down\uparrow$:

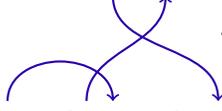
 \begin{center}\begin{tikzpicture}[baseline = 25pt, scale=0.35, color=\clr]
        \draw[->,thick] (5,0) to[out=up,in=down] (8,4);
        \draw[<-,thick] (10,0) to[out=up,in=down] (6,4);
        \draw[->,thick] (2,0) to[out=up,in=left] (4,1.5) to[out=right,in=up] (6,0);
           \end{tikzpicture}.
\end{center}
Two oriented Brauer diagrams of type $a\rightarrow b$ are said to be  equivalent if their underlying Brauer diagrams are equivalent.

A \textsf{dotted  oriented Brauer diagram} of type $a\rightarrow b$ is an oriented Brauer diagram of type $a\rightarrow b$ such that each segment is decorated  with some non-negative number of  $\bullet$'s (called dots), where a segment means a connected component of the diagram obtained when all crossings are deleted.
A \textsf{normally ordered dotted oriented Brauer diagram}  of type $a\rightarrow b$ is  a dotted oriented Brauer diagram  of type $a\rightarrow b$ such that:
 \begin{itemize}\item  whenever a dot appears on a strand, it is on the outward-pointing boundary,
\item there are at most $ \ell-1$ dots on each strand.\end{itemize}
     Suppose $\ell=4$. In the following pair of diagrams,  the  right one is a normally ordered dotted oriented Brauer diagram of  type  $\uparrow\uparrow\down \down \rightarrow \down\uparrow$  and the left one is not:
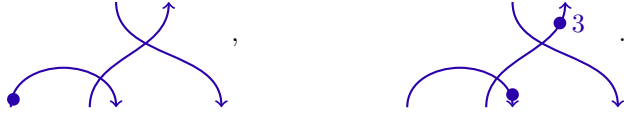
\begin{center}
     \begin{tikzpicture}[baseline = 25pt, scale=0.35, color=\clr]
        \draw[->,thick] (5,0) to[out=up,in=down] (8,4);
         \draw[<-,thick] (10,0) to[out=up,in=down] (6,4);
         \draw  (2.1,0.3) \bdot;
         \draw[->,thick] (2,0) to[out=up,in=left] (4,1.5) to[out=right,in=up] (6,0);
           \end{tikzpicture}~,
    \qquad\qquad\qquad
     \begin{tikzpicture}[baseline = 25pt, scale=0.35, color=\clr]
        \draw[->,thick] (5,0) to[out=up,in=down] (8,4);
         \draw[<-,thick] (10,0) to[out=up,in=down] (6,4);
         \draw[->,thick] (2,0) to[out=up,in=left] (4,1.5) to[out=right,in=up] (6,0);
          \draw  (6,0.48) \bdot;
           \draw  (7.8,3.2) \bdot;
            \draw (8.5,3.2) node{$3$};
           \end{tikzpicture}.
\end{center}
 Two  normally ordered dotted oriented Brauer diagrams are said to be equivalent if the underlying oriented Brauer diagrams are equivalent and there are the same number of dots on their corresponding strands.

\begin{Theorem}\label{bcnr}  \cite[Theorem~1.5]{BCNR} Suppose  $a, b\in J$.\begin{itemize}\item[(1)]  Two equivalent normally ordered  dotted oriented  diagrams represent the same morphism in   $\OB(\mathbf u, \mathbf u' ) $.\item [(2)]
   $\Hom_{ \OB(\mathbf u, \mathbf u' )} ( a,   b)$ has  basis given by  the set of all equivalence classes of   normally ordered dotted  oriented Brauer diagrams  of type $ a\rightarrow b$.\end{itemize}  \end{Theorem}
 Thanks to Theorem~\ref{bcnr},  $A$ is  locally finite dimensional in the sense that
 $$\dim 1_a A 1_b <\infty$$ for all $a, b\in J$. We are going to explain that $A$ admits an upper finite  weakly triangular decomposition~\cite[Proposition~2.2]{GRS3}.

Suppose    $ a= a_1a_2\cdots  a_h$, where $a_i\in \{\uparrow, \downarrow\}$, $1\le i\le h$.
Define
 $$\ell_\down( a)=|\{i:  a_i=\down\}|, \ \ \   \ell_\uparrow( a)=|\{i:  a_i=\uparrow\}| ,$$  where $|D|$ is the cardinality of  a set $D$.
 When $h=0$, i.e.,  $a$ is the empty word,  $\ell_\down( a)= \ell_\uparrow( a)=0$.
For any $a, b\in J$,   write  $ a\sim  b$ if $(\ell_\downarrow( a),\ell_\uparrow(a))=( \ell_\downarrow( b),\ell_\uparrow(b))$.
Then $\sim$ is an equivalence relation on $J$. Let  $I= J/\sim$. As sets, \begin{equation}\label{I} I \cong \mathbb N^2 .\end{equation}

\begin{Defn} \label{ocbap}  For any  $(r,s), (r_1,s_1)\in I$,  define  $(r,s)\preceq (r_1,s_1)$ if $r=r_1+k$ and $s=s_1+k$ for some $k\in\mathbb N$.\end{Defn}
 Then $\preceq$  is a   partial order on $I$. Later on, we also use $\ob a, \ob b$ etc. to denote elements in $I$. The partial order $\preceq$ on $I$ is upper finite in the sense that  $\{\ob b\in I\mid \ob a\preceq \ob b\}$ is finite for all $\ob a\in I$. It  induces a partial order on $J$
such that $a\prec b$ if $ a\in\ob a$, $b\in \ob b$  for $\ob a, \ob b\in I$ and   $\ob a\prec \ob b$.
If $\ob a\in I$, define $$1_{\ob a}=\sum_{b\in \ob a} 1_b $$ and    $B_{\ob a, \ob b}=\bigoplus_{c\in \ob a, d\in \ob b} 1_{c} B 1_{d}$ for any   $\ob a, \ob b\in I$ and any subspace $B$ of $A$.  Then $B_{\ob a, \ob b}=1_{\ob a} B 1_{\ob b}$.

\begin{Defn} For any $a, b, c\in J$ and $b\sim c$, define
\begin{itemize}\item[(a)]$Y(a,b):$ the set of all normally ordered dotted oriented Brauer diagrams of type $b\rightarrow a$ on which there are  neither  caps  nor crossings among vertical strands,  and there are no dots on vertical strands,
\item [(b)] $H(b,c):$ the set of all normally ordered dotted oriented Brauer diagrams of type $c\rightarrow b$ on which there are neither  cups nor caps,
 \item [(c)] $X(b, a)$ : the set of all normally ordered dotted oriented Brauer diagrams of type $a\rightarrow b$ on which there are neither cups  nor crossings among vertical strands,  and  there are no dots on vertical strands.
\end{itemize}\end{Defn}

\begin{Lemma}\label{nonzero} Suppose $a, b\in J$. Then  $X(b, a)\neq \emptyset$ if and only if $Y(a, b)\neq \emptyset $, and  $X(b, a)=Y(a,b)=\emptyset$  unless  $a\preceq b$. Furthermore,
 \begin{itemize}\item[(1)]if  $a=\uparrow^s \downarrow^r$ for some $r,s$,  then   $X(b, a)\neq \emptyset$  if and only if $b=\uparrow^m \downarrow^n$ for some $m,n$ and $a\preceq b$;
    \item[(2)] if  $a=\downarrow^r \uparrow^s$ for some $r,s$, then   $X(b, a)\neq \emptyset$ if and only if   $b=\downarrow^m \uparrow^n$ for some $m,n$ and $a\preceq b$.\end{itemize}
\end{Lemma}
\begin{proof}Easy exercise. \end{proof}
In \cite{GRS3}, we define three subspaces $A^+$, $A^-$ and $A^\circ$ (not subalgebras!) of $A$ such that
 $$A^\pm =\bigoplus_{   b, c \in J}  A_{b,c}^\pm, \ \  A^\circ =\bigoplus_{\ob a\in I}\bigoplus_{b, c\in \ob a}   A_{b, c}^\circ ,$$
  where $A_{b,c}^-$ (resp., $A_{b,c}^+$, resp., $A_{b,c}^\circ$) is the $\Bbbk$-space with basis $Y(b, c)$ (resp., $X(b, c)$, resp., $H(b, c)$).
\begin{Prop}\label{WT} \cite[Proposition~2.2]{GRS3} The data $(I, A^-, A^\circ, A^+)$ satisfies the following conditions:
\begin{itemize} \item[(1)]  $(I, \preceq)$ is upper finite, where $I$ is given in Definition~\ref{ocbap}.
\item [(2)] $ A^-_{\ob a,  \ob b}=0$ and   $A^+_{\ob b,    \ob a}=0$ unless $    \ob a\preceq   \ob b$.  Furthermore,  $ A^-_{\ob a }=A^+_{\ob a}=\bigoplus_{c\in  \ob a}\Bbbk 1_{c}$.
    \item[(3)]  $A^-\otimes_{\mathbb K }  A^\circ \otimes_{\mathbb K}  A^+\cong A$
as  $\Bbbk$-spaces where $\mathbb K=\oplus_{  b\in J}  \Bbbk 1_{ b}$.   The required isomorphism  is given by the multiplication on $A$.\end{itemize}\end{Prop}
In other words, $(I, A^-, A^\circ, A^+)$ is an upper finite weakly triangular decomposition in the sense of \cite[Definition~2.1]{GRS3}.
When $\ell=1$,  $(I, A^-, A^\circ, A^+)$ is the same as that for the oriented Brauer category $\OB(\delta_1)$ in \cite{Re}. In fact, it gives  a triangular decomposition in the sense of \cite{BS}.

\subsection {Quotient algebras }\label{cba}
For any $\ob a\in I$, define  $$A_{\preceq \ob a}= A/A^{\not\preceq \ob a}, \quad   \bar A_{\ob a}=\bar 1_{ \ob a } A_{\preceq \ob a}\bar 1_{ \ob a}, $$
where $A^{\diamond \ob a}$  is   the two-sided ideal of $A$ generated by $I^{\diamond\ob a}=\{1_\ob b\mid \ob b\diamond \ob a\}$,  $\diamond\in \{\succ, \not\preceq, \nprec\}$. For any $x\in A$, let $\bar x$ be its image in any quotient algebra of $A$.
\begin{Lemma}\label{cellbasis}\cite[Lemma~2.6, Proposition~2.10]{GRS3} Suppose $e, c\in J$ and $\ob a\in I$.
\begin{itemize}\item[(1)]$  A_{e, c}$ has basis
 $   \{ yhx\mid (y, h, x)\in \bigcup_{  d\succeq  c,   b\succeq e, b\sim d} Y(  e,   b)\times H(  b,d)\times X(d,  c)\}$.
 \item[(2)]  $ \bar 1_eA_{\preceq  \ob a}\bar 1_c$ has basis
 $\{\bar y\bar h\bar x\mid (y,h,x)\in \bigcup_{b,d\in \ob b, \ob b\preceq \ob a} Y(  e, b)\times H(  b,d)\times X(d,c)\}$.

\end{itemize}
\end{Lemma}

 Suppose $a, b\in J$ and $a\sim b$.
Following \cite{Re}, let $_a\sigma _b$ be the unique element in $H(a,b)$  on which there are neither crossings among  strands of the same orientation (up or down) nor dots on each strand. For example,
$$_{\downarrow\downarrow\uparrow}\sigma _{\uparrow\downarrow\downarrow}=\begin{tikzpicture}[baseline = -1mm]
	\draw[<-,thick,darkblue] (0.45,-.6) to (-0.45,.6);
        \draw[<-,thick,darkblue] (0.9,-.6) to (0,0.6);
	\draw[<-,thick,darkblue] (0.45,.6) to (-0.45,-.6);
\end{tikzpicture}.$$
The following result follows immediately from the definition of $_a\sigma _b$.
\begin{Lemma}\label{unie} Suppose $a, b, c\in J$ such that $a\sim b\sim c$. Then
$_a\sigma_b \circ {_b\sigma_c}= {_a\sigma_c}$ and $_a\sigma_a=1_a$.
\end{Lemma}

 For any $a, b\in\ob a$, let $\bar A_{a, b}=\bar 1_a\bar A_\ob a\bar 1_b$. Then   $\bar A_{a, a}$, which will be denoted by $\bar A_a$, is a subalgebra of $\bar A_\ob a$.
 For any $a,b, c,d\in\ob a$, thanks to Lemma~\ref{unie}, there is a   $\Bbbk$-linear isomorphism
 \begin{equation}\label{isomor}\bar A_{a, b}\overset{\sim}\longrightarrow \bar A_{c, d}, \quad \bar g\mapsto (\bar {_c\sigma_a})\bar g(\bar {_b\sigma_d}), \quad \forall g\in H(a, b).\end{equation}
 When $a=b$ and $c=d$, the   isomorphism in \eqref{isomor} is  an algebra isomorphism.
  \begin{Lemma}\label{isomhecke} Let $\bar A^\circ:=\bigoplus_{\ob a\in I}\bar A_\ob a$.
\begin{itemize} \item[(1)] If   $\ob a=(r,s)\in I$, then there  is an algebra isomorphism
  $\phi: \text{Mat}_{({_{\text{ }r}^{r+s}})}(\bar A_{\downarrow^r\uparrow^s})\cong\bar A_{\ob a}$, where the rows and columns of matrices are indexed by $b\in \ob a$.
  \item [(2)]  $\bar A^\circ\text{-fdmod}  \overset {Morita}\sim  \bigoplus_{r, s\in \mathbb N}\bar A_{\downarrow^r \uparrow^s}\text{-fdmod}$.\end{itemize}
  \end{Lemma}
  \begin{proof} Suppose  $\ob a=(r,s)$. Thanks to \eqref{isomor}, the required algebra isomorphism $\phi$ in (1) satisfies
$$\phi (\sum_{a,b\in\ob a}\tau_{a, b}e_{a,b}) = \sum_{a,b\in\ob a}( \bar{_a\sigma_{\downarrow^r\uparrow^s}}) \tau_{a, b}( \bar{_{\downarrow^r\uparrow^s}\sigma_b}), \forall \tau_{a, b}\in \bar A_{\downarrow^r\uparrow^s}, $$  where $e_{a,b}$'s  are the corresponding matrix
units. Now, (2) immediately follows from (1). The required functor $\alpha:\bar A^\circ\text{-fdmod} \rightarrow  \bigoplus_{r, s\in \mathbb N}\bar A_{\downarrow^r \uparrow^s}\text{-fdmod}$ and  its inverse $\beta$ are
\begin{equation}\label{isomhecke2} \alpha = \bigoplus_{r, s\in \mathbb N} \bar 1_{\downarrow^r \uparrow^s}(?), \quad \beta =\bigoplus_{r, s\in \mathbb N} ( \bar A^\circ \bar1_{\downarrow^r \uparrow^s} \otimes _{\bar A_{\downarrow^r\uparrow^s}} ?) \end{equation}

\end{proof}

For any endofunctor $\mathcal F$ of $\bar A^\circ\text{-fdmod}$ and any endofunctor $\mathcal G$ of $\bigoplus_{r, s\in \mathbb N}\bar A_{\downarrow^r \uparrow^s}\text{-fdmod}$, set \begin{equation}\label{morif} \mathcal F\sim \mathcal G\end{equation} if there is a natural isomorphism between $\mathcal F$ and $\beta \mathcal G\alpha$ where $\alpha$ and $\beta$ are given \eqref{isomhecke2}. Obviously, $\mathcal F$ is exact if and only if $\mathcal G$ is exact.

\subsection{Degenerate cyclotomic Hecke algebras}\label{hecke}
Given  $\mathbf e=(e_1,\ldots, e_\ell)$, the degenerate  cyclotomic Hecke algebra $ H_{\ell, n}(\mathbf e)$ is the associative  $\Bbbk$-algebra generated by $L_1, S_1, \ldots, S_{n-1}$ subject to the relations:
$$\begin{cases} S_i^2=1, & \text{$1\le i\le n-1$,}\\
 S_iS_j=S_jS_i, & \text{$1\le i<j-1\le n-2$,}\\
 S_iS_{i+1}S_i=S_{i+1}S_iS_{i+1} , & \text{$1\le i\le n-2$,}\\
  L_1S_i=S_iL_1, &\text{$2\le i\le n-1$,}\\
 (S_1L_1S_1+S_1)L_1=L_1(S_1L_1S_1+S_1), &\\
 (L_1-e_1)(L_1-e_2)\cdots (L_1-e_\ell)=0. &\\
\end{cases}
$$

Let   $L_i=S_{i-1}L_{i-1}S_{i-1}+S_{i-1}$ for  $ 2\le i\le n$. Then $\{L_i\mid 1\le i\le n \}$  generates a commutative subalgebra of $ H_{\ell, n}(\mathbf e)$.
Thanks to  \cite[Lemma~6.6]{AMR}, all generalized eigenvalues of $L_i$,  $1\le i\le n$,   are  of forms $e_j+k$,    $1\le j\le \ell $ and $1-n\le k\le n-1$.

\begin{Lemma}\label{Aa} For any  $r, s\in\mathbb N$, $\bar A_{\uparrow^r\downarrow^s}\cong  H_{\ell, r}(\mathbf u)\bigotimes H_{\ell, s}(-\mathbf u')$ and
 $\bar A_{\downarrow^s\uparrow^r}\cong  H_{\ell, s}(-\mathbf u') \bigotimes H_{\ell, r}(\mathbf u)$.
\end{Lemma}
\begin{proof}
Define the $\Bbbk$-algebra homomorphism  $\phi: H_{\ell, r}(\mathbf u)\bigotimes H_{\ell, s}(-\mathbf u')\rightarrow  \bar A_{\uparrow^r\downarrow^s}$ such that
$$ 1\otimes g\mapsto 1\otimes g^{\downarrow}, h\otimes 1\mapsto h^\uparrow \otimes 1,  $$ for any generators $1\otimes g$ and $h\otimes 1$ of $H_{\ell, r}(\mathbf u)\bigotimes H_{\ell, s}(-\mathbf u')$,
where$$\begin{aligned}&1\otimes L^\downarrow_1:=-\bar{\begin{tikzpicture}[baseline = -1mm,  color=\clr]
	\draw[->,thick,darkblue] (0.45,-.6) to (0.45,.6);
        \draw[->,thick,darkblue] (0,-.6) to (0,0.6);
           \draw (0.8,0) \bdot;
           \draw(0.25,-0.5) node{$ \cdots$};
           \draw(0.25,0.45) node{$ \cdots$};
           \draw (0.45,-0.75) node{$r$};
            \draw(1.1,-0.5) node{$ \cdots$};
           \draw(1.1,0.45) node{$ \cdots$};
            \draw[<-,thick,darkblue] (0.8,-.6) to (0.8,0.6);
        \draw[<-,thick,darkblue] (1.5,-.6) to (1.5,0.6);
         \draw (1.5,0.75) node{$r+s$};
\end{tikzpicture}}, L^\uparrow_1\otimes 1:=\bar{\begin{tikzpicture}[baseline = -1mm,  color=\clr]
	\draw[->,thick,darkblue] (0.45,-.6) to (0.45,.6);
        \draw[->,thick,darkblue] (0,-.6) to (0,0.6);
           \draw (0,0) \bdot;
           \draw(0.25,-0.5) node{$ \cdots$};
           \draw(0.25,0.45) node{$ \cdots$};
           \draw (0.45,-0.75) node{$r$};
            \draw(1.1,-0.5) node{$ \cdots$};
           \draw(1.1,0.45) node{$ \cdots$};
            \draw[<-,thick,darkblue] (0.8,-.6) to (0.8,0.6);
        \draw[<-,thick,darkblue] (1.5,-.6) to (1.5,0.6);
         \draw (1.5,0.75) node{$r+s$};\end{tikzpicture}},
         \\& S_i^\uparrow\otimes 1:=\bar{\begin{tikzpicture}[baseline = 25pt, scale=0.35, color=\clr]
       \draw[->,thick](0,1.1)to(0,3.9);
       \draw(0.75,1.4) node{$ \cdots$}; \draw(0.75,3.7) node{$ \cdots$};
       \draw[->,thick](1.5,1.1)to(1.5,3.9);
       \draw[->,thick] (2,1) to[out=up, in=down] (3,3.9);
       \draw(2,4.5)node{\tiny$i$}; \draw(3.2,4.5)node{\tiny$i+1$};
        \draw[->,thick] (3,1) to[out=up, in=down] (2,3.9);
         \draw[->,thick](4,1.1)to(4,3.9);
         \draw[->,thick](5.5,1.1)to(5.5,3.9);
          \draw(4.8,1.4) node{$ \cdots$}; \draw(4.8,3.7) node{$ \cdots$};
 \draw(5.5,4.5)node{\tiny$r$};
         \draw[<-,thick](4.5+1.5,1.1)to(4.5+1.5,3.9);
         \draw(4.75+0.5+1.5,1.4) node{$ \cdots$}; \draw(4.75+0.5+1.5,3.7) node{$ \cdots$};
         \draw[<-,thick](5.5+0.5+1.5,1.1)to(5.5+0.5+1.5,3.9);
         \draw(5.5+0.5+1.5,4.5)node{\tiny$r+s$};
           \end{tikzpicture}}\text{,}\ \ 1\otimes S_j^\downarrow:=\bar{\begin{tikzpicture}[baseline = 25pt, scale=0.35, color=\clr]
       \draw[<-,thick](0,1.1)to(0,3.9);
       \draw(0.75,1.4) node{$ \cdots$}; \draw(0.75,3.7) node{$ \cdots$};
       \draw[<-,thick](1.5,1.1)to(1.5,3.9);
       \draw[<-,thick] (2,1) to[out=up, in=down] (3,3.9);
       \draw(2,4.5)node{\tiny$r+j$}; \draw(3.2,-0.1)node{\tiny$r+j+1$};
        \draw[<-,thick] (3,1) to[out=up, in=down] (2,3.9);
         \draw[<-,thick](4,1.1)to(4,3.9);
         \draw[->,thick](-2.5,1.1)to(-2.5,3.9);
         \draw(-1.5,1.4) node{$ \cdots$}; \draw(-1.5,3.7) node{$ \cdots$};
         \draw[->,thick](-0.5,1.1)to(-0.5,3.9);
          \draw(4.8,1.4) node{$ \cdots$}; \draw(4.8,3.7) node{$ \cdots$};
          \draw[<-,thick](6,1.1)to(6,3.9);
         \draw(-0.5,4.5)node{\tiny$r$};
           \end{tikzpicture}}.
         \end{aligned}$$
In order to verify that $\phi$ is well-defined, it suffices to verify that the images of generators of $H_{\ell, r}(\mathbf u)\bigotimes H_{\ell, s}(-\mathbf u')$ above satisfy the defining relations for $H_{\ell, r}(\mathbf u)\bigotimes H_{\ell, s}(-\mathbf u')$. This can be verified directly by using \eqref{relation 3}-\eqref{relation 4}, \eqref{relation 6} and \eqref{relation 9}-\eqref{relation 11} together with the fact that  dots can be slide freely in any dotted oriented Brauer diagram if we consider it as element in   $\bar A_{\uparrow^r\downarrow^s}$ (see
 Lemma~\ref{dots}). We  give an example  and leave other details to the reader.
 By  Lemma~\ref{dots}(1)-(2),
$$1\otimes L^\downarrow_1=-\bar{\begin{tikzpicture}[baseline = -1mm,  color=\clr]
	\draw[->,thick,darkblue] (0.45,-.6) to (0.45,.6);
        \draw[->,thick,darkblue] (0,-.6) to (0,0.6);
        \draw[-,thick, darkblue] (-0.45,0) to[out=up,in=down] (0.6,0.6);
          \draw[<-,thick, darkblue](0.6,-0.6) to[out=up,in=down] (-0.45,0);
           \draw (-0.45,0) \bdot;
           \draw(0.25,-0.5) node{$ \cdots$};
           \draw(0.25,0.45) node{$ \cdots$};
           \draw (0.45,-0.75) node{$r$};
            \draw(1.1,-0.5) node{$ \cdots$};
           \draw(1.1,0.45) node{$ \cdots$};
            \draw[<-,thick,darkblue] (0.8,-.6) to (0.8,0.6);
        \draw[<-,thick,darkblue] (1.5,-.6) to (1.5,0.6);
         \draw (1.5,0.75) node{$r+s$};
\end{tikzpicture}}.$$ So, by Remark~\ref{anothre}, $\tilde{f'}(1\otimes L^\downarrow_1)=0$, where $\tilde {f'}(u)=f'(-u)$, and   $f'(u)$  is given in \eqref{ff}.
Thanks to Lemma~\ref{cellbasis},  $\bar H(  \uparrow^r\downarrow^s , \uparrow^r\downarrow^s)=\{\overline{g\otimes h}| (g, h)\in H(\uparrow^r, \uparrow^r)\times H(\downarrow^s, \downarrow^s)\}$ is a basis of $\bar A_{\uparrow^r\downarrow^s}$. Since $\phi$ sends the well-known basis of $H_{\ell, r}(\mathbf u)\bigotimes H_{\ell, s}(-\mathbf u')$ to $\bar H(\uparrow^r\downarrow^s, \uparrow^r\downarrow^s)$, it  is an isomorphism. The last isomorphism can be proved similarly. In fact, the required  $\Bbbk$-algebra isomorphism is   $\phi': H_{\ell, s}(-\mathbf u') \bigotimes H_{\ell, r}(\mathbf u)\rightarrow  \bar A_{\downarrow^s\uparrow^r}$ such that
$$ 1\otimes g\mapsto 1\otimes g^{\uparrow}, h\otimes 1\mapsto h^\downarrow \otimes 1,  $$ for any generators $1\otimes g$ and $h\otimes 1$ of $H_{\ell, s}(-\mathbf u') \bigotimes H_{\ell, r}(\mathbf u)$,
where \begin{equation}\label{inde}\begin{aligned}&1\otimes L^\uparrow_1:=\bar{\begin{tikzpicture}[baseline = -1mm,  color=\clr]
	\draw[<-,thick,darkblue] (0.45,-.6) to (0.45,.6);
        \draw[<-,thick,darkblue] (0,-.6) to (0,0.6);
           \draw (0.8,0) \bdot;
           \draw(0.25,-0.5) node{$ \cdots$};
           \draw(0.25,0.45) node{$ \cdots$};
           \draw (0.45,-0.75) node{$s$};
            \draw(1.1,-0.5) node{$ \cdots$};
           \draw(1.1,0.45) node{$ \cdots$};
            \draw[->,thick,darkblue] (0.8,-.6) to (0.8,0.6);
        \draw[->,thick,darkblue] (1.5,-.6) to (1.5,0.6);
         \draw (1.5,0.75) node{$r+s$};
\end{tikzpicture}}, L^\downarrow_1\otimes 1:=-\bar{\begin{tikzpicture}[baseline = -1mm,  color=\clr]
	\draw[<-,thick,darkblue] (0.45,-.6) to (0.45,.6);
        \draw[<-,thick,darkblue] (0,-.6) to (0,0.6);
           \draw (0,0) \bdot;
           \draw(0.25,-0.5) node{$ \cdots$};
           \draw(0.25,0.45) node{$ \cdots$};
           \draw (0.45,-0.75) node{$s$};
            \draw(1.1,-0.5) node{$ \cdots$};
           \draw(1.1,0.45) node{$ \cdots$};
            \draw[->,thick,darkblue] (0.8,-.6) to (0.8,0.6);
        \draw[->,thick,darkblue] (1.5,-.6) to (1.5,0.6);
         \draw (1.5,0.75) node{$r+s$};\end{tikzpicture}},
         \\& S_i^\downarrow\otimes 1:=\bar{\begin{tikzpicture}[baseline = 25pt, scale=0.35, color=\clr]
       \draw[<-,thick](0,1.1)to(0,3.9);
       \draw(0.75,1.4) node{$ \cdots$}; \draw(0.75,3.7) node{$ \cdots$};
       \draw[<-,thick](1.5,1.1)to(1.5,3.9);
       \draw[<-,thick] (2,1) to[out=up, in=down] (3,3.9);
       \draw(2,4.5)node{\tiny$i$}; \draw(3.2,4.5)node{\tiny$i+1$};
        \draw[<-,thick] (3,1) to[out=up, in=down] (2,3.9);
         \draw[<-,thick](4,1.1)to(4,3.9);
         \draw[<-,thick](5.5,1.1)to(5.5,3.9);
         \draw(5.5,4.5)node{\tiny$s$};
          \draw(4.8,1.4) node{$ \cdots$}; \draw(4.8,3.7) node{$ \cdots$};
         \draw[->,thick](4.5+1.5,1.1)to(4.5+1.5,3.9);
         \draw(4.75+0.5+1.5,1.4) node{$ \cdots$}; \draw(4.75+0.5+1.5,3.7) node{$ \cdots$};
         \draw[->,thick](5.5+0.5+1.5,1.1)to(5.5+0.5+1.5,3.9);
         \draw(5.5+0.5+1.5,4.5)node{\tiny$r+s$};
           \end{tikzpicture}}\text{ and }1\otimes S_j^\uparrow:=\bar{\begin{tikzpicture}[baseline = 25pt, scale=0.35, color=\clr]
       \draw[->,thick](0,1.1)to(0,3.9);
       \draw(0.75,1.4) node{$ \cdots$}; \draw(0.75,3.7) node{$ \cdots$};
       \draw[->,thick](1.5,1.1)to(1.5,3.9);
       \draw[->,thick] (2,1) to[out=up, in=down] (3,3.9);
       \draw(2,4.5)node{\tiny$s+j$}; \draw(3.2,-0.1)node{\tiny$s+j+1$};
        \draw[->,thick] (3,1) to[out=up, in=down] (2,3.9);
         \draw[->,thick](4,1.1)to(4,3.9);
         \draw[<-,thick](-2.5,1.1)to(-2.5,3.9);
         \draw(-1.5,1.4) node{$ \cdots$}; \draw(-1.5,3.7) node{$ \cdots$};
         \draw[<-,thick](-0.5,1.1)to(-0.5,3.9);
          \draw(4.8,1.4) node{$ \cdots$}; \draw(4.8,3.7) node{$ \cdots$};
          \draw[->,thick](6,1.1)to(6,3.9);
         \draw(-0.5,4.5)node{\tiny$s$};
           \end{tikzpicture}}.
         \end{aligned}\end{equation}
  \end{proof}

  The algebra  $H_{\ell, n}(\mathbf e)$ is a cellular algebra in the sense of \cite{GL} with certain cellular basis given in \cite[Theorem~6.3]{AMR}.  The corresponding cell modules are denoted by   $S(\lambda)$, $\lambda\in \Lambda_{\ell, n}$, where  $\Lambda_{\ell, n}$
is the set of all $\ell$-partitions $(\lambda^{(1)}, \lambda^{(2)}, \ldots, \lambda^{(\ell)})$  of $n$. When all   $e_j$'s  are in the same $\Bbb Z$-orbit in the sense that $e_i-e_j\in {\mathbb Z} 1_\Bbbk$ for all $1\le i< j\le \ell$, the complete set of pairwise inequivalent irreducible modules are given by
$$\{D(\lambda)\mid \lambda\in \bar\Lambda_{\ell, n} \}, $$
where $\bar\Lambda_{\ell, n}$ is the set of $\mathbf e$-restricted $\ell$-partitions in the sense of \cite[(3.14)]{K}.
Moreover,
    $D(\lambda)$ appears as the  simple head of $S(\lambda)$ for all $\lambda\in \bar\Lambda_{\ell, n}$ (e.g., \cite{K}). If $\mathbf e$ is  a disjoint union of certain orbits, then the above result on the classification of simple modules  is still available (see \cite[Remark 6.2]{GRS3}).

\subsection{Irreducible $\bar A_{\ob a}$-modules }\label{irrddjs} Suppose  $\ob a=(r,s)\in I$, where  $r, s\in\mathbb N$. Recall  $\mathbf u, \mathbf u'$ in \eqref{mathbfu}.
Let $\Lambda^\uparrow_{\ell, s}$ (resp., $\Lambda^\downarrow_{\ell, r}$, resp., $\bar\Lambda^\uparrow_{\ell, s}$, resp., $\bar\Lambda^\downarrow_{\ell, r}$) be the set of $\ell$-partitions of $s$ (resp., $\ell$-partitions of $r$, resp., $\mathbf u$-restricted $\ell$-partitions of $s$, resp.,  $-\mathbf u'$-restricted $\ell$-partitions of $r$).
 Define  \begin{equation}\label{bipar} \Lambda_{\ob a}=\Lambda^\downarrow_{\ell, r}\times  \Lambda^\uparrow_{\ell, s}, \quad  \bar \Lambda_{\ob a}=\bar\Lambda^\downarrow_{\ell, r} \times  \bar\Lambda^\uparrow_{\ell, s}.\end{equation}
 Thanks to Lemma~\ref{Aa},   $\bar A_{\downarrow^r\uparrow^s}$  is a cellular algebra with a cellular basis given  by those of $H_{\ell, r}(-\mathbf u') \bigotimes H_{\ell, s}(\mathbf u)$. In this case, the
 cell modules can be considered as $S(\lambda^\downarrow)\boxtimes S(\lambda^\uparrow)$'s, where $\lambda=(\lambda^\downarrow,\lambda^\uparrow)\in \Lambda_{\ob a}$. Furthermore, $\{D(\lambda^\downarrow)\boxtimes D(\lambda^\uparrow)|\lambda\in \bar\Lambda_\ob a\}$ gives a complete set of pairwise inequivalent irreducible  $\bar A_{\downarrow^r\uparrow^s}$-modules. As proved in \cite[\S~6.2]{GRS3}, each indecomposable projective  module of degenerate cyclotomic Hecke algebras is also the injective hull of the same irreducible  module. Let $P(\lambda^\downarrow)\boxtimes P(\lambda^\uparrow)$ be the projective cover (injective hull) of $D(\lambda^\downarrow)\boxtimes D(\lambda^\uparrow)$. For any $N\in \bar A_{\downarrow^r\uparrow^s}$-fdmod, $\beta(N)$ is an $\bar A_{\ob a }$-module where  $\beta$ is the functor given in \eqref{isomhecke2}. For any $\lambda=(\lambda^\downarrow,\lambda^\uparrow)\in \Lambda_\ob a$ and $\mu=(\mu^\downarrow,\mu^\uparrow)\in \bar\Lambda_\ob a$, define
  $$S(\lambda)=\beta( S(\lambda^\downarrow)\boxtimes S(\lambda^\uparrow)),  \ \  D(\mu)=\beta( D(\mu^\downarrow)\boxtimes D(\mu^\uparrow)), \ \
  P(\mu)=\beta(P(\mu^\downarrow)\boxtimes P(\mu^\uparrow)).$$
  Then   $P(\mu)$ is the  projective cover ( and injective hull) of the irreducible $\bar A_{\ob a}$-module  $D(\mu)$.

   Recall the anti-involution $\tau_A$ in subsection~\ref{OB}. Mimicking arguments in \cite{Re, GRS3}, we see that there is an exact contravariant duality functor $\circledast$  on $A\text{-lfdmod}$ (resp., $\bar A_\ob a\text{-fdmod}$) such that for any $V\in A$-lfdmod and $W\in \bar A_{\ob a}$-fdmod,  \begin{equation}
 V^\circledast=\bigoplus_{a\in J}\Hom_{\Bbbk}(1_{a}V,\Bbbk),\qquad W^\circledast=\Hom_{\Bbbk}(W,\Bbbk).
 \end{equation}
Thanks to Lemmas~\ref{isomhecke}(1),~\ref{Aa},  it is not difficult to verify that $\bar A_{\ob a}$ is a cellular algebra with a suitable cellular basis such that $S(\lambda)$'s are the corresponding cell modules. By \cite[Chapter~2, Exercise ~7]{Ma}, \begin{equation}\label{key1234}D(\lambda)^\circledast\cong D(\lambda)\end{equation} for all $\lambda\in\bar\Lambda_\ob a$.
Since $P(\lambda)$ is the projective cover and injective hull of $D(\lambda)$,
  \begin{equation}\label{pro}P(\lambda)^\circledast\cong P(\lambda).\end{equation}

\subsection{Induction and restriction functors}

Suppose  $a=\downarrow^r \uparrow^s$, $b=\downarrow^{r+1} \uparrow^{s}$ and $c=\downarrow^{r} \uparrow^{s+1}$, where $r, s\in \mathbb N$. For $ 2\le i\le r$ and  $2\le j\le s$,
define   $$\begin{aligned} L_i^\downarrow\otimes 1& =(S_{i-1}^\downarrow\otimes 1) (L_{i-1}^\downarrow \otimes 1)(S_{i-1}^\downarrow\otimes 1)+S_{i-1}^\downarrow\otimes 1,\\ 1\otimes L_j^\uparrow & =(1\otimes S_{j-1}^\uparrow) (1\otimes L_{j-1}^\uparrow  )(1\otimes S_{j-1}^\uparrow)+1\otimes S_{j-1}^\uparrow, \\
\end{aligned}$$
 where $S_{i-1}^\downarrow\otimes 1$, $L_{1}^\downarrow \otimes 1$, $1\otimes S_{j-1}^\uparrow$ and $1\otimes L_{1}^\uparrow $
  are given in \eqref{inde}.
 Then     $\{L_i^\downarrow\otimes 1, 1\otimes L_j^\uparrow\mid 1\le i\le r, 1\le j\le s\}$ generates a   commutative subalgebra of  $ \bar A_{\downarrow^r \uparrow^s}$. By Lemma~\ref{Aa} and  the well-known results on bases of cyclotomic Hecke algebras (e.g., \cite[Theorem~6.1]{AMR}), we have four exact functors: $$\text{res}_{r,s}^{r+1, s}:=\bar A_b\otimes_{\bar A_b} ?, \quad \text{ind}_{r,s}^{r+1, s}:=\bar A_b\otimes_{\bar A_a} ?, \quad \text{res}_{r,s}^{r, s+1}:=\bar A_c\otimes_{\bar A_c} ?, \quad \text{ind}_{r,s}^{r, s+1}:=\bar A_c\otimes_{\bar A_a} ?.$$
  Suppose $\phi: M\rightarrow M$ is a $\Bbbk$-linear map. For any $i\in \Bbbk$, the $i$-generalized eigenspace of $\phi$ is
 \begin{equation} \label{gesp} M_{i}=\{m\in M\mid (\phi-i)^n m=0, \forall n\gg 0\}.\end{equation}
 Then
 \begin{equation}\label{fen}\begin{aligned}&\text{res}_{r,s}^{r+1, s}=\bigoplus_{i\in\Bbbk}i\text{-res}_{r,s}^{r+1, s}, \quad \text{ind}_{r,s}^{r+1, s}=\bigoplus_{i\in\Bbbk}i\text{-ind}_{r,s}^{r+1, s}, \\&
 \text{res}_{r,s}^{r, s+1}=\bigoplus_{i\in\Bbbk}i\text{-res}_{r,s}^{r, s+1} \quad \text{ind}_{r,s}^{r, s+1}=\bigoplus_{i\in\Bbbk}i\text{-ind}_{r,s}^{r, s+1},\end{aligned}
 \end{equation}
 where $$\begin{aligned}&i\text{-res}_{r,s}^{r+1, s}(M) \text{ and } i\text{-ind}_{r,s}^{r+1, s}(N)\text{ are the generalized
$i$-eigenspaces of $-L_{r+1}^\downarrow\otimes 1$  on $M$ and $N$},\\&i\text{-res}_{r,s}^{r, s+1}(M') \text{ and } i\text{-ind}_{r,s}^{r, s+1}(N)\text{ are the generalized
$i$-eigenspaces of $1\otimes L_{s+1}^\uparrow$  on $M'$ and $N$},\end{aligned}$$ for any $(M, M', N)\in  \bar A_b\text{-fdmod}\times \bar A_c\text{-fdmod}\times  \bar A_a\text{-fdmod}$.

\subsection{Irreducible $A$-modules}\label{bi}
Following \cite[(3.1)-(3.2)]{GRS3},  there are  exact functors
\begin{equation}\label{exa}   \Delta=\bigoplus_{\ob a\in I}j^{\ob a}_!, ~~ \text{   $ \nabla=\bigoplus_{\ob a\in I}j^{\ob a}_*$,}\end{equation}
from  $ \bigoplus_{\ob a \in I} \bar A_{\ob a}\text{-fdmod}$ to $ A\text{-lfdmod}$ where
       $ j^{ \ob a}_!:= A_{\preceq \ob a}\bar 1_{ \ob a}\otimes_{\bar A_{\ob a}} ?$ and   $ j^{ \ob a}_*:=\bigoplus_{ \ob b\in I}\Hom_{\bar A_{\ob a}}(\bar 1_{ \ob a} A_{\preceq \ob a}\bar 1_{\ob b},?)$. Let
 \begin{equation}\label{lambdada} \Lambda= \bigcup_{\ob b\in I}\Lambda_\ob b, \quad \quad  ~~ \bar\Lambda=\bigcup_{\ob b\in I}\bar\Lambda_\ob b.\end{equation}

\begin{Defn}\label{stanpro}
For  any $\lambda\in \Lambda $ and $\mu\in \bar\Lambda $, let $\tilde\Delta(\lambda)=\Delta(S (\lambda))$,
 $\Delta(\mu)=\Delta(P (\mu))$, $\bar\Delta(\mu)= \Delta(D(\mu))$, $\nabla(\mu)=\nabla(P(\mu))$ and $\bar \nabla(\mu)= \nabla(D(\mu))$. \end{Defn}
Following \cite{LW}, $\Delta(\mu)$, $\bar\Delta(\mu)$, $\nabla(\mu)$ and $\bar \nabla(\mu)$ are called  the standard, proper standard, costandard and proper costandard  modules, respectively.

\begin{Cor}\label{irr}Suppose $\ob a=(r, s)$ and $\lambda\in\bar  \Lambda_\ob a$.
\begin{itemize}\item [(1)]$\Delta(\lambda)$
has a unique irreducible quotient  $L(\lambda)$ such that $\bar 1_\ob aL(\lambda)=D(\lambda)$ as $\bar A_\ob a$-modules.
\item  [(2)]  $\{L(\mu)\mid \mu\in \bar \Lambda\}$ is a complete set of pairwise inequivalent
irreducible $A$-modules.
\end{itemize}
\end{Cor}
\begin{proof} This result is a special case of \cite[Theorem~3.4(2)-(3)]{GRS3} which is available for any locally unital $\Bbbk$-algebra admitting an upper finite weakly triangular decomposition. Now, Proposition~\ref{WT} says that $A$ admits such a decomposition and hence   (1)-(2) follow from     Lemmas~\ref{isomhecke}(2) and \ref{Aa}.\end{proof}

\subsection{Stratified categories} A left  $A$-module $V$  has a finite $\Delta$-flag if it has a finite filtration  such that its  sections are  isomorphic to $\Delta (\lambda)$ for various  $\lambda\in \bar \Lambda$.
Let  \begin{equation}\label{rho}  \rho : \bar \Lambda\rightarrow I\end{equation}
 such that $ \rho(\lambda)=\ob a$  for any  $\lambda\in \bar\Lambda_\ob a$, where $I$ is given in Definition~\ref{ocbap}. Following \cite{BS}, define
 \begin{equation}\label{sgn}\Delta_\varepsilon(\lambda)=\left\{
                       \begin{array}{ll}
                         \Delta(\lambda), & \hbox{if $\varepsilon(\rho(\lambda))=+$,} \\
                         \bar\Delta(\lambda), & \hbox{if $\varepsilon(\rho(\lambda))=-$,}
                       \end{array}
                     \right.
 \end{equation}
 for  any given sign function $\varepsilon: I \rightarrow \{\pm\}$. Similarly we have the notion of finite $\Delta_\varepsilon$-flag.

\begin{Theorem}\label{COBMW1}\cite[Theorem~3.7]{GRS3} The
  $A$-lfdmod is an upper finite fully  stratified category in the sense of \cite[Definition~3.36]{BS} with respect to the stratification $\rho$ in \eqref{rho}. \end{Theorem}

In other words,  for each $\lambda\in \bar \Lambda$, there exists a projective object $P_\lambda$ admitting a finite  $\Delta_\varepsilon$-flag
with $\Delta_\varepsilon(\lambda)$ at the top and other sections $\Delta_\varepsilon(\mu)$ for $\mu\in\bar\Lambda$ with $\rho(\mu)\succeq \rho(\lambda)$.
Let $A$-mod$^{\Delta}$ be the category of all left $A$-modules with a finite $\Delta$-flag. Since $\Delta(\lambda)\in A\text{-lfdmod}$ for any $\lambda\in\bar \Lambda$, $A$-mod$^{\Delta}$ is a subcategory
of  $A\text{-lfdmod}$. For any $V\in A$-{\rm mod}$^{\Delta}$, let $(V:\Delta(\lambda))$  be    the   multiplicity
of $\Delta(\lambda)$ in a $\Delta$-flag of $V$.
  For any  simple $A$-module $L$ and any  $A$-module $V$, define  $$[V:L]=\text{sup}|\{i\mid V_{i+1}/V_i\cong L\}|$$ the supremum being taken over all filtrations by submodules $0=V_0\subset \cdots \subset V_n=V$.

\begin{Cor}\label{ijxxexeu} For any  $\lambda\in\bar\Lambda_\ob a$, let $\mathbf P(\lambda)$ be the projective cover of $L(\lambda)$.
  \begin{itemize}
  \item [(1)] $\mathbf P(\lambda)\in A$-mod$^{\Delta}$, and $(\mathbf P(\lambda): \Delta(\mu))=[\bar\Delta(\mu): L(\lambda)]$, which is  non-zero   for  $\mu\neq \lambda$  only if  $\mu\in \bigcup_{\ob c\succ \ob a} \bar \Lambda_\ob c$.
 In particular, $(\mathbf P(\lambda):\Delta(\lambda))=1$.
 \item[(2)] $\mathbf P(\lambda)$ has a finite $\tilde\Delta$-flag.     If $\tilde \Delta( \mu)$ appears as a section, then $\mu\in \bigcup_{\ob b\succeq \ob a}  \Lambda_\ob b$. Furthermore, the  multiplicity of $\tilde \Delta( \mu)$ in this flag is  $[\tilde \Delta( \mu):L(\lambda)]$.
\end{itemize}\end{Cor}
\begin{proof}  Thanks to \eqref{key1234} and \eqref{pro} (i.e. \cite[Assumption~3.12]{GRS3} holds for $A$), (1) is a
special case of \cite[Proposition 3.9(2), Lemma~3.13(2)]{GRS3} and (2) is a special case of \cite[Corollary 4.4(2)]{GRS3}.
\end{proof}

\section{ Endofunctors and categorical actions }\label{indkkk}
Motivated by  \cite{Br,Re, GRS3}, we study certain endofunctors so as to  give   a categorical action on   $A $-lfdmod.

\subsection{Endofunctors} For any $\diamond\in \{\uparrow, \downarrow\}$, define
 $A_\diamond=\bigoplus_{a, b\in J} 1_{ a} A_\diamond1_{b} $ and $_\diamond A=\bigoplus_{a, b\in J} 1_{ a} ({_\diamond }A)1_{ b}$, where
\begin{equation}\label{weights123} 1_{ a}  ({_\diamond A})1_{ b}=(1_{a \diamond })A1_{ b}, \quad  1_{ a} A_\diamond1_{ b}=1_{a}A1_{ b \diamond} .\end{equation}
 Then both  $_\diamond A$ and  $A_\diamond$ are   $(A,A)$-bimodules
 such that the right (resp., left) action of  $A$  on $_\diamond A$ (resp., $A_\diamond$) is given by the usual multiplication, whereas the left (resp., right) action of $A$ on   $_\diamond A$ (resp., $A_\diamond$) is given as follows:
\begin{equation}\begin{aligned}\label{act123}& a\cdot m= \begin{tikzpicture}[baseline = 3mm, color=\clr]
                \draw (-0.1,0.4) rectangle (1.2,0);
                \draw(0.5, 0.2) node{$m$};
                \draw (0.9,0.55) rectangle (-0.1,0.85); \draw[-,thick] (0, 0.4)to[out=up,in=down](0,0.55); \draw(0.3, 0.48) node{$ \cdots$};   \draw[-,thick] (0.6,0.4)to[out=up,in=down](0.6,0.55);
                \draw(0.5, 0.7) node{$a$}; \draw[<-,thick] (1.06,0.4)to[out=up,in=down](1.06,0.85);
    \end{tikzpicture}, \quad  g\cdot a=\begin{tikzpicture}[baseline = 3mm, color=\clr]
                \draw (-0.1,0.4) rectangle (0.9,0);
                \draw(0.4, 0.2) node{$a$};
                \draw[-<,thick] (1.06,0.53)to[out=up,in=down](1.06,0.01);
                \draw (1.2,0.55) rectangle (-0.1,0.85); \draw[-,thick] (0, 0.4)to[out=up,in=down](0,0.55); \draw(0.3, 0.48) node{$ \cdots$};  \draw[-,thick] (0.6,0.4)to[out=up,in=down](0.6,0.55);
                \draw(0.6, 0.7) node{$g$};
    \end{tikzpicture},\text{   for all   $ (m, g, a)\in 1_{b\downarrow} A\times   A1_{c\downarrow}\times  1_cA1_b $; }
    \\&
    a\cdot m= \begin{tikzpicture}[baseline = 3mm, color=\clr]
                \draw (-0.1,0.4) rectangle (1.2,0);
                \draw(0.5, 0.2) node{$m$};
                \draw (0.9,0.55) rectangle (-0.1,0.85); \draw[-,thick] (0, 0.4)to[out=up,in=down](0,0.55); \draw(0.3, 0.48) node{$ \cdots$};   \draw[-,thick] (0.6,0.4)to[out=up,in=down](0.6,0.55);
                \draw(0.5, 0.7) node{$a$}; \draw[->,thick] (1.06,0.4)to[out=up,in=down](1.06,0.85);
    \end{tikzpicture}, \quad  g\cdot a=\begin{tikzpicture}[baseline = 3mm, color=\clr]
                \draw (-0.1,0.4) rectangle (0.9,0);
                \draw(0.4, 0.2) node{$a$};
                \draw[->,thick] (1.06,0.01)to[out=up,in=down](1.06,0.53);
                \draw (1.2,0.55) rectangle (-0.1,0.85); \draw[-,thick] (0, 0.4)to[out=up,in=down](0,0.55); \draw(0.3, 0.48) node{$ \cdots$};  \draw[-,thick] (0.6,0.4)to[out=up,in=down](0.6,0.55);
                \draw(0.6, 0.7) node{$g$};
    \end{tikzpicture},\text{   for all   $ (m, g, a)\in 1_{b\uparrow} A\times   A1_{c\uparrow}\times  1_cA1_b $. }\end{aligned}\end{equation}
  Similarly, we have    $(\bar A^\circ, \bar A^\circ)$-bimodules
   $\bar A^\circ_\diamond=\bigoplus_{a, b\in J}\bar 1_{ a} \bar A^\circ_\diamond\bar1_{b} $ and $_\diamond \bar A^\circ=\bigoplus_{a, b\in J}\bar 1_{ a} ({_\diamond }\bar A^\circ)\bar1_{ b}$    such that
\begin{equation}\label{weights123circ} \bar1_{ a}  ({_\diamond \bar A^\circ})\bar1_{ b}=(\bar1_{a \diamond })\bar A^\circ\bar1_{ b},\quad \text{ $\bar1_{ a} \bar A^\circ_\diamond\bar1_{ b}=\bar1_{a}\bar A^\circ\bar1_{ b \diamond} $,}\end{equation} where $\bar A^\circ$ is given  in Lemma~\ref{isomhecke}.

    \begin{Prop}\label{isob} As $(A,A)$-bimodules, $A_\uparrow\cong {_\downarrow A}$ and $A_\downarrow\cong {_\uparrow A}$.
    \end{Prop}
    \begin{proof} Thanks to Lemma~\ref{cellbasis}(1) and \eqref{weights123}, there are four $\Bbbk$-linear maps
    $\phi: {_\uparrow A}\rightarrow A_\downarrow$, $\psi: A_\downarrow\rightarrow {_\uparrow A}$, $\phi': {_\downarrow A}\rightarrow A_\uparrow$ and $\psi': A_\uparrow\rightarrow {_\downarrow A}$ such that  $$\phi(m)=\begin{tikzpicture}[baseline = 1mm, color=\clr]
                \draw[-,thick] (0,0.5)to[out=up,in=down](0,0.8); \draw(0.2, 0.6) node{\tiny$ \cdots$};  \draw[-,thick] (0.4,0.5)to[out=up,in=down](0.4,.8);  \draw[-,thick,darkblue] (0.5,0.5) to[out=up,in=left] (0.78,0.78) to[out=right,in=up] (1.06,0.5);\draw (-0.1,0.5) rectangle (0.9,0);
                \draw(0.4, 0.25) node{$m$}; \draw[-<,thick] (1.06,0.5)to[out=up,in=down](1.06,0.01);
    \end{tikzpicture},\qquad\psi(g)=\begin{tikzpicture}[baseline = 1mm, color=\clr]
                \draw[-,thick] (0,-0.3)to[out=up,in=down](0,0); \draw(0.2, -0.2) node{\tiny$ \cdots$};  \draw[-,thick] (0.4,-0.3)to[out=up,in=down](0.4,0);  \draw[-,thick,darkblue] (0.5,0) to[out=down,in=left] (0.78,-0.28) to[out=right,in=down] (1.06,0);\draw (-0.1,0.5) rectangle (0.9,0);
                \draw(0.4, 0.25) node{$g$}; \draw[->,thick] (1.06,0.01)to[out=up,in=down](1.06,0.5);
    \end{tikzpicture},\qquad
    \phi'(m')=\begin{tikzpicture}[baseline = 1mm, color=\clr]
                \draw[-,thick] (0,0.5)to[out=up,in=down](0,0.8); \draw(0.2, 0.6) node{\tiny$ \cdots$};  \draw[-,thick] (0.4,0.5)to[out=up,in=down](0.4,.8);  \draw[-,thick,darkblue] (0.5,0.5) to[out=up,in=left] (0.78,0.78) to[out=right,in=up] (1.06,0.5);\draw (-0.1,0.5) rectangle (0.9,0);
                \draw(0.4, 0.25) node{$m'$}; \draw[->,thick] (1.06,0.01)to[out=up,in=down](1.06,0.5);
    \end{tikzpicture},\qquad\psi'(g')=\begin{tikzpicture}[baseline = 1mm, color=\clr]
                \draw[-,thick] (0,-0.3)to[out=up,in=down](0,0); \draw(0.2, -0.2) node{\tiny$ \cdots$};  \draw[-,thick] (0.4,-0.3)to[out=up,in=down](0.4,0);  \draw[-,thick,darkblue] (0.5,0) to[out=down,in=left] (0.78,-0.28) to[out=right,in=down] (1.06,0);\draw (-0.1,0.5) rectangle (0.9,0);
                \draw(0.4, 0.25) node{$g'$}; \draw[<-,thick] (1.06,0.01)to[out=up,in=down](1.06,0.5);
    \end{tikzpicture}~,
    $$ for all basis elements  $m, g, m', g'$ of ${_\uparrow A}$, $ A_\downarrow$, ${_\downarrow A}$ and $ A_\uparrow$ given in Lemma~\ref{cellbasis}(1), respectively.
By \eqref{relation 7}-\eqref{relation 8} (resp., \eqref{relation 1}-\eqref{relation 2}), $\phi^{-1}=\psi$ (resp., $\phi'^{-1}=\psi'$).
Finally, it is easy to verify  that  both $\phi$ and  $\phi'$ are $(A,A)$-homomorphisms.
    \end{proof}
Thanks to Proposition~\ref{isob},  ${_\uparrow A}\otimes _{A}?\cong {A_\downarrow}\otimes _{A}?$ and ${_\downarrow A}\otimes _{A}?\cong {A_\uparrow}\otimes _{A}?$ as functors. Define
 \begin{equation}\label{EF} E={_\uparrow A}\otimes _{A}?, \quad   F={_\downarrow A}\otimes _{A}?.\end{equation}

\begin{Defn}\label{etaep} Suppose that $E$ and $F$ are two functors in \eqref{EF}. Define four natural transformations
  $$\eta: Id_{A\text{-mod}}\rightarrow FE,\ \  \eta': Id_{A\text{-mod}}\rightarrow EF,\  \ \varepsilon: EF\rightarrow Id_{A\text{-mod}}, \ \ \varepsilon': FE \rightarrow Id_{A\text{-mod}}$$
such that
\begin{itemize}\item[(1)] $\eta$ and $\eta'$ are induced  by the $(A, A)$-homomorphisms $\alpha: A\rightarrow {_\downarrow A} \otimes_A ({_\uparrow A}) $ and  $\alpha': A\rightarrow {_\uparrow A} \otimes_A ({_\downarrow A})  $, respectively, where \begin{equation}\label{alphaa}
\alpha(f)=f\text{ } \xli\otimes 1_a\text{ }\lcupl,\quad \quad  ~ \alpha'(f)=f\text{ } \sli\otimes 1_a\text{ }\lcup,\qquad \forall f\in 1_a A \text{ and $a\in J$}.
\end{equation}
\item [(2)]   $\varepsilon$ and $\varepsilon'$ are  induced by the $(A, A)$-homomorphisms $\beta: {_\uparrow A} \otimes_A ({_\downarrow A})\rightarrow A$ and  $\beta': {_\downarrow A} \otimes_A  ({_\uparrow A}) \rightarrow A$, respectively,  where
 \begin{equation}
 \beta(f\otimes g)=\begin{tikzpicture}[baseline = 3mm, color=\clr]
                \draw (-0.1,0.4) rectangle (1.2,0);
                \draw(0.5, 0.2) node{$g$};
                \draw (0.9,0.55) rectangle (-0.1,0.85); \draw[-,thick] (0, 0.4)to[out=up,in=down](0,0.55); \draw(0.3, 0.48) node{$ \cdots$};   \draw[-,thick] (0.6,0.4)to[out=up,in=down](0.6,0.55);
                \draw(0.45, 0.7) node{$f$}; \draw[<-,thick] (1.06,0.4)to[out=up,in=down](1.06,0.85);
                 \draw[-,thick,darkblue] (0.7,0.85) to[out=up,in=left] (0.88,1.05) to[out=right,in=up] (1.06,0.85);
    \end{tikzpicture},~ \beta'(f_1\otimes g_1)
    =\begin{tikzpicture}[baseline = 3mm, color=\clr]
                \draw (-0.1,0.4) rectangle (1.2,0);
                \draw(0.5, 0.2) node{$g_1$};
                \draw (0.9,0.55) rectangle (-0.1,0.85); \draw[-,thick] (0, 0.4)to[out=up,in=down](0,0.55); \draw(0.3, 0.48) node{$ \cdots$};   \draw[-,thick] (0.6,0.4)to[out=up,in=down](0.6,0.55);
                \draw(0.45, 0.7) node{$f_1$}; \draw[-,thick] (1.06,0.4)to[out=up,in=down](1.06,0.85);
                 \draw[<-,thick,darkblue] (0.7,0.85) to[out=up,in=left] (0.88,1.05) to[out=right,in=up] (1.06,0.85);
    \end{tikzpicture}
    \end{equation}
for all $ (f, g, f_1, g_1)\in {_\uparrow A}1_a\times 1_{a\downarrow}A\times  {_\downarrow A}1_a\times 1_{a\uparrow}A$ and all $a\in J$.\end{itemize} \end{Defn}

\begin{Lemma}\label{bijiont}  $E$ and $F$ are biadjoint to each other.\end{Lemma}
\begin{proof}
    Using \eqref{relation 7}-\eqref{relation 8} (resp., \eqref{relation 1}-\eqref{relation 2}) yields $
    \varepsilon E  \circ  E  \eta =\text{Id}_{ E}$, $F\varepsilon\circ \eta F=  \text{Id}_F$,    $\varepsilon' F  \circ  F  \eta' =\text{Id}_{ F}$ and $ E\varepsilon'\circ \eta' E=  \text{Id}_E$. So, $E$ and $F$ are biadjoint to each other.
\end{proof}

\begin{Lemma}\label{jme} Suppose  $a, b\in J$. \begin{itemize} \item [(1)]   $(1_a \xd)\circ (m \text{ }\sli)=(m\text{ }\sli)\circ(1_b \xd)$ and $(1_a \xdx)\circ (m\text{ } \xli)=(m\text{ } \xli)\circ(1_b \xdx)$ for all  $ m\in 1_aA1_b$.
 \item [(2)] For any $\diamond\in \{\uparrow, \downarrow\}$, the linear map  $x^\diamond$ is $(A,A)$-homomorphism, where
  $x^\diamond\in\End_\Bbbk({_\diamond A})$  such that
    $x^\uparrow (m)=(1_a \xd)\circ m$ and $x^\downarrow (m')=(1_a \xdx)\circ m'$ for all   $(m ,m')\in 1_{a \uparrow} A\times 1_{a\downarrow} A$.\item[(3)]  For any $\diamond\in \{\uparrow, \downarrow\}$,  both
     $x^\diamond_L$ and $x^\diamond_R$ are $(\bar A^\circ, \bar A^\circ)$-homomorphisms,  where
    $x^\diamond_L\in\End_\Bbbk({_\diamond \bar A^\circ})$ and $x^\diamond_R\in\End_\Bbbk({ \bar A^\circ_\diamond})$ such that
    $x^\uparrow_L (m)=(\bar{1_a \xd})\circ m$, $ x^\downarrow_L (m')=(\bar{1_a \xdx})\circ m'$,
    $x^\uparrow_R (n)=n\circ (\bar{1_a \xd})$ and  $ x^\downarrow_R (n')=n'\circ (\bar{1_a \xdx})$
 for any $(m ,m', n, n')\in \bar1_{a \uparrow}\bar A^\circ\times \bar1_{a\downarrow}  \bar A^\circ\times \bar A^\circ\bar1_{a \uparrow} \times \bar A^\circ \bar 1_{a\downarrow}$.
    \end{itemize}\end{Lemma}

  \begin{proof} (1) is trivial and (2)-(3) follow immediately from  (1).
     \end{proof}
 For any $i\in \Bbbk$, define   \begin{equation}\label{Eii} E_i=(_\uparrow A)_i\otimes_A?, \quad    F_i=(_\downarrow A)_i\otimes_A?,\end{equation}  where  $(_\diamond A)_i$ is the generalized $i$-eigenspace of $x^\diamond$ on $_\diamond A$, $\diamond\in \{\uparrow, \downarrow\}$.
  \begin{Lemma}\label{Ei}
  $E=\bigoplus_{i\in\Bbbk} E_i$ and $ F=\bigoplus_{i\in\Bbbk} F_i$.
   \end{Lemma}  \begin{proof} Thanks to Lemma~\ref{cellbasis}(1),  $1_a({_\diamond A})1_b$  is a finite dimensional $\Bbbk$-space for any  $a,b\in J$ and any $\diamond\in \{\uparrow, \downarrow\}$. Since $x^\diamond$  preserves any  $1_a({_\diamond A})1_b$, the result follows.\end{proof}
\begin{Lemma}\label{bimoudisomah}  For any $i\in\Bbbk$,  $E_i$ and $F_i$ are biadjoint to each other.
\end{Lemma}
\begin{proof} Recall $\epsilon, \epsilon', \eta$ and $\eta'$ in Definition~\ref{etaep}. Suppose  $M$ and $N$ are two left  $A$-modules. Thanks to the proof of
 Lemma~\ref{bijiont},  there are  $\Bbbk$-linear isomorphisms\begin{equation} \label{EEE}\begin{aligned}&\Hom_A(EM, N)\overset{\alpha_{M, N}}\rightarrow  \Hom_A(M, FN), \qquad \Hom_A(M, FN)\overset{\alpha_{M, N}^{-1}}\rightarrow  \Hom_A(EM, N),\\& \Hom_A(FM, N)\overset{\beta_{M, N}}\rightarrow  \Hom_A(M, EN), \qquad \Hom_A(M, EN)\overset{\beta_{M, N}^{-1}}\rightarrow  \Hom_A(FM, N),\end{aligned}\end{equation} such that $$\begin{aligned}&\alpha_{M,N}(h)=F(h)\circ \eta_M,\qquad\alpha^{-1}_{M,N}(h')=\varepsilon_N\circ E(h'),\\&
     \beta_{M,N}(h_1)=E(h_1)\circ \eta'_M,\qquad \beta^{-1}_{M,N}(h'_1)=\varepsilon'_N\circ F(h'_1), \end{aligned}$$ for all $(h, h', h_1, h'_1)\in \Hom_A(EM, N)\times \Hom_A(M, FN)\times \Hom_A(FM, N)\times  \Hom_A(M, EN)$.

     For any $(h, f, m)\in \Hom_A(E_iM, N) \times  1_aA\times M$, we have (under the  isomorphism $A\otimes_AM\cong M$)
     \begin{equation}\label{E}\begin{aligned}\alpha_{M, N}(h)(f\otimes m)&=F(h)\circ\eta_M(f\otimes m)\\&=F(h)(1_a\text{ }\xli\otimes f\text{ }\lcupl\otimes m)=F(h)(f\text{ }\xli\otimes 1_a\text{ }\lcupl\otimes m).\end{aligned}\end{equation}
     Write $f~\xli=\sum_{j\in\Bbbk}f_j~\xli$,
     where $f_j~\xli\in ({_\downarrow}A)_j$.  Thanks to  Lemma~\ref{dots}(7), $f_j\text{ }\lcupl\in ({_\uparrow}A)_j$.
     Note that  $h(E_jM)=0$ if $i\neq j$.  By \eqref{E}, $$\alpha_{M,N}(h)(f\otimes m)= f_i~\xli\otimes h( 1_a\text{ }\lcupl\otimes m)\in F_iN.$$ So,
 $\alpha_{M,N}(h)\in \Hom_A(M, F_iN)$.

 For any $(h, m)\in \Hom_A(M, F_iN)\times M$, there are  finite numbers of  $f_k\otimes n_k$'s $\in ({_\downarrow A})_i\otimes N$ such that \begin{equation} \label{hm}  h(m)=\sum_kf_k\otimes n_k.\end{equation}
  For all $(f, m)\in{_\uparrow A}\times M$,  we have (under the  isomorphism $A\otimes_AM\cong M$)
 \begin{equation}\label{F}\alpha_{M, N}^{-1}(h)(f\otimes m)=\varepsilon_NE(h)(f\otimes m)=\varepsilon_N(f\otimes \sum_k f_k\otimes n_k)=\sum_k\begin{tikzpicture}[baseline = 3mm, color=\clr]
                \draw (-0.1,0.4) rectangle (1.2,0);
                \draw(0.5, 0.2) node{$f_k$};
                \draw (0.9,0.55) rectangle (-0.1,0.85); \draw[-,thick] (0, 0.4)to[out=up,in=down](0,0.55); \draw(0.3, 0.48) node{$ \cdots$};   \draw[-,thick] (0.6,0.4)to[out=up,in=down](0.6,0.55);
                \draw(0.45, 0.7) node{$f$}; \draw[<-,thick] (1.06,0.4)to[out=up,in=down](1.06,0.85);
                 \draw[-,thick,darkblue] (0.7,0.85) to[out=up,in=left] (0.88,1.05) to[out=right,in=up] (1.06,0.85);
    \end{tikzpicture}\otimes n_k.\end{equation}
We claim $\alpha_{M, N}^{-1}(h)(f\otimes m)=0$ if   $f\in ( {_\uparrow A})_j$ for any $j\ne i$. If so,  $\alpha_{M, N}^{-1}(h)\in\Hom_A(E_iM, N)$ and the restriction  $\alpha_{M, N}$ in \eqref{EEE} to $\Hom_A(E_iM, N)$ is an isomorphism between $\Hom_A(E_iM, N)$ and $\Hom_A(M, F_iN)$.  This proves that $(E_i, F_i)$ is an adjoint pair.

    In fact,      $$(x^\uparrow-j)^t\alpha_{M, N}^{-1}(h)(f\otimes m)=\alpha_{M, N}^{-1}(h)((x^\uparrow-j)^tf\otimes n)=0$$ for some integer  $t\gg 0$ if   $f\in ( {_\uparrow A})_j$ and $i\neq j$.
    Similarly,    $(x^\downarrow-i)^sf_k=0$ for some integer  $s\gg 0$, where $f_k$'s  are  given \eqref{hm}.
     Since there are only finite number of $f_k$, we can find an $s\gg 0$ which is independent of $f_k$  such that $(x^\downarrow-i)^sf_k=0$ for all admissible $k$.
     Thanks to \eqref{F} and Lemma~\ref{dots}(8), $$(x^\uparrow-i)^s\alpha_{M, N}^{-1}(h)(f\otimes m)=\alpha_{M, N}^{-1}(h)((x^\uparrow-i)^sf\otimes n)=\sum_k
    \begin{tikzpicture}[baseline = 1mm, color=\clr]
                \draw[-,thick] (0,0.5)to[out=up,in=down](0,0.8); \draw(0.2, 0.6) node{\tiny$ \cdots$};  \draw[-,thick] (0.4,0.5)to[out=up,in=down](0.4,.8);  \draw[-,thick,darkblue] (0.5,0.5) to[out=up,in=left] (0.78,0.78) to[out=right,in=up] (1.06,0.5);\draw (-0.1,0.5) rectangle (0.9,0);
                \draw(0.4, 0.25) node{$f$}; \draw[-<,thick] (1.06,0.5)to[out=up,in=down](1.06,0.01);
    \end{tikzpicture}\circ (x^\downarrow-i)^sf_k\otimes n_k=0,$$  forcing $\alpha_{M, N}^{-1}(h)(f\otimes m)=0$. This proves our claim.

  Similarly, the restriction of $\beta_{M, N}$ in \eqref{EEE} to $\Hom_A(F_iM, N)$ induces an isomorphism between $\Hom_A(F_iM, N)$ and $\Hom_A(M, E_iN)$. The only difference is that one has to replace Lemma~\ref{dots}(7)--\ref{dots}(8)  by Lemma~\ref{dots}(5)--(6). This proves that $(F_i, E_i)$ is an adjoint pair.
\end{proof}

 Recall $x^\diamond_L$ and $x^\diamond_R$ in Lemma~\ref{jme}(3) for any  $\diamond\in \{\uparrow, \downarrow\}$. For any  $i\in \Bbbk$  let $(_\diamond \bar A^\circ)_i$ (resp., $ (\bar A^\circ_\diamond)_i$) be the generalized $i$-eigenspace of $x^\diamond_L$ (resp., $x^\diamond_R$) on ${_\diamond \bar A^\circ}$ (resp., $ \bar A^\circ_\diamond$). Define  \begin{equation} \label{efsx1} E^\diamond={_\diamond \bar A^\circ}\otimes _{\bar A^\circ}?, \quad F^\diamond={ \bar A^\circ_\diamond}\otimes _{\bar A^\circ}?.\end{equation}

\begin{Lemma}\label{EFSX}  $E^\diamond=\bigoplus_{i\in\Bbbk} E_i^\diamond$ and $ F^\diamond=\bigoplus_{i\in\Bbbk} F_i^\diamond$,
     where $E_i^\diamond=(_\diamond \bar A^\circ)_i\otimes _{\bar A^\circ}?$ and $F_i=(\bar A^\circ_\diamond )_i\otimes _{\bar A^\circ}?$.\end{Lemma}
\begin{proof}The result follows from the fact that $x^\diamond_L$ (resp., $x^\diamond_R$ ) preserves the finite dimensional $\Bbbk$-spaces $\bar1_a({_\diamond \bar A^\circ})\bar1_b$'s  and  $\bar1_a({ \bar A^\circ_\diamond})\bar1_b$'s.
 \end{proof}

  Thanks to Lemma~\ref{isomhecke}(2) and  \eqref{morif}, we have
\begin{equation}\label{contss} E^\uparrow\sim\bigoplus_{r, s\in\mathbb N} \text{res}^{r, s+1}_{r, s},\quad  E^\downarrow\sim \bigoplus_{r, s\in\mathbb N} \text{ind}^{r+1, s}_{r, s},\quad F^\uparrow\sim\bigoplus_{r, s\in\mathbb N} \text{ind}^{r, s+1}_{r, s},\quad F^\downarrow\sim \bigoplus_{r, s\in\mathbb N} \text{res}^{r+1, s}_{r, s}.
\end{equation}
So,  both $E^\diamond$ and $F^\diamond$ are exact. By Lemma~\ref{Aa},  \eqref{inde} and \eqref{fen},    \begin{equation}\label{word}\begin{aligned}
&E^\uparrow_i\sim\bigoplus_{r, s\in\mathbb N} i\text{-res}^{r, s+1}_{r, s},\text{ }  E^\downarrow_i\sim \bigoplus_{r, s\in\mathbb N} i\text{-ind}^{r+1, s}_{r, s},\text{ }\\&  F^\uparrow_i\sim\bigoplus_{r, s\in\mathbb N} i\text{-ind}^{r, s+1}_{r, s},\text{ }  F^\downarrow_i\sim \bigoplus_{r, s\in\mathbb N} i\text{-res}^{r+1, s}_{r, s}.\end{aligned}
\end{equation}

\begin{Lemma}\label{intailh1}
There are two short exact sequence of functors from $\bar A^\circ$-fdmod to $A$-lfdmod:
\begin{equation}\begin{aligned}\label{diejdhd1}
&0\rightarrow \Delta\circ  F^\downarrow\rightarrow  F\circ \Delta\rightarrow\Delta\circ  F^\uparrow\rightarrow0;\\&0\rightarrow \Delta\circ  E^\uparrow\rightarrow  E\circ \Delta\rightarrow\Delta\circ  E^\downarrow\rightarrow0
.
\end{aligned}\end{equation}
\end{Lemma}
\begin{proof} Suppose $\ob a=(s, r)$, $\ob b=( s-1,r)$, $\ob c=(s,r+1)$, where   $ r, s$ are any admissible non-negative integers. We claim that  there are  short exact sequences of $(A,\bar A^\circ)$-bimodules
\begin{equation}\label{shoretmor1}\begin{aligned}&
0\rightarrow A_{\preceq \ob b}\otimes _{\bar A_{\ob b}^\circ} ({_\downarrow \bar A}^\circ)\bar 1_\ob a\overset {\varphi}\rightarrow {_\downarrow A}\otimes _{A}A_{\preceq \ob a}\bar 1_\ob a\overset{\psi}\rightarrow A_{\preceq \ob c}\otimes _{ \bar A_{\ob c}^\circ}\bar A_\uparrow^\circ\bar 1_ \ob a\rightarrow0,
\\&
0\rightarrow A_{\preceq \ob b}\otimes _{\bar A_{\ob b}^\circ} ({_\uparrow \bar A}^\circ)\bar 1_\ob a\overset {\eta}\rightarrow {_\uparrow A}\otimes _{A}A_{\preceq \ob a}\bar 1_\ob a\overset{\varepsilon}\rightarrow A_{\preceq \ob c}\otimes _{ \bar A_{\ob c}^\circ}\bar A_\downarrow^\circ\bar 1_ \ob a\rightarrow0, \end{aligned}
\end{equation}
where the required morphisms  $\varphi, \psi, \eta$ and $\varepsilon$ satisfy:
$$\begin{aligned}&\varphi(\bar f\otimes \bar g)= ( f~\begin{tikzpicture}[baseline = 10pt, scale=0.5, color=\clr]
                \draw[<-,thick] (0,0.5)to[out=up,in=down](0,1.2);
    \end{tikzpicture})\otimes \bar g, \ \ \text{    $\psi( f_1\otimes \bar g_1)=\bar{\begin{tikzpicture}[baseline = 1mm, color=\clr]
                \draw[-,thick] (-0.06,0.5)to[out=up,in=down](-0.06,0.8); \draw(0.2, 0.6) node{$ \tiny\cdots$};  \draw[-,thick] (0.4,0.5)to[out=up,in=down](0.4,.8);  \draw[<-,thick,darkblue] (0.5,0.5) to[out=up,in=left] (0.78,0.78) to[out=right,in=up] (1.06,0.5);\draw (-0.1,0.5) rectangle (0.9,0);
                \draw(0.4, 0.25) node{$ f_1$}; \draw[-,thick] (1.06,0.5)to[out=up,in=down](1.06,0.01);
    \end{tikzpicture}}\otimes (  \bar{g_1~\begin{tikzpicture}[baseline = 10pt, scale=0.5, color=\clr]
                \draw[->,thick] (0,0.5)to[out=up,in=down](0,1.2);
    \end{tikzpicture}~})$ },\\&
    \eta(\bar h\otimes \bar k)= ( h~\begin{tikzpicture}[baseline = 10pt, scale=0.5, color=\clr]
                \draw[->,thick] (0,0.5)to[out=up,in=down](0,1.2);
    \end{tikzpicture})\otimes \bar k, \ \ \text{    $\varepsilon( h_1\otimes \bar k_1)=\bar{\begin{tikzpicture}[baseline = 1mm, color=\clr]
                \draw[-,thick] (-0.06,0.5)to[out=up,in=down](-0.06,0.8); \draw(0.2, 0.6) node{$ \tiny\cdots$};  \draw[-,thick] (0.4,0.5)to[out=up,in=down](0.4,.8);  \draw[->,thick,darkblue] (0.5,0.5) to[out=up,in=left] (0.78,0.78) to[out=right,in=up] (1.06,0.5);\draw (-0.1,0.5) rectangle (0.9,0);
                \draw(0.4, 0.25) node{$ h_1$}; \draw[<-,thick] (1.06,0.5)to[out=up,in=down](1.06,0.01);
    \end{tikzpicture}}\otimes (  \bar{k_1~\begin{tikzpicture}[baseline = 10pt, scale=0.5, color=\clr]
                \draw[<-,thick] (0,0.5)to[out=up,in=down](0,1.2);
    \end{tikzpicture}~})$ },
    \end{aligned}
$$
for any admissible basis diagrams  $f,g, f_1,g_1, h,k,h_1,k_1$ of $A$ in  Lemma~\ref{cellbasis}(1).
If the claim is true, then   \eqref{diejdhd1} follows immediately.

It is routine  to check that   $\varphi $, $\psi$, $\varepsilon $ and $\eta$ are well-defined  $(A,\bar A^\circ)$-homomorphisms.
We are going to prove the exactness of the  first  sequence in \eqref{shoretmor1}.    One can  verify the second one   similarly.

    Suppose  $d\sim c \in J$ and  $a, b\in \ob a$.  Thanks to  Lemma~\ref{unie} and \eqref{isomor}, there is a commutative diagram

\begin{equation}\label{conmuf1}
\begin{aligned}
0\rightarrow\bar 1_dA_{\preceq \ob b}\otimes _{\bar A_{\ob b}^\circ} & ({_\downarrow \bar A}^\circ)\bar 1_a
\overset{\varphi_{a,d}}\longrightarrow 1_d ({_\downarrow A}\otimes _{A}A_{\preceq \ob a})\bar 1_a\overset{\psi_{a,d}}\longrightarrow \bar 1_dA_{\preceq \ob c}\otimes _{ \bar A_{\ob c}^\circ}\bar A_\uparrow^\circ\bar 1_a\rightarrow 0\\
\wr\downarrow\tau_1\quad\quad & \quad\quad\quad\quad\quad \quad\wr\downarrow \tau_2\quad\quad  \quad\quad\quad\quad\quad \quad\wr\downarrow \tau_{3}\\0\rightarrow
\bar 1_{c}A_{\preceq \ob b}\otimes _{\bar A_{\ob b}^\circ} & ({_\downarrow \bar A}^\circ)\bar 1_b
\overset{\varphi_{b,c}}\longrightarrow 1_c({_\downarrow A}\otimes _{A}A_{\preceq \ob a})\bar 1_b\overset{\psi_{b, c}}\longrightarrow \bar 1_cA_{\preceq \ob c}\otimes _{ \bar A_{\ob c}^\circ}\bar A_\uparrow^\circ\bar 1_b\rightarrow 0\\
\end{aligned},
\end{equation}
where  $\varphi_{a, d}$ and $\varphi_{b,c}$ (resp., $\psi_{a, d}$ and $\psi_{b,c}$) are restrictions of $\varphi$ (resp., $\psi$) in \eqref{shoretmor1}.
Here, the $\Bbbk$-linear isomorphisms $\tau_i$ are defined so that
$$\tau_2( f\otimes \bar g):={_c\sigma_d}(f\otimes \bar g) \bar{_a\sigma_b}, \quad    \tau_i(\bar f\otimes \bar g):={_c\sigma_d}(\bar f\otimes \bar g) \bar{_a\sigma_b},\quad  i\in\{1,3\},$$
for all admissible  basis diagrams $ f,   g $ of $A$,   where  ${_c}\sigma_d$ is given in Lemma~\ref{unie}.
Therefore,   it's enough to   verify  the exactness of the first   sequence in \eqref{conmuf1} as $\Bbbk$-spaces under the assumptions  $a=\uparrow^r\downarrow^s$ and $d=\uparrow^m\downarrow^n$, $\forall m, n\in \mathbb N$.
If so, we immediately obtain the first short exact sequence in \eqref{shoretmor1}.

We define $B_1\otimes B_2=\{b_1\otimes b_2\mid b_1\in B_1, b_2\in B_2\}$ for any  sets $B_1, B_2$.
For any $B\subset A$, and any quotient of $A$, we denote by
$\bar B=\{\bar t\mid t\in B\}$ the corresponding subset  in the quotient algebra.  For convenience, $H(h, e)$ will be denoted by $H(e)$ if $h=e$.

By Lemma~\ref{nonzero},  three $\Bbbk$-spaces in the first sequence of \eqref{conmuf1} are zero if  $\uparrow^{r}\downarrow^{s-1}\nsucceq d$  and hence there is nothing to be proved.
 Suppose  $\uparrow^{r}\downarrow^{s-1}\succeq d$ and define
  $$H_1(\uparrow^{r+1}\downarrow^{s})=\left\{{\begin{tikzpicture}[baseline = 4mm, color=\clr]
               \draw (-0.18,0.4) rectangle (0.7,0);
                \draw(0.2, 0.2) node{$h_1$};
                  \draw(0.4, 0.7) node{$g$};
               \draw (0.98,0.55) rectangle (-0.18,0.85); \draw[->,thick] (-0.07, 0.4)to[out=up,in=down](-0.07,0.55); \draw(0.23, 0.48) node{$ \tiny\cdots$};  \draw[->,thick] (0.5,0.4)to[out=up,in=down](0.5,0.55);
                \draw[->,thick] (0.9, -0)to[out=up,in=down](0.9,0.55);
                 \draw (-0.18+1.25,0.4) rectangle (0.7+1.25,0);
                   \draw(-0.3+0.7+1.25, 0.2) node{$h_2$};
                 \draw[<-,thick] (-0.18+1.25+0.8,0.4)to[out=up,in=down](-0.18+1.25+0.8,0.8);
                 \draw[<-,thick] (1.2,0.4)to[out=up,in=down](1.2,0.8);
                \draw(1.5, 0.6) node{$ \tiny\cdots$};  
    \end{tikzpicture}}|(g, h_1, h_2)\in D(\uparrow^{r+1})\times H(\uparrow^r)\times H(\downarrow^s)\right\},$$
  where
     $$D(\uparrow^{r+1})=\left\{\begin{tikzpicture}[baseline = 25pt, scale=0.35, color=\clr]
       \draw[->,thick](0,1.1)to(0,3.9);
       \draw(0.75,1.4) node{$ \cdots$}; \draw(0.75,3.5) node{$ \cdots$};
       \draw[->,thick](1.5,1.1)to(1.5,3.9);
       \draw(2,4.5)node{\tiny$i$}; 
        \draw[->,thick] (5,1) to[out=up, in=down] (2,3.9);
         \draw(4.9,1.4)\bdot;
           \draw(5.4,1.4)node{\tiny$j$};
         \draw[->,thick](2.8,1.1)to(2.8,3.9);
         \draw[->,thick](5.5-1.2,1.1)to(5.5-1.2,3.9);
         \draw(5,0.6)node{\tiny$r+1$};
          \draw(4.8-1.2,1.4) node{$ \cdots$}; \draw(4.8-1.2,3.5) node{$ \cdots$};
           \end{tikzpicture}~|~0\leq j\leq \ell-1, 1\leq i\leq r+1\right\}. $$ Thanks to Lemma~\ref{Aa} and the general result on the basis of degenerate cyclotomic Hecke algebra, it is easy to check that $\bar H_1(\uparrow^{r+1} \downarrow^s)$ is a basis of $\bar A_{\uparrow^{r+1} \downarrow^s}$.   Let $$\bar H_1(\uparrow^{r+1}\downarrow^{s})(\sigma):=\{\bar{g\circ {_{\uparrow^{r+1}\downarrow^s}}\sigma_{\uparrow^r\downarrow^s\uparrow}}~|~g\in H_1(\uparrow^{r+1}\downarrow^s)\}.$$
   By  Lemmas~\ref{nonzero}, ~\ref{cellbasis} and \eqref{isomor},  we have
    \begin{itemize}
    \item [(a)] $ \bar 1_dA_{\preceq \ob b}\otimes _{\bar A_{\ob b}^\circ} ({_\downarrow \bar A}^\circ)\bar 1_a$ has basis $ \bar Y( \uparrow^m\downarrow^n, \uparrow^{r}\downarrow^{s-1})\otimes \bar H(\uparrow^{r}\downarrow^{s})$,
        \item [(b)] $1_d~({_\downarrow A})\otimes _{A}A_{\preceq \ob a}~\bar 1_a$ has basis $ Y(\uparrow^m\downarrow^{n+1}, \uparrow^r\downarrow^{s})\otimes \bar H(\uparrow^r\downarrow^{s})$,
            \item [(c)] $ \bar 1_dA_{\preceq \ob c}\otimes _{ \bar A_{\ob c}^\circ}\bar A_\uparrow^\circ\bar 1_a$ has basis $ \bar Y(\uparrow^{m}\downarrow^{n}, \uparrow^{r+1}\downarrow^{s})\otimes   \bar H_1(\uparrow^{r+1}\downarrow^{s})(\sigma)$.
    \end{itemize}
     Thanks to $\bar {f~\lcap}=0$ and $\bar {f~\rcap}=0$ in $\bar 1_dA_{\preceq \ob c}$ for any $f\in Y( \uparrow^m\downarrow^n, \uparrow^{r}\downarrow^{s-1})$, we have $\psi_{a, d} \varphi_{a, d}=0$ and $\varepsilon_{a, d}\eta_{a, d}=0$.
Since $\{g~~\begin{tikzpicture}[baseline = 10pt, scale=0.5, color=\clr]
                \draw[<-,thick] (0,0.5)to[out=up,in=down](0,1.2);
    \end{tikzpicture}~ | g\in Y( \uparrow^m\downarrow^n, \uparrow^{r}\downarrow^{s-1})\}\subseteq Y( \uparrow^m\downarrow^{n+1}, \uparrow^{r}\downarrow^{s})$, $\varphi_{a, d}$ is injective. Define
    $$ Y_1( \uparrow^m\downarrow^{n+1}, \uparrow^{r}\downarrow^{s})=\left\{{ \begin{tikzpicture}[baseline = 1mm, color=\clr]
                \draw[-,thick,darkblue] (0.5,0) to[out=down,in=left] (1.25,-0.5) to[out=right,in=down] (2,0);\draw (-0.18,0.4) rectangle (0.7,0);
                \draw(0.2, 0.2) node{$g$}; \draw[<-,thick] (2,0.01)to[out=up,in=down](2,0.85);
                \draw (1.8,0.55) rectangle (-0.18,0.85); \draw[->,thick] (-0.07, 0.4)to[out=up,in=down](-0.07,0.55); \draw(0.23, 0.48) node{$ \tiny\cdots$};  \draw[->,thick] (0.5,0.4)to[out=up,in=down](0.5,0.55);
                \draw[<-,thick] (0.9, -0.7)to[out=up,in=down](0.9,0.55); \draw(1.25, 0.2) node{$ \tiny\cdots$};  \draw[<-,thick] (1.5,-0.7)to[out=up,in=down](1.5,0.55); \draw(0.9, 0.7) node{$f$};
    \end{tikzpicture}}~|~(f, g)\in Y(\uparrow^{m}\downarrow^{n}, \uparrow^{r+1}\downarrow^{s})\times D(\uparrow^{r+1})\right\}.$$
 So $Y_1( \uparrow^m\downarrow^{n+1}, \uparrow^{r}\downarrow^{s})\otimes \bar H(\uparrow^r\downarrow^s)\subseteq 1_d~({_\downarrow A})\otimes _{A}A_{\preceq \ob a}~\bar 1_a$. If $f\in Y(\uparrow^{m}\downarrow^{n}, \uparrow^{r+1}\downarrow^{s})$, $g\in D(\uparrow^{r+1})$ and $h\in  H(\uparrow^r\downarrow^s)$, then  $$ \begin{aligned}\psi_{a, d}&\left(\begin{tikzpicture}[baseline = 1mm, color=\clr]
                \draw[-,thick,darkblue] (0.5,0) to[out=down,in=left] (1.25,-0.5) to[out=right,in=down] (2,0);\draw (-0.18,0.4) rectangle (0.7,0);
                \draw(0.2, 0.2) node{$g$}; \draw[<-,thick] (2,0.01)to[out=up,in=down](2,0.85);
                \draw (1.8,0.55) rectangle (-0.18,0.85); \draw[->,thick] (-0.07, 0.4)to[out=up,in=down](-0.07,0.55); \draw(0.23, 0.48) node{$ \tiny\cdots$};  \draw[->,thick] (0.5,0.4)to[out=up,in=down](0.5,0.55);
                \draw[<-,thick] (0.9, -0.7)to[out=up,in=down](0.9,0.55); \draw(1.25, 0.2) node{$ \tiny\cdots$};  \draw[<-,thick] (1.5,-0.7)to[out=up,in=down](1.5,0.55); \draw(0.9, 0.7) node{$f$};
    \end{tikzpicture}\otimes \bar h\right)=\psi_{a, d}\left(\begin{tikzpicture}[baseline = 1mm, color=\clr]
                \draw[-,thick,darkblue] (0.5,0) to[out=down,in=left] (1.25,-0.5) to[out=right,in=down] (2,0);\draw (-0.18,0.4) rectangle (0.7,0);
                \draw(0.2, 0.2) node{$g_1$}; \draw[<-,thick] (2,0.01)to[out=up,in=down](2,0.85);
                \draw (1.8,0.55) rectangle (-0.18,0.85); \draw[->,thick] (-0.07, 0.4)to[out=up,in=down](-0.07,0.55); \draw(0.23, 0.48) node{$ \tiny\cdots$};  \draw[->,thick] (0.5,0.4)to[out=up,in=down](0.5,0.55);
                \draw[<-,thick] (0.9, -0.7)to[out=up,in=down](0.9,0.55); \draw(1.25, 0.2) node{$ \tiny\cdots$};  \draw[<-,thick] (1.5,-0.7)to[out=up,in=down](1.5,0.55);
                \draw(2, 0.5) \bdot;
                \draw(2.2, 0.4)  node{$j$};
               \draw(0.9, 0.7) node{$f$};
    \end{tikzpicture}\otimes \bar h\right)\qquad\text{by Lemma~\ref{dots}(2),}\\&=
   \bar{ \begin{tikzpicture}[baseline = 1mm, color=\clr]
                \draw[-,thick,darkblue] (0.5,0) to[out=down,in=left] (1.25,-0.5) to[out=right,in=down] (2,0);\draw (-0.18,0.4) rectangle (0.7,0);
                \draw(0.2, 0.2) node{$g_1$}; \draw[-,thick] (2,0.01)to[out=up,in=down](2,0.85);
                \draw (1.8,0.55) rectangle (-0.18,0.85); \draw[->,thick] (-0.07, 0.4)to[out=up,in=down](-0.07,0.55); \draw(0.23, 0.48) node{$ \tiny\cdots$};  \draw[->,thick] (0.5,0.4)to[out=up,in=down](0.5,0.55);
                \draw[<-,thick] (0.9, -0.7)to[out=up,in=down](0.9,0.55); \draw(1.25, 0.2) node{$ \tiny\cdots$};  \draw[<-,thick] (1.5,-0.7)to[out=up,in=down](1.5,0.55);
                \draw(2, 0.5) \bdot;

                \draw[-,thick,darkblue] (0.5+1.5,0.5+0.35) to[out=up,in=left] (0.78+1.5,0.78+0.35) to[out=right,in=up] (1.06+1.5,0.5+0.35);
                \draw[->,thick] (1.06+1.5,0)to[out=up,in=down](1.06+1.5,0.5+0.35);

                \draw(2.2, 0.4)  node{$j$};
               \draw(0.9, 0.7) node{$f$};
    \end{tikzpicture}}\otimes \bar {h~\sli}\\&=
   \bar{ \begin{tikzpicture}[baseline = 1mm, color=\clr]
               \draw (-0.18,0.4) rectangle (0.7,0);
                \draw(0.2, 0.2) node{$g$};
                \draw (1.8,0.55) rectangle (-0.18,0.85); \draw[->,thick] (-0.07, 0.4)to[out=up,in=down](-0.07,0.55); \draw(0.23, 0.48) node{$ \tiny\cdots$};  \draw[->,thick] (0.5,0.4)to[out=up,in=down](0.5,0.55);
                \draw[<-,thick] (0.9, -0.2)to[out=up,in=down](0.9,0.55); \draw(1.25, 0.2) node{$ \tiny\cdots$};  \draw[<-,thick] (1.5,-0.2)to[out=up,in=down](1.5,0.55); \draw(0.9, 0.7) node{$f$};
    \end{tikzpicture}}\otimes \bar {{_{\uparrow^{r+1}\downarrow^s}}\sigma_{\uparrow^{r}\downarrow^s\uparrow} \circ( h~\sli)}\qquad\text{by Lemma~\ref{dots}(2) and \eqref{relation 1},}\\&=
   \bar{ \begin{tikzpicture}[baseline = 1mm, color=\clr]
               \draw (-0.18,0.4) rectangle (0.7,0);
                \draw(0.2, 0.2) node{$g$};
                \draw (1.8,0.55) rectangle (-0.18,0.85); \draw[->,thick] (-0.07, 0.4)to[out=up,in=down](-0.07,0.55); \draw(0.23, 0.48) node{$ \tiny\cdots$};  \draw[->,thick] (0.5,0.4)to[out=up,in=down](0.5,0.55);
                \draw[<-,thick] (0.9, -0.2)to[out=up,in=down](0.9,0.55); \draw(1.25, 0.2) node{$ \tiny\cdots$};  \draw[<-,thick] (1.5,-0.2)to[out=up,in=down](1.5,0.55); \draw(0.9, 0.7) node{$f$};
    \end{tikzpicture}}\otimes \bar {( h_1~\sli~h_2)\circ {_{\uparrow^{r+1}\downarrow^s}}\sigma_{\uparrow^r\downarrow^s\uparrow}}
    \\&=\bar f\otimes\bar{\left(
    \begin{tikzpicture}[baseline = 4mm, color=\clr]
               \draw (-0.18,0.4) rectangle (0.7,0);
                \draw(0.2, 0.2) node{$h_1$};
                  \draw(0.4, 0.7) node{$g$};
               \draw (0.98,0.55) rectangle (-0.18,0.85); \draw[->,thick] (-0.07, 0.4)to[out=up,in=down](-0.07,0.55); \draw(0.23, 0.48) node{$ \tiny\cdots$};  \draw[->,thick] (0.5,0.4)to[out=up,in=down](0.5,0.55);
                \draw[->,thick] (0.9, -0)to[out=up,in=down](0.9,0.55);
                 \draw (-0.18+1.25,0.4) rectangle (0.7+1.25,0);
                   \draw(-0.3+0.7+1.25, 0.2) node{$h_2$};
                 \draw[<-,thick] (-0.18+1.25+0.8,0.4)to[out=up,in=down](-0.18+1.25+0.8,0.8);
                 \draw[<-,thick] (1.2,0.4)to[out=up,in=down](1.2,0.8);
                \draw(1.5, 0.6) node{$ \tiny\cdots$};  
    \end{tikzpicture}\right)\circ
     {_{\uparrow^{r+1}\downarrow^s}\sigma_{\uparrow^r\downarrow^s\uparrow}}
     }\in \bar Y(\uparrow^{m}\downarrow^{n}, \uparrow^{r+1}\downarrow^{s})\otimes \bar H_1(\uparrow^{r+1}\downarrow^{s})(\sigma),
    \end{aligned}$$ where  $g_1$ is obtained from $g$ by removing all dots on $g$, and $(h_1, h_2)\in H(\uparrow^r)\times H(\downarrow^s)$ such that $h=h_1\otimes h_2$.
     Thanks to (c),
     $\psi_{a,d}$ is surjective.
   Now, the exactness of the  first   sequence in \eqref{conmuf1} follows immediately since
   $$\begin{aligned}\text{dim}\bar 1_dA_{\preceq \ob b}\otimes _{\bar A_{\ob b}^\circ} {_\downarrow \bar A}^\circ\bar 1_a
 &= |Y( \uparrow^m\downarrow^{n+1}, \uparrow^{r}\downarrow^{s})| |H(\uparrow^{r}\downarrow^{s})|\\&=\left (|Y( \uparrow^m\downarrow^{n}, \uparrow^{r}\downarrow^{s-1})|+|Y_1( \uparrow^m\downarrow^{n+1}, \uparrow^{r}\downarrow^{s})|\right )
 |H(\uparrow^{r}\downarrow^{s})|\\&=|Y( \uparrow^m\downarrow^{n}, \uparrow^{r}\downarrow^{s-1})| |H(\uparrow^{r}\downarrow^{s})|+|Y( \uparrow^m\downarrow^{n}, \uparrow^{r+1}\downarrow^{s})||D(\uparrow^{r+1})| |H(\uparrow^{r}\downarrow^{s})|\\&
  = \text{dim}1_d ({_\downarrow A}\otimes _{A}A_{\preceq \ob a})\bar 1_a+\text{dim}\bar 1_dA_{\preceq \ob c}\otimes _{ \bar A_{\ob c}^\circ}\bar A_\uparrow^\circ\bar 1_a.\end{aligned}$$
   \end{proof}

\begin{Lemma}\label{msikdjde1}  Suppose  $\varphi$, $\psi$, $\eta$ and $\varepsilon$ are $(A,\bar A^\circ)$-homomorphisms    in \eqref{shoretmor1}. Then
\begin{multicols}{2}
\item[(1)] $( x^\downarrow\otimes \text{Id})\circ \varphi=\varphi \circ (\text{Id}\otimes x_L^\downarrow )$,\item[(2)]$(\text{Id}\otimes  x_R^\uparrow )\circ \psi=\psi\circ ( x^\downarrow\otimes \text{Id})$,
 \item [(3)] $( x^\uparrow\otimes \text{Id})\circ \eta=\eta \circ (\text{Id}\otimes x_L^\uparrow )$,\item[(4)]$(\text{Id}\otimes  x_R^\downarrow )\circ \varepsilon=\varepsilon\circ ( x^\uparrow\otimes \text{Id})$.
    \end{multicols}
\end{Lemma}
\begin{proof}
Suppose $ (\bar f, \bar g)\in A_{\preceq \ob b}\bar 1_b\times \bar 1_{b \downarrow} { \bar A}^\circ\bar 1_\ob a$, where  $\ob a=(s,r)$, $b\in\ob b=(s-1, r)$ and $s\geq 1$. Thanks to Lemma~\ref{jme}(1),  $$( x^\downarrow\otimes \text{Id})\circ \varphi(\bar f\otimes \bar g)=( x^\downarrow\otimes \text{Id})(( f~\begin{tikzpicture}[baseline = 10pt, scale=0.5, color=\clr]
                \draw[<-,thick] (0,0.5)to[out=up,in=down](0,1.2);
    \end{tikzpicture})\otimes \bar g)=(f~\xdx )\otimes \bar g=f~\begin{tikzpicture}[baseline = 10pt, scale=0.5, color=\clr]
                \draw[<-,thick] (0,0.5)to[out=up,in=down](0,1.2);
    \end{tikzpicture}\otimes  \bar {1_b~\xdx}\circ \bar g =\varphi \circ (\text{Id}\otimes x_L^\downarrow )(\bar f\otimes \bar g)
 ,$$ proving (1). One can check  (3) similarly. If $ (f, \bar g)\in 1_c {_\downarrow A}\otimes _{A}A_{\preceq \ob a}\bar 1_\ob a$ and $c\in J$, then
    $$ \begin{aligned}
    &(\text{Id}\otimes  x_R^\uparrow  )\circ \psi(f\otimes \bar g)=(\text{Id}\otimes  x_R^\uparrow  )( \bar{\begin{tikzpicture}[baseline = 1mm, color=\clr]
                \draw[-,thick] (-0.06,0.5)to[out=up,in=down](-0.06,0.8); \draw(0.2, 0.6) node{$ \tiny\cdots$};  \draw[-,thick] (0.4,0.5)to[out=up,in=down](0.4,.8);  \draw[<-,thick,darkblue] (0.5,0.5) to[out=up,in=left] (0.78,0.78) to[out=right,in=up] (1.06,0.5);\draw (-0.1,0.5) rectangle (0.9,0);
                \draw(0.4, 0.25) node{$ f$}; \draw[-,thick] (1.06,0.5)to[out=up,in=down](1.06,0.01);
    \end{tikzpicture}}\otimes (  \bar{g~\begin{tikzpicture}[baseline = 10pt, scale=0.5, color=\clr]
                \draw[->,thick] (0,0.5)to[out=up,in=down](0,1.2);
    \end{tikzpicture}~}))
    =  \bar{\begin{tikzpicture}[baseline = 1mm, color=\clr]
                \draw[-,thick] (-0.06,0.5)to[out=up,in=down](-0.06,0.8); \draw(0.2, 0.6) node{$ \tiny\cdots$};  \draw[-,thick] (0.4,0.5)to[out=up,in=down](0.4,.8);  \draw[<-,thick,darkblue] (0.5,0.5) to[out=up,in=left] (0.78,0.78) to[out=right,in=up] (1.06,0.5);\draw (1.06,0.3) {\bdot};\draw (-0.1,0.5) rectangle (0.9,0);
                \draw(0.4, 0.25) node{$ f$}; \draw[-,thick] (1.06,0.5)to[out=up,in=down](1.06,0.01);
    \end{tikzpicture}}\otimes (  \bar{g~\begin{tikzpicture}[baseline = 10pt, scale=0.5, color=\clr]
                \draw[->,thick] (0,0.5)to[out=up,in=down](0,1.2);
    \end{tikzpicture}~})
    \overset{\text{Lemma~\ref{dots}(6)}}= \bar{\begin{tikzpicture}[baseline = 1mm, color=\clr]
                \draw[-,thick] (-0.06,0.5)to[out=up,in=down](-0.06,0.8); \draw(0.2, 0.6) node{$ \tiny\cdots$};  \draw[-,thick] (0.4,0.5)to[out=up,in=down](0.4,.8);  \draw[<-,thick,darkblue] (0.5,0.5) to[out=up,in=left] (0.78,0.78) to[out=right,in=up] (1.06,0.5);\draw (-0.1,0.5) rectangle (0.9,0);
                \draw (0.55,0.61) node{$ \bullet$};
                \draw(0.4, 0.25) node{$ f$}; \draw[-,thick] (1.06,0.5)to[out=up,in=down](1.06,0.01);
    \end{tikzpicture}}\otimes (  \bar{g~\begin{tikzpicture}[baseline = 10pt, scale=0.5, color=\clr]
                \draw[->,thick] (0,0.5)to[out=up,in=down](0,1.2);
    \end{tikzpicture}~})\\
    = &\psi\circ ( x^\downarrow\otimes \text{Id}) (f\otimes  \bar g),
    \end{aligned} $$
   proving (2). Replacing  Lemma~\ref{dots}(6) by  Lemma~\ref{dots}(8), one can verify (4) by arguments similar to those for (2).
\end{proof}

\begin{Theorem}\label{usuactifuc1} For each $i\in \Bbbk$,
   there are two short exact sequences of functors  from $\bar A^\circ$-fdmod to $A$-lfdmod:
 \begin{equation}\label{keyses}\begin{aligned}& 0\rightarrow \Delta\circ F_i^\downarrow\rightarrow  F_i\circ \Delta\rightarrow\Delta\circ   F_{i}^\uparrow\rightarrow0,\\&0\rightarrow \Delta\circ E_i^\uparrow\rightarrow  E_i\circ \Delta\rightarrow\Delta\circ   E_{i}^\downarrow\rightarrow0.
 \end{aligned}\end{equation}
\end{Theorem}

\begin{proof} Thanks to  Lemma~\ref{msikdjde1}, the short exact sequences in \eqref{keyses} follow from  those in \eqref{diejdhd1} by
 passing to appropriate generalized eigenspaces.\end{proof}

\subsection{Characters}Suppose $a=a_1\cdots a_{r+s}\in \ob a=(r, s)$.  For any $i$, $1\leq i\leq r+s$, define
$$X_i1_a=\begin{cases} 1_{a_1\cdots a_{i-1}}~\xd~1_{a_{i+1}\cdots a_{r+s}}, & \text{if $a_i=\sli$~,}\\
1_{a_1\cdots a_{i-1}}~\xdx~1_{a_{i+1}\cdots a_{r+s}}, & \text{if $a_i=\xli$~.}\\
   \end{cases}
   $$
   Then     $\{ X_i1_{a}\mid 1\le i\le {r+s}, a\in \ob a\}$ generates a finite dimensional  commutative subalgebra of  $ A_\ob a$.
 For any $\mathbf i=(i_1,i_2,\ldots, i_{r+s})\in \Bbbk^{r+s}$, there is an idempotent $1_{a;\mathbf i}\in  A_\ob a$ which projects any $M\in  A_\ob a$-fdmod
onto $M_{a;\mathbf i}$,  the simultaneous generalized eigenspace of $ X_11_{ a} ,\ldots, X_{ r+s}1_{ a}$ with respect to $\mathbf i$.
When $r=s=0$, $\bar A_\ob a\cong \Bbbk$. In this case, there is a unique idempotent $1_{\emptyset; \emptyset}$.
\begin{Defn} For any $V\in A$-lfdmod, define
\begin{equation}
\text{ch}V=\sum_{a\in J, \mathbf i\in \Bbbk^{\ell(a)}}(\dim 1_{a;\mathbf i}V)e^{a;\mathbf i},
\end{equation}
where $e^{a;\mathbf i}$ are formal  symbols, $\ell(a)=\ell_\uparrow(a)+\ell_\downarrow(a)$, and $1_{a; \mathbf i}V$ is the   simultaneous generalized eigenspace of $ X_11_{ a} ,\ldots, X_{\ell(a)} 1_{ a}$ corresponding to $\mathbf i$.
\end{Defn}

 Suppose $\ob a=(r, s)$ and $\lambda=(\lambda^\downarrow, \lambda^\uparrow)\in \Lambda_\ob a$, where  $\Lambda_\ob a$ is given in \eqref{bipar}. Motivated by
  Lemma~\ref{Aa} and \eqref{contss}-\eqref{word}, We define  the content $c(x)$  of a node $x$  with respect to  $\lambda$  as follow:
\begin{equation}\label{colr}  c(x)=\begin{cases} c_\uparrow(x), & \text{if $x$ is in $[\lambda^\uparrow]$,}\\
     c_\downarrow(x),   & \text{if $x$ is in $[\lambda^\downarrow]$,} \end{cases} \end{equation} where
$c_\uparrow(x)=u_j+k-l$ (resp., $c_\downarrow(x)=u_j'-k+l$)  if $x$ is at the $l$th row and $k$th column of   the $j$th component of the Young diagram $[\lambda^\uparrow]$ (resp., $[\lambda^\downarrow]$).
Recall the 
cell modules $S(\lambda)$'s of $\bar A_\ob a$ in subsection~\ref{irrddjs}.  The following result follows from Lemma~\ref{Aa}, \eqref{word} and  the branching rules  of the cell modules for degenerate cyclotomic Hecke algebras.

\begin{Lemma}\label{bran}Suppose $i\in\Bbbk$ and $\lambda=(\lambda^\downarrow, \lambda^\uparrow)\in \Lambda_{(r, s)} $, where $r, s\in \mathbb N$. We have
\begin{itemize}
\item [(1)] $F_i^\downarrow S(\lambda)$ (resp., $E_i^\uparrow  S(\lambda)$) has a multiplicity-free filtration with sections $S(\mu)$, where  $ \mu\in\Lambda_{(r-1, s)}$ (resp., $\mu\in\Lambda_{(r, s-1)}$)
is obtained by removing  a box in $[\lambda^\downarrow]$ (resp., $[\lambda^\uparrow]$) of content $i$.
\item [(2)]$ E_i^\downarrow  S(\lambda)$ (resp., $F_i^\uparrow S(\lambda)$) has a multiplicity-free filtration with sections $S(\mu)$, where  $ \mu\in\Lambda_{(r+1, s)}$ (resp., $\mu\in\Lambda_{(r, s+1)}$)
is obtained by adding  a box in $[\lambda^\downarrow]$ (resp., $[\lambda^\uparrow]$) of content $i$.
\end{itemize}
\end{Lemma}
Recall $\mathbb I, \mathbb I_\mathbf u$ and $\mathbb I_{\mathbf u'}$ in \eqref{iu}.
By Lemma~\ref{bran}, $E_i^\uparrow$ and $F_i^\uparrow$ (resp., $E_i^\downarrow$ and $F_i^\downarrow$) are non-zero only if $i\in\mathbb I_\mathbf u$ (resp., $i\in \mathbb I_{\mathbf u'}$).
Thanks to  Theorem~\ref{usuactifuc1} and  Lemma~\ref{Ei}, \begin{equation}\label{EFi}
E=\bigoplus_{i\in \mathbb I} E_i, \quad \quad     F=\bigoplus_{i\in \mathbb I} F_i.
\end{equation}

Recall that $\Lambda=\bigcup_{\ob a\in I}\Lambda_\ob a$ in \eqref{lambdada}. Let  $\Xi$ be  the graph such that the set of  vertices is  $\Lambda$ and  any edge is of form   $\lambda$---$\mu$
whenever $\mu$ is obtained from $\lambda$ by either   adding   a box or removing a box.
A path $\gamma:\lambda \rightsquigarrow\mu$ in $\Xi$ is a finite sequence of vertices $\lambda=\lambda_0, \lambda_1,\ldots,\lambda_m=\mu$ with each $\lambda_{j}$---$\lambda_{j+1}$ connected
by an edge.  We color the edge  $\lambda_{j}$---$\lambda_{j+1}$ by $c(x)$ if
$\lambda_{j+1}$ is obtained  from $\lambda_i$ by either   adding an addable node $x$  or  removing a removable  node $x$.

 \begin{Defn}\label{typer}Suppose $\mathbf i=(i_1,\ldots, i_m)$. A path   $\gamma$ is of type  $(a; \mathbf i)$ if  $i_j$ is the color of the $j$th edge of $\gamma$, and if $a=a_1\cdots a_m \in J$  such that
 \begin{itemize}
\item [(1)] $a_{j+1}=\uparrow$ if $\lambda_{j+1}$ is obtained from $\lambda_j$ by adding a box to $\lambda_j^\uparrow$ or removing a box from $\lambda_j^\downarrow$,
\item [(2)] $a_{j+1}=\downarrow$ if  $\lambda_{j+1}$ is obtained from $\lambda_j$ by adding a box to $\lambda_j^\downarrow$   or removing a box from  $\lambda_j^\uparrow$.
 \end{itemize}
 \end{Defn}
When $\lambda=\mu=(\emptyset, \emptyset)$, we say there is a unique path from $\lambda$ to $\mu$ with type $(\emptyset; \emptyset)$.

\begin{Prop}\label{charcter1}
For any $\lambda\in\Lambda$,
$\text{ch}\tilde\Delta(\lambda)=\sum_{\gamma}e^{\text{type}(\gamma)}$
where the summation ranges over all paths $\gamma: (\emptyset, \emptyset) \rightsquigarrow \lambda$ and $\tilde\Delta(\lambda)$ is given in Definition~\ref{stanpro}.
\end{Prop}
\begin{proof}Suppose $\lambda\in\Lambda_\ob a$ and $\ob a\in I$. Thanks to  \eqref{exa}, $\tilde\Delta(\lambda)=\bigoplus_{\ob d\succeq \ob a}1_{ {\ob d}}\tilde\Delta(\lambda)$ and  $ 1_{\ob a} \tilde\Delta(\lambda)= S(\lambda)$.
In order to prove the required formula on $\text{ch}\tilde\Delta(\lambda)$, it
suffices to show that \begin{equation}\label{kkkey} \dim 1_{a;\mathbf i}\tilde \Delta(\lambda)=| \{\gamma: (\emptyset, \emptyset) \rightsquigarrow\lambda\mid \gamma\text{ is of type }(a;\mathbf i)\}|\end{equation}  for all $\mathbf  i\in  \Bbbk^m$ and all $a=a_1\cdots a_m\in J$ and all $m\in \mathbb N$.
We prove \eqref{kkkey}  by induction on $m\in \mathbb N $.

Suppose that  $m=0$. If $\lambda\neq(\emptyset, \emptyset)$, then $1_{\emptyset;\emptyset} \tilde\Delta(\lambda)=0$. Otherwise,   $1_{\emptyset; \emptyset} \tilde\Delta(\lambda)=\Bbbk$. So, the result holds for $m=0$. In general,  by \eqref{EF}, $$ 1_a(EV)=1_{a} {_{\uparrow}A}\otimes_A V \cong 1_{a\uparrow}V, \quad 1_a(FV)= 1_a {_{\downarrow} A}\otimes_A V \cong 1_{a\downarrow}V$$
for  any $V\in A$-lfdmod. So, \begin{equation}\label{cdcedccde1}\dim 1_{a; \mathbf i} E_i V=\dim 1_{a \uparrow; \mathbf i i} V,\qquad \dim 1_{a; \mathbf i} F_i V=\dim 1_{a \downarrow; \mathbf i i} V.\end{equation}
Thanks to \eqref{keyses} and Lemma~\ref{bran},  $ E_{i}\tilde \Delta(\lambda)$ (resp., $ F_{i}\tilde \Delta(\lambda)$) has a multiplicity-free
$\tilde\Delta$-filtration such that $\tilde\Delta(\mu)$ appears as a section if and only if  $ \mu$
is obtained by either removing  a box in $[\lambda^\uparrow]$ (resp., $[\lambda^\downarrow]$) of content $i$ or adding  a box in $[\lambda^\downarrow]$ (resp., $[\lambda^\uparrow]$) of content $i$.
Now \eqref{kkkey} follows from \eqref{cdcedccde1} and induction on $m$, immediately.
\end{proof}

\begin{Cor}\label{twopath1}
Suppose  $(\lambda, \mu)\in\Lambda\times \bar \Lambda_\ob a$, $\ob a\in I$.  If $[\tilde \Delta(\lambda):L(\mu)]\neq0$, then there are two  paths $\gamma: (\emptyset, \emptyset)  \rightsquigarrow \lambda$
and $\delta: (\emptyset, \emptyset)\rightsquigarrow\mu$  such that  $\gamma$ and $\delta$ are of the same type $(a; \mathbf i)$ and $a\in \ob a$.
\end{Cor}
\begin{proof}Mimicking arguments in the proof of \cite[Corollary~5.11]{GRS3}, one can verify this result by using
Corollary~\ref{irr}(1) and  Proposition~\ref{charcter1}, immediately.
\end{proof}

\begin{Theorem}\label{sjdsdh}
Suppose  $\mathbb I_{\mathbf u}\bigcap \mathbb I_{\mathbf u'}=\emptyset$, where  $\mathbb I_{\mathbf u}$ and $\mathbb I_{\mathbf u'}$ are in  ~\eqref{iu}. Then
 \begin{itemize}\item [(1)] $\mathbf P(\mu)=\Delta(\mu)$ for all  $\mu\in \bar\Lambda$,
\item [(2)]  $\Delta: \bar A^\circ\text{-mod}\rightarrow A\text{-mod} $ is an equivalence of categories. \end{itemize}
\end{Theorem}
\begin{proof} Take an arbitrary $\mu\in \bar \Lambda_\ob a$. If  $\mathbf P(\mu)\neq \Delta(\mu)$, by  Corollary~\ref{ijxxexeu}(1), $[\bar\Delta(\lambda):L(\mu)]\neq 0$ for some $ \lambda\in \bar\Lambda_\ob b$ and $\lambda\neq \mu$.
Since $\Delta$ is exact and $D(\lambda)$ is the simple head of $S(\lambda)$,   there is an epimorphism from $\tilde \Delta(\lambda)$ to $\bar \Delta(\lambda)$. So,  $[\tilde\Delta(\lambda):L(\mu)]\neq 0$.
By Corollary~\ref{twopath1}, there are    two paths   $\gamma:(\emptyset, \emptyset) \rightsquigarrow \lambda$  and $\delta: (\emptyset, \emptyset)\rightsquigarrow\mu$ such that $\gamma$ and $\delta$ are of the same type $(a; \mathbf i)$ and $a\in \ob a$.

We claim     $\ob a=\ob b$.  Otherwise, $\ob b\succ \ob a$. By Definition~\ref{typer},   there is an edge, say $\lambda_{j-1}$---$\lambda_j$ such that
$\lambda_j$ is obtained by removing a box $x$ either in  $\lambda_{j-1}^\downarrow$   with $c(x)\in \mathbb I_{\mathbf u}$  or
in  $\lambda_{j-1}^\uparrow $ with   $c(x)\in \mathbb I_{\mathbf u'}$. In  any case,
  $c(x)\in  \mathbb I_{\mathbf u}\bigcap \mathbb I_{\mathbf u'}\neq\emptyset$, a contradiction.
By the definition of $\Delta$ in \eqref{exa}, $1_{ \ob a}\bar\Delta(\lambda)=D(\lambda)$. Since $1_\ob a L(\mu)=D(\mu)$, it is a composition factor of the simple  $\bar A_\ob a$-module $D(\lambda)$, forcing $\lambda=\mu$, a contradiction.  So, $\mathbf P(\mu)=\Delta(\mu)$ for all  $\mu\in \bar\Lambda$.
Now (2) immediately follows from  \cite[Corollary~2.5]{BD} and (1),
    since the exact functor $\Delta$ sends  projective $\bar A^\circ$-modules $P(\lambda)$'s  to  projective $A$-modules $\mathbf P(\lambda)$'s for any $\lambda\in \bar\Lambda$.
\end{proof}

\begin{Cor}\label{ssimple} The category  $A$-mod is completely reducible if
  \begin{itemize} \item[(1)] $u_i- u_j'\not \in \mathbb Z1_\Bbbk$ for all $1\le i, j\le \ell$,
   \item [(2)] $u_i-u_j\not\in \mathbb Z 1_\Bbbk$, $u_i'-u_j'\not\in \mathbb Z 1_\Bbbk$  for all $1\le i<j\le\ell$  and $p=0$.
       \end{itemize}\end{Cor}
       \begin{proof} Thanks to (1) and  Theorem~\ref{sjdsdh}(2),  $ \bar A^\circ\text{-mod}$ is Morita equivalent to $ A\text{-mod} $. By (2) and \cite[Theorem~6.11]{AMR}, both $ H_{\ell,r}(\mathbf u)$ and $ H_{\ell, s}(-\mathbf u')$ are semisimple for all $r, s\in \mathbb N$. Now, the result follows immediately from   Lemma~\ref{isomhecke}(2) and
       Lemma~\ref{Aa}.
       \end{proof}
We expect that Corollary~\ref{ssimple}(1)-(2) are necessary and sufficient conditions
for $A$-mod being completely reducible.
\subsection{Categorical actions }\label{CA} Let $\mathfrak g$ be   the complex Kac-Moody Lie algebra $\mathfrak g$ associated to  Cartan matrix $(a_{i,j})_{i,j\in \mathbb I}$   defined by \eqref{lie}.
 Then $\mathfrak g$ is the Lie algebra generated by its Cartan subalgebra and  Chevalley generators $\{e_i,f_i\mid i\in\mathbb I\}$ subject to the usual Serre relations. Furthermore,  $\mathfrak g$ is isomorphic to  a direct sum of certain $\mathfrak {sl}_\infty$ (resp., $\hat{\mathfrak {sl}}_p$) if $p=0$ (resp., $p>0$) depending on both $\mathbf u $ and $\mathbf u'$.
Let \begin{equation}\label{Pi} \Pi=\{\alpha_i\mid i\in \mathbb I\}, \end{equation} the set of simple roots. The weight lattice  is
\begin{equation}\label{wei}P:= \{\lambda\in \mathfrak h^*\mid \langle h_i,\lambda\rangle\in \mathbb Z \text{ for all } i\in\mathbb I\}, \end{equation}
where $h_i:=[e_i,f_i]$. Let
 \begin{equation}\label{wei1}
 P^+=\{\lambda\in \mathfrak h^*\mid \langle h_i,\lambda\rangle\in \mathbb N \text{ for all } i\in\mathbb I\},\end{equation}
 and   $\{\omega_i\mid i\in\mathbb I\}$ be the set of  fundamental weights of $\mathfrak g$.
There is a  usual dominance order on $P$ in the sense that $\lambda\leq \mu$ if   $\mu-\lambda\in \sum_{i\in \mathbb I}\mathbb N \alpha_i$.
  The partial order on $P$ induces a partial order on $P\times P$
such that \begin{equation}\label{preceq1} (\lambda_1, \lambda_2)\succeq (\mu_1, \mu_2)\ \  \text{if  $\lambda_1+\lambda_2=\mu_1+\mu_2$ and $\lambda_1\leq \mu_1$.}\end{equation}
 Define
\begin{equation}\label{nsjwhfun}
\text{wt}: \Lambda  \rightarrow P\times P, \quad
   \lambda  \rightarrow (\text{wt}_\downarrow(\lambda) , \text{wt}_\uparrow(\lambda  ))
   \end{equation}
   where
$   \text{wt}_\downarrow(\lambda )=-\omega_{\mathbf u'}+\sum_{y\in[\lambda^\downarrow]}\alpha_{c (y)}$,  $   \text{wt}_\uparrow(\lambda )=\omega_\mathbf u-\sum_{x\in[\lambda^\uparrow]}\alpha_{c(x)}$, and
   $\omega_\mathbf u,  \omega_{\mathbf u'}$ are given in \eqref{param}.
\begin{Prop}\label{dedji}
Suppose  $\lambda\in\Lambda$ and $\mu\in\bar \Lambda$. If
$[\tilde \Delta(\lambda):L(\mu)] \neq 0$, then   $\text{wt}(\mu) \preceq \text{wt}(\lambda)$.
\end{Prop}

\begin{proof}
Suppose  $\mu\in\bar\Lambda_{(r, s)}$. Thanks to  Corollary ~\ref{twopath1}, there are two  paths $\gamma: (\emptyset, \emptyset)\rightsquigarrow \lambda$
and $\delta: (\emptyset, \emptyset)\rightsquigarrow\mu$  such that  $\gamma$ and $\delta$ are of the same type $(a;\mathbf i)$, where $a=a_1\cdots a_{r+s}$ and $\mathbf i=(i_1,\ldots, i_{r+s})$.
 We are going to prove the result  by induction on $r+s$.

 If $r+s=0$, then $\lambda=\mu=(\emptyset, \emptyset)$ and there is nothing to prove.  Otherwise,  $r+s>0$. Removing the last edge in  both $\gamma$  and $\delta$ yields two shorter paths $\gamma': (\emptyset, \emptyset)\rightsquigarrow \lambda'$
and $\delta': (\emptyset, \emptyset)\rightsquigarrow \mu'$  such that  $\gamma'$ and $\delta'$ are of the same type $(b; \mathbf j)$, where $b=a_1\cdots a_{r+s-1}$ and $\mathbf j=(i_1,\ldots, i_{r+s-1})$.
We deal with the case  $a_{r+s}=\uparrow$ and leave the details on $a_{r+s}=\downarrow$ to the reader since the proof is similar.

There are two cases we need to consider.
If $\lambda$ is obtained from $\lambda'$ by adding a box $x$ in $[\lambda'^\uparrow]$, then  $\text{wt}_\uparrow(\lambda')=\text{wt}_\uparrow(\lambda)+\alpha_{c(x)}$ and $\text{wt}_\downarrow(\lambda')=\text{wt}_\downarrow(\lambda)$.
 Since $\mu\in \Lambda_{(r, s)}$, \begin{equation}\label{weiw}
\text{wt}_\uparrow(\mu')=\text{wt}_\uparrow(\mu)+\alpha_{c(x)},\qquad \text{wt}_\downarrow(\mu')=\text{wt}_\downarrow(\mu).
\end{equation}
By induction assumption, we have  $\text{wt}(\mu') \preceq \text{wt}(\lambda')$, forcing   $\text{wt}(\mu) \preceq \text{wt}(\lambda)$. Otherwise,  $\lambda$ is obtained from $\lambda'$ by removing a box $x$ in $[\lambda'^\downarrow]$. So,  $\text{wt}_\uparrow(\lambda')=\text{wt}_\uparrow(\lambda)$,  $\text{wt}_\downarrow(\lambda')=\text{wt}_\downarrow(\lambda)+\alpha_{c(x)}$,   and  \eqref{weiw} still holds true.
 By induction assumption, we have   $\text{wt}(\mu) \preceq \text{wt}(\lambda)$ since
 $$\text{wt}_\uparrow(\mu)+\text{wt}_\downarrow(\mu) =\text{wt}_\uparrow(\lambda)+\text{wt}_\downarrow(\lambda), \quad  \text{wt}_\downarrow(\lambda)\leq \text{wt}_\downarrow(\mu).$$
\end{proof}
\begin{Cor}\label{hlink}
Suppose  $\lambda,\mu\in\bar\Lambda$. \begin{itemize}\item [(1)]  If  $L(\lambda)$ and $L(\mu)$ are in the  same block of $A$-mod,  then $\text{wt}_\uparrow(\lambda)+\text{wt}_\downarrow(\lambda)=\text{wt}_\uparrow(\mu)+\text{wt}_\downarrow(\mu)$.
\item [(2)]   If  $[\bar \Delta(\lambda):L(\mu)]\neq0$ and  $\lambda\neq\mu$, then $\text{wt}(\mu)\prec\text{wt}(\lambda)$.\end{itemize}
\end{Cor}
\begin{proof}Mimicking arguments in the proof of  \cite[Theorem~5.17]{GRS3} and using
Proposition~\ref{dedji} and Corollary~\ref{ijxxexeu}(2) yields   (1). We leave the details to the reader. (2) follows from Proposition~\ref{dedji} and Corollary~\ref{ijxxexeu}.
\end{proof}

Let
$ j^{ \ob a}: A_{\preceq  \ob a}\text{-lfdmod}\rightarrow   \bar A_{\ob a} \text{-fdmod}
$ be the exact idempotent truncation functor. Then   $j^{ \ob a} V  = \bar 1_{\ob a} V$
 for any $V\in A_{\preceq \ob a}\text{-lfdmod}$. Recall  $j^{\ob a}_!$ and $ j^{\ob a}_*$ in  \eqref{exa}. Then
$(j^{\ob a}_!,j^{\ob a} , j^{\ob a}_*)$ agree with the  adjoint triple between $A\text{-lfdmod}_{\preceq \ob a}$ and  its Serre quotient $A\text{-lfdmod}_{\preceq \ob a}/A\text{-lfdmod}_{\prec \ob a} $. In fact, this result is available for any locally unital algebra associated to an upper finite  weakly triangular category (see the proof of \cite[Theorem~3.5]{GRS3}).

\begin{Theorem}\label{COBMW11} The
  $A$-lfdmod is an upper finite fully  stratified category in the sense of \cite[Definition~3.36]{BS} with respect to the stratification $\text{wt}:\bar \Lambda\rightarrow P\times P$ in \eqref{nsjwhfun} with the order $\preceq$ on $P\times P$. \end{Theorem}

  \begin{proof}
  It is clear that wt is a new stratification  of $A\text{-lfdmod}$ in the sense of \cite[Definition~3.1]{BS}. Moreover, the image of wt (denoted by $\bar P$)
 is upper finite by its definition in \eqref{nsjwhfun}.
 By the well-known results on block decomposition of degenerate   cyclotomic Hecke algebras, $\bar P$ gives a block decomposition of $\bar A^\circ$-lfdmod.

 Suppose that $\bar A_{\ob a,(\rho,\sigma)}$-fdmod  is the block of $\bar A_\ob a$ indexed by $(\rho,\sigma)\in \bar P$  with $\ob a=(r,s)$ (i.e., $\rho$ is obtained from $-\omega_{\mathbf u'}$ by adding
 $r$ simple roots and $\sigma$ is obtained from $\omega_{\mathbf u}$ by subtracting $s$ simple roots).
  Let $A\text{-lfdmod}_{\preceq (\rho,\sigma)}$ be the Serre subcategory of $A\text{-lfdmod}$ generated by
 $\{L(\lambda)\mid \text{wt}(\lambda)\preceq (\rho,\sigma)\}$.
Note that $\lambda\in\bar\Lambda_{(r+k, s+k)}$ for some $k\in \mathbb N$ if
 $\text{wt}(\lambda)\preceq (\rho,\sigma)$. So, $A\text{-lfdmod}_{\preceq (\rho,\sigma)}$ is actually a Serre subcategory of $A_{\preceq \ob a}\text{-lfdmod}$.
 Similarly we have $A\text{-lfdmod}_{\prec (\rho,\sigma)}$.

 Let $j^{\ob a,(\rho,\sigma)}$ be the restriction of $j^\ob a$ to $A\text{-lfdmod}_{\preceq(\rho,\sigma)}$.
Since  $j^{\ob a,(\rho,\sigma)}(M)\in \bar A_{\ob a,(\rho,\sigma)}$-fdmod for any $M\in A_{\preceq \ob a}\text{-lfdmod}_{\preceq (\rho,\sigma)}$, $j^{\ob a,(\rho,\sigma)}$ is actually a functor from  $A\text{-lfdmod}_{\preceq(\rho,\sigma)}$ to $\bar A_{\ob a,(\rho,\sigma)}$-fdmod.
Moreover,  $j^{\ob a,(\rho,\sigma)}$ induces an equivalence of categories between $A \text{-lfdmod}_{\preceq(\rho,\sigma)}/A \text{-lfdmod}_{\prec(\rho,\sigma)}$ and $\bar A_{\ob a,(\rho,\sigma)}$-fdmod.

Let $j^{\ob a,(\rho,\sigma)}_!$  be the restriction of  $j^{\ob a }_!$   to $\bar A_{\ob a,(\rho,\sigma)}$-fdmod.
 Then  $j^{\ob a,(\rho,\sigma)}_!$ is actually  a functor from $\bar A_{\ob a,(\rho,\sigma)}$-fdmod to $A\text{-lfdmod}_{\preceq(\rho,\sigma)}$.
In fact, for any $M\in \bar A_{\ob a,(\rho,\sigma)}$-fdmod, $j^{\ob a,(\rho,\sigma)}_!(M)$ has a $\bar\Delta$-flag, and  $\bar\Delta(\mu)$ appears as a section if  $[M: D(\mu)]\neq 0$. In this case,  $\text{wt}(\mu)=(\rho,\sigma)$. By Corollary~\ref{hlink}(2),
we see that $j^{\ob a,(\rho,\sigma)}_!(M)\in A\text{-lfdmod}_{\preceq(\rho,\sigma)}$.
Furthermore, since  $(j^{\ob a}_!,j^{\ob a})  $ is an  adjoint  pair, so is
  $(j^{\ob a,(\rho,\sigma)}_!,j^{\ob a,(\rho,\sigma)})$.
  Hence the required standard and proper standard  objects coincide with those in Definition~\ref{stanpro}.

  Suppose $\lambda\in \bar \Lambda$. Thanks to
  Corollaries~\ref{ijxxexeu}(1) and \ref{hlink}(2), $\mathbf P(\lambda)$ has  a finite $\Delta$-flag such that $\Delta(\lambda)$ appears as the top section
and  other sections $\Delta(\mu)$  with $\text{wt}(\lambda)\preceq\text{wt}(\mu)$. So, $A$-lfdmod is an upper finite $+$-stratified category in the sense of \cite[Definition~3.36]{BS} with respect to the stratification $\text{wt}$.
It is   fully stratified  since $\Delta(\mu)$  has a finite $\bar \Delta$-flag
 with sections $\bar \Delta(\nu)$
such that  $\text{wt}(\nu)=\text{wt}(\mu)$.
  \end{proof}

  Let $K_0(\bar A^\circ \text{-pmod})$ be the Grothendieck group of $\bar  A^\circ \text{-pmod}$.
 Recall  $V(\omega_{\mathbf u})$ (resp., $\tilde V(-\omega_{\mathbf u'})$) is the integrable highest (resp., lowest) weight $\mathfrak g$-module of weight $\omega_{\mathbf u}$ (resp., $-\omega_{\mathbf u'}$).
Let $ \mathfrak g^\uparrow=\{y^\uparrow\mid y\in \mathfrak g\}$ and $\mathfrak g^\downarrow=\{y^\downarrow\mid y\in \mathfrak g\}$ be the two copies of $\mathfrak g$, where both $y^\uparrow$ and $y^\downarrow$ are $y$.
\begin{Prop}\label{dddddcate}
As   $\mathfrak g^\downarrow  \oplus \mathfrak g^\uparrow$-modules,     $$ \mathbb C\otimes_{\mathbb Z}K_0(\bar A^\circ\text{-pmod})  \cong \tilde V(-\omega_{\mathbf u'})\boxtimes V(\omega_\mathbf u),$$ where
the Chevalley generators $e_i^\uparrow$ and $ f_i^\uparrow $ (resp.,  $  e_i^\downarrow$ and $  f_i^\downarrow$)   act  on $  \mathbb C\otimes_{\mathbb Z}K_0(\bar A^\circ\text{-pmod})  $ via the endomorphisms induced  by  $E_i^\uparrow$ and  $F_i^\uparrow$ (resp.,  $E_i^\downarrow$ and $F_i^\downarrow$) if $i\in \mathbb I_{\mathbf u}$ (resp., $\mathbb I_{\mathbf u'}$) and $0$, otherwise.
 \end{Prop}
\begin{proof}This follows from \cite[\S~5.3]{K} (see also \cite[Remark~6.2]{GRS3}) and  \eqref{word}.
\end{proof}
Let $U(\mathfrak t)$ be  the universal enveloping  algebra of any Lie algebra $\mathfrak t$.  There is a Lie algebra homomorphism
from $\mathfrak g $ to $\mathfrak g^\downarrow  \oplus \mathfrak g^\uparrow$ sending $y$ to $y^\downarrow+y^\uparrow$.
This homomorphism induces  the  usual comultiplication  on
$U(\mathfrak g)$ since   $U(\mathfrak g)\otimes U(\mathfrak g)$ can be identified  with $U(\mathfrak g^\downarrow \oplus \mathfrak g^\uparrow )$. So,   $\tilde V(-\omega_{\mathbf u'})\boxtimes V(\omega_\mathbf u)$ becomes the $\mathfrak g$-module $\tilde V(-\omega_{\mathbf u'})\otimes V(\omega_\mathbf u)$
via the above homomorphism. Let $K_0(A \text{-mod}^\Delta)$ be the Grothendieck group of $A\text{-mod}^\Delta$.

\begin{Theorem}\label{categoricalactofa}
As $\mathfrak g$-modules, $\mathbb C\otimes_{\mathbb Z}K_0(A\text{-mod}^\Delta) \cong \tilde V(-\omega_{\mathbf u'})\otimes V(\omega_\mathbf u)$, where the Chevalley generators
$e_i,f_i$ act on $\mathbb C\otimes_{\mathbb Z}K_0(A\text{-mod}^\Delta)$ via the endomorphisms induced  by the  $E_i$ and $F_i$ for all $i\in\mathbb I$.
\end{Theorem}
\begin{proof}This follows from Proposition~\ref{dddddcate} and  \eqref{keyses}.
\end{proof}
\section{Proof of Theorem~\ref{main}}

 Recall that $\mathfrak g$ is  the    Kac-Moody Lie algebra in subsection~\ref{CA}.
The quiver Hecke category   $\mathcal {QH}$ associated to $\mathfrak g$
is a $\Bbbk$-linear strict monoidal category generated
 by objects $\mathbb I$ in \eqref{iu} and morphisms $$\begin{tikzpicture}[baseline=-.5mm]
\draw[->,thick,darkred] (0,-.3) to (0,.3);
      \draw (0,0) \rdot;
      \node at (0,-.4) {$\color{darkblue}\scriptstyle i$};
\end{tikzpicture}:  i\rightarrow  i, \ \ \    \begin{tikzpicture}[baseline = 2.5mm]
	\draw[->,thick,darkred] (0.28,0) to[out=90,in=-90] (-0.28,.6);
 \node at (-0.29,-.1) {$\color{darkblue}\scriptstyle i$};\node at (0.29,-.1) {$\color{darkblue}\scriptstyle j$};
	\draw[->,thick,darkred] (-0.28,0) to[out=90,in=-90] (0.28,.6);
\end{tikzpicture}: i\otimes j\rightarrow j\otimes i$$
subject to certain relations  in \cite[Definition~3.4]{BD}, where the parameters  $\{t_{ij}\in\Bbbk^\times\mid i,j\in\mathbb I\} $  and $\{s_{ij}^{mq}\in \Bbbk\mid 0<m<-a_{i,j}, 0<q<-a_{j,i}\}$ (only appear when $p=2$) are given   as follows:
$$t_{ij}= \left\{
     \begin{array}{ll}
       -1, & \hbox{$i=j-1$;} \\
       1, & \hbox{otherwise.}
     \end{array}
   \right., \ \ \ ~ s_{ij}^{11}=s_{ji}^{11}=2 \text{ for } i=j\pm1.
$$
For any $\mathbf i=(i_d,i_{d-1},\ldots, i_1)\in \mathbb I^d$,  $d>0$, we identify $\mathbf i$ with the object $i_d\otimes i_{d-1}\otimes\ldots\otimes  i_1\in \text{ob }\mathcal {QH}$.
The locally unital algebra associated to $\mathcal{QH}$ is $\bigoplus_{d\in \mathbb N} QH_d$, where
$$QH_d:= \bigoplus _{\mathbf i,\mathbf i'\in \mathbb I^d}\Hom_{\mathcal {QH}}(\mathbf i,\mathbf i'). $$
When $d>0$,
$QH_d$ is known as  the quiver Hecke algebra associated to $\mathfrak g$\cite{KL,Ro}.
It is  generated by
$$ \{e(\mathbf i)\mid \mathbf i\in \mathbb I^d\}\bigcup \{y_1,\ldots,y_d\}\bigcup \{\psi_1,\ldots,\psi_{d-1}\}$$
subject to the relations (1.7)--(1.15) in \cite[Theorem~1.1]{BK},
  where
$$e(\bar {\mathbf i})= \begin{tikzpicture}[baseline = -.5mm] \draw[->,thick,darkred](0,-.3) to (0,.3);\draw(0,-.4)node{\tiny$i_d$};\draw(0.8,-.4)node{\tiny$i_2$};\draw(1.2,-.4)node{\tiny$i_1$};
\draw(0.5,0) node{$ \cdots$};\draw[->,thick,darkred](0.8,-.3)to (0.8,.3);\draw[->,thick,darkred](1.1,-.3) to (1.1,.3);
 \end{tikzpicture},
~~ y_r e(\bar{\mathbf i})= \begin{tikzpicture}[baseline = -.5mm] \draw[->,thick,darkred](0,-.3) to (0,.3);\draw(0,-.4)node{\tiny$i_d$};\draw(0.8,-.4)node{\tiny$i_r$}; \draw(1.5,-.4)node{\tiny$i_1$};\draw (0.8,0) \rdot;\draw(1.1,0) node{$ \cdots$};
\draw(0.5,0) node{$ \cdots$};\draw[->,thick,darkred](0.8,-.3)to (0.8,.3);\draw[->,thick,darkred](1.5,-.3) to (1.5,.3);
 \end{tikzpicture}, ~~\psi_re({\bar {\mathbf i}})=\begin{tikzpicture}[baseline = -.5mm] \draw[->,thick,darkred](0,-.3) to (0,.3);\draw(0,-.4)node{\tiny$i_d$}; \draw(1.5,-.4)node{\tiny$i_1$};\draw(0.83,-.4)node{\tiny$i_r$};
\draw[->,thick,darkred] (0.85,-.3) to[out=90,in=-90] (0.55,.3);\draw[->,thick,darkred](0.95,-.3)to (0.95,.3);\draw[->,thick,darkred](0.4,-.3)to (0.4,.3);
\draw[->,thick,darkred] (0.55,-.3) to[out=90,in=-90] (0.85,.3);
\draw(1.2,0) node{$ \cdots$};
\draw(0.2,0) node{$ \cdots$};
\draw[->,thick,darkred](1.5,-.3) to (1.5,.3);
 \end{tikzpicture},
 $$
and $\bar {\mathbf i}=(i_1, i_2, \ldots, i_d)$.
For any $\mu\in P^+$ given in \eqref{wei1}, let $QH_d(\mu)$ be the cyclotomic quotient of $QH_d$ by the two-sided ideal generated by $\{y_1^{\langle h_{i_1},\mu\rangle} e(\bar {\mathbf i})\mid \mathbf i\in\mathbb I^d\}$.

Following \cite{Dav}, let $\mathcal{AH}$ be the $\Bbbk$-linear strict monoidal category generated by the single
  object $\downarrow$ and  two morphisms  \begin{tikzpicture}[baseline = 2.5mm]
	\draw[<-,thick,darkblue] (0.28,0) to[out=90,in=-90] (-0.28,.6);
	\draw[<-,thick,darkblue] (-0.28,0) to[out=90,in=-90] (0.28,.6);
\end{tikzpicture} and \begin{tikzpicture}[baseline=-.5mm]
\draw[<-,thick,darkblue] (0,-.3) to (0,.3);
      \draw (0,0) \bdot;
\end{tikzpicture} satisfying the relations \eqref{relation 9}--\eqref{relation 11}.
Then   $\mathcal{AH}$,   known as the  degenerate affine Hecke category,  can be considered as
 the subcategory of $\AOB$ generated by the single object $\downarrow$ and morphisms  \begin{tikzpicture}[baseline = 2.5mm]
	\draw[<-,thick,darkblue] (0.28,0) to[out=90,in=-90] (-0.28,.6);
	\draw[<-,thick,darkblue] (-0.28,0) to[out=90,in=-90] (0.28,.6);
\end{tikzpicture} and \begin{tikzpicture}[baseline=-.5mm]
\draw[<-,thick,darkblue] (0,-.3) to (0,.3);
      \draw (0,0) \bdot;
\end{tikzpicture}. Let
   $$AH_d:=\Hom_{\mathcal{AH}}(\downarrow^{\otimes d},\downarrow^{\otimes d} ),$$
     the degenerate affine Hecke algebra~\cite{Dav}. Following \cite{Dav}, define
$$ x_r=  \begin{tikzpicture}[baseline = 12.5, scale=0.35, color=\clr]
       \draw[>-,thick](0,1.1)to[out= down,in=up](0,2.6);
       \draw(0.6,1) node{$ \cdots$}; \draw(0.6,2.5) node{$ \cdots$};
       \draw[<-,thick](1.8,0.8)to (1.8,2.8);
       \draw(3,3)node{\tiny$r$}; \draw(5.5,3.2)node{\tiny$1$};
        \draw[<-,thick] (3,0.8) to[out=up, in=down] (3,2.8);
         \draw[>-,thick](4,1.1)to[out= down,in=up](4,2.7);\draw (3, 1.8) \bdot;
         \draw(4.6,1.1) node{$ \cdots$}; \draw(4.6,2.5) node{$ \cdots$};
         \draw[>-,thick](5.5,1.1)to[out= down,in=up](5.5,2.7);
     \end{tikzpicture}\in AH_d.$$
For any $\mu\in P^+$, let $AH_d(\mu)$ be the cyclotomic quotient of $AH_d$ by the two-sided ideal generated by $$g(x_1)=\prod_{i\in \mathbb I}(x_1-i)^{\langle h_i,\mu\rangle}.$$
 The algebra $AH_d(\mu)$ is isomorphic to some  degenerated  cyclotomic Hecke algebra in subsection~\ref{hecke}.  At moment, we use $AH_d(\mu)$ so as to  compare it with $QH_d(\mu)$.
It is known that there is a  set of mutually orthogonal idempotents of  $AH_d(\mu)$,    denoted by $\{1_\mathbf i\mid \mathbf i\in\mathbb I^d\}$ such that, for any  $AH_d(\mu)$-module $M$
$$1_\mathbf i M=\bigcap_{k=1}^d  M_{i_k},$$ where $M_{i_k}$ is the
$i_k$-generalized eigenspace of $M$ of $x_k$.
If $\langle h_i,\mu, \rangle\leq \langle h_i,\mu'\rangle$ for all $i\in \mathbb I$, then there are epimorphisms $$QH_d(\mu')\twoheadrightarrow QH_d(\mu), \quad  AH_d(\mu')\twoheadrightarrow AH_d(\mu).$$ So,    $\{QH_d(\mu)\mid \mu\in P^+\}$ and $\{AH_d(\mu)\mid \mu\in P^+\}$  form two  inverse systems of locally unital algebras.
 Taking inverse limits yields two   completions
$$\widehat {AH}_d:= \lim_{\leftarrow}AH_d(\mu), ~\widehat {QH}_d:= \lim_{\leftarrow}QH_d(\mu).$$
It is proved in \cite[Theorem~1.1]{BK} that  there is an isomorphism of locally unital algebras $QH_d(\mu) \cong {AH}_d(\mu)$
such that
\begin{equation}\label{isohecke}
 \begin{tikzpicture}[baseline = -.5mm] \draw[->,thick,darkred](0,-.3) to (0,.3);\draw(0,-.4)node{\tiny$i_d$};\draw(0.8,-.4)node{\tiny$i_2$};\draw(1.2,-.4)node{\tiny$i_1$};
\draw(0.5,0) node{$ \cdots$};\draw[->,thick,darkred](0.8,-.3)to (0.8,.3);\draw[->,thick,darkred](1.1,-.3) to (1.1,.3);
 \end{tikzpicture}\mapsto 1_\mathbf i, \ \ \  \begin{tikzpicture}[baseline = -.5mm] \draw[->,thick,darkred](0,-.3) to (0,.3);\draw(0,-.4)node{\tiny$i_d$};\draw(0.8,-.4)node{\tiny$i_r$}; \draw(1.5,-.4)node{\tiny$i_1$};\draw (0.8,0) \rdot;\draw(1.1,0) node{$ \cdots$};
\draw(0.5,0) node{$ \cdots$};\draw[->,thick,darkred](0.8,-.3)to (0.8,.3);\draw[->,thick,darkred](1.5,-.3) to (1.5,.3);
 \end{tikzpicture}\mapsto (x_r-i_r)1_\mathbf i.
\end{equation}
These isomorphisms induces an isomorphism of locally unital algebras \begin{equation}\label{isowideaff}\widehat{QH}_d \cong\widehat{AH}_d. \end{equation}
See also \cite[\S~3.2.6]{Ro} and  \cite[Theorem~3.10]{We}. Moreover, there is a   locally unital  embedding  \begin{equation}\label{emb1} QH_d\hookrightarrow\widehat{QH}_d.\end{equation}

Now, we go on studying the cyclotomic oriented Brauer category $\OB(\mathbf u,\mathbf u')$. Recall  $A$  is  the locally unital $\Bbbk$-algebra   associated to $\OB(\mathbf u,\mathbf u')$ in  \eqref{Ba1}.
There are endofunctors $E$, $F$ in  \eqref{EF}. For any  $i\in\mathbb I$ there are endofunctors  $E_i$, $F_i$ in ~\eqref{EFi}. Thanks to Lemma~\ref{jme}(1),  there is an $(A, A)$-homomorphism $x_\downarrow: A_\downarrow\rightarrow A_\downarrow $ defined on $A_{\downarrow} 1_b$ by right multiplication by $1_b \xdx$. So, $x^\uparrow$ and $x_\downarrow$ are intertwined by the isomorphism $ {_\uparrow}A \cong A_\downarrow$ in Proposition~\ref{isob}, where $x^\uparrow$ is given
in Lemma~\ref{jme}(2). Moreover,  $x_\downarrow$  induces a natural transformation  $ \begin{tikzpicture}[baseline=-.5mm]
\draw[<-,thick,darkblue] (0,-.3) to (0,.3);
     \draw (0,0) \bdot;
\end{tikzpicture}: E\rightarrow E$ such that $$\left(\begin{tikzpicture}[baseline=-.5mm]
\draw[<-,thick,darkblue] (0,-.3) to (0,.3);
     \draw (0,0) \bdot;
\end{tikzpicture}\right)_M= x_\downarrow\otimes \text{Id}:  A_\downarrow\otimes_AM \rightarrow  A_{\downarrow}\otimes_AM.$$
Similarly, there is  a natural transformation denoted by $\begin{tikzpicture}[baseline = 2.5mm]
	\draw[<-,thick,darkblue] (0.28,0) to[out=90,in=-90] (-0.28,.6);
	\draw[<-,thick,darkblue] (-0.28,0) to[out=90,in=-90] (0.28,.6);
\end{tikzpicture}: E^2\rightarrow E^2$ such that
$$\left(\begin{tikzpicture}[baseline = 2.5mm]
	\draw[<-,thick,darkblue] (0.28,0) to[out=90,in=-90] (-0.28,.6);
	\draw[<-,thick,darkblue] (-0.28,0) to[out=90,in=-90] (0.28,.6);
\end{tikzpicture}\right)_M: A_{\downarrow\downarrow}\otimes_AM \rightarrow ~A_{\downarrow\downarrow}\otimes_AM, \quad y\otimes m \mapsto y\circ(1_a \begin{tikzpicture}[baseline = 2.5mm]
	\draw[<-,thick,darkblue] (0.28,0) to[out=90,in=-90] (-0.28,.6);
	\draw[<-,thick,darkblue] (-0.28,0) to[out=90,in=-90] (0.28,.6);
\end{tikzpicture})\otimes m, $$
 where $(y, m)\in A 1_{a\downarrow\downarrow} \times M$ and   $A_{\downarrow\downarrow}= A_\downarrow\otimes_A  A_\downarrow  $ in the obvious way.

\begin{Lemma}\label{affhecke}
There is a strict monoidal functor $ \Psi: \mathcal {AH}\rightarrow \END( A\text{-lfdmod})$ such that
$$ \Psi(\downarrow)= E,\quad  \Psi(\begin{tikzpicture}[baseline = 2.5mm]
	\draw[<-,thick,darkblue] (0.28,0) to[out=90,in=-90] (-0.28,.6);
	\draw[<-,thick,darkblue] (-0.28,0) to[out=90,in=-90] (0.28,.6);
\end{tikzpicture} )=\begin{tikzpicture}[baseline = 2.5mm]
	\draw[<-,thick,darkblue] (0.28,0) to[out=90,in=-90] (-0.28,.6);
	\draw[<-,thick,darkblue] (-0.28,0) to[out=90,in=-90] (0.28,.6);
\end{tikzpicture},\quad  \Psi(\begin{tikzpicture}[baseline=-.5mm]
\draw[<-,thick,darkblue] (0,-.3) to (0,.3);
      \draw (0,0) \bdot;
\end{tikzpicture} )=\begin{tikzpicture}[baseline=-.5mm]
\draw[<-,thick,darkblue] (0,-.3) to (0,.3);
    \draw (0,0) \bdot;
\end{tikzpicture}. $$
\end{Lemma}
\begin{proof}The  result follows directly from  \eqref{relation 9}--\eqref{relation 11}.
\end{proof}
\begin{Lemma}\label{quvier hecke action} For all   $i,j\in\mathbb I$, there
 are natural transformations
$$\begin{tikzpicture}[baseline=-.5mm]
\draw[->,thick,darkred] (0,-.3) to (0,.3);
     \draw (0,0) \rdot;
      \node at (0,-.4) {$\color{darkblue}\scriptstyle i$};
\end{tikzpicture}: E_i\rightarrow E_i, \text{ and }   \begin{tikzpicture}[baseline = 2.5mm]
	\draw[->,thick,darkred] (0.28,0) to[out=90,in=-90] (-0.28,.6);
 \node at (-0.29,-.1) {$\color{darkblue}\scriptstyle i$};\node at (0.29,-.1) {$\color{darkblue}\scriptstyle j$};
	\draw[->,thick,darkred] (-0.28,0) to[out=90,in=-90] (0.28,.6);
\end{tikzpicture}: E_i\circ E_j\rightarrow E_j\circ E_i,$$
 which induce  a strict monoidal functor from $\mathcal {QH}$ to $\END( A\text{-lfdmod})$.
\end{Lemma}
\begin{proof}
By Lemma~\ref{affhecke}, there is an algebra homomorphism $ \Psi_d: AH_d\rightarrow \text{End}(E^d)$ for any $d>0$. Thanks to Lemma~\ref{bijiont}, $E$ and $F$ are biadjoint to each other. So,
$E$ is a sweet endofunctor of $A$-lfdmod in the sense of \cite[Definition~2.10]{BD}.
  By \cite[Theorem~2.11]{BD},
  $E$ preserve the set of finitely generated projective modules.
  Therefore, $E^dP$ is finitely generated projective for any  $P\in A$-pmod. In particular, $E^dP\in A$-lfdmod since any finitely generated left $A$-module has to be locally finite dimensional (see \cite[Section~2.2]{BD}).  It is proved in \cite[Section~2.2]{BD} that $dim\text{Hom}_A (V, W)<\infty$ for any finitely generated left $A$-module $V$, and any locally finite dimensional left $A$-module $W$. In particular,  $\text{End}_{ A}(E^dP)$ is finite dimensional for each  $P\in A$-pmod and hence  $(x_1)_P$ has a minimal polynomial with roots in $\mathbb I$ (see \eqref{EFi}). So,   the composition
$$ \text{ev}_{P}\circ \Psi_d: AH_d\rightarrow \text{End}_{ A}(E^dP)$$
of  $ \Psi_d$ with the evaluation at $P$ factors through some cyclotomic quotients $AH_d(\mu)$ and hence factors through all sufficient large cyclotomic quotients $AH_d(\mu')$.
Note that any two natural transformations $\eta,\xi: E^d\rightarrow E^d$ are equal  if $\eta_P=\xi_P$  for each $P\in A$-pmod.
So, $ \text{ev}_{M}\circ \Psi_d$  factors through all sufficient large cyclotomic quotients $AH_d(\mu')$ for any $M\in A$-lfdmod.
This shows that $\Psi_d$ factors through $\widehat{AH}_d$. Let
\begin{equation}\label{sjwjsss} \hat \Psi_d: \widehat{AH}_d\rightarrow \text{End}(E^d).\end{equation}
    Composing $\hat \Psi_d$ with the isomorphism in \eqref{isowideaff}  and the inclusion in \eqref{emb1}
yields  an algebra homomorphism $$\Phi_d: QH_d\rightarrow\text{End}(E^d).$$ Moreover, by  the argument for the  proof of \cite[(6.13)]{BD1} (i.e., by Lemma~\ref{affhecke} and  the explicit formulae of the isomorphism $\widehat{QH}_d \cong\widehat{AH}_d $   induced by  \cite[Theorem~1.1]{BK} or    the non-degenerate analogue in \cite[Proposition~3.10]{We}), these maps are compatible with the monoidal structure in the categories $\mathcal {QH}$ and  $\END( A\text{-lfdmod})$ such that  there is   a monoidal functor $\Phi: \mathcal {QH}\mapsto\END( A\text{-lfdmod})$ given by $\Phi(i)=E_i$ and $\Phi(h)=\Phi_d(h)$ if $h\in QH_d$.
\end{proof}

 \begin{Defn}\label{defioftpc}\cite[Definition~3.2, Remark~3.6]{LW} Fix $\omega_\mathbf u$ and $\omega_{\mathbf u'}$ in \eqref{param}.
The data of a tensor product categorification of $\tilde V(-\omega_{\mathbf u'})\otimes V(\omega_\mathbf u)$ consists of two parts:
\begin{itemize}
\item [(TPC1)] A (locally) Schurian category $\mathcal C$ such that the isomorphism classes of irreducible objects labeled by the indexing set for the basis of $\tilde V(-\omega_{\mathbf u'})\otimes V(\omega_\mathbf u)$.
\item [(TPC2)] A nilpotent categorical action (in the sense of \cite[Definition~4.25]{BD}) making $\mathcal C$ into a 2-representation (e.g., \cite[Definition~5.1.1]{Ro} or \cite[Definition~4.1]{BD}) of the associated Kac-Moody 2-category $ \mathfrak{U}(\mathfrak g)$ (e.g., \cite[\S~4.1.3]{Ro} or \cite[Definition~3.8]{BD}).
\end{itemize}
These two pieces of data need to satisfy some compatibility conditions as follows:
\begin{itemize}
\item[(TPC3)] The category $\mathcal C$ is an  upper finite fully stratified category in the sense of Brundan and Stroppel\cite{BS} with poset $P\times P$.
    Moreover, the poset structure  of $P\times P$ is given by the ``inverse dominance order" $\preceq$ on $P\times P$
    in \eqref{preceq1}.
\item[(TPC4)]There is a categorical $(\mathfrak g^\downarrow\oplus \mathfrak g^\uparrow)$-action on $ \text{gr }\mathcal C:= \bigoplus_{(\lambda,\mu)\in P\times P}\mathcal C_{(\lambda,\mu)}$ such that
    $$\mathbb C\otimes K_0(\text{proj-}\text{gr}~ \mathcal C) \cong\tilde V(-\omega_{\mathbf u'})\boxtimes V(\omega_\mathbf u)$$
     as $(\mathfrak g^\downarrow\oplus \mathfrak g^\uparrow)$-modules,
where \begin{itemize}\item[(1)] $\mathcal C_{(\lambda,\mu)}:=\mathcal C_{\preceq(\lambda,\mu)}/\mathcal C_{\prec(\lambda,\mu)} $,
\item [(2)]
 $ \mathfrak g^\uparrow=\{x^\uparrow\mid x\in \mathfrak g\}$ and $\mathfrak g^\downarrow=\{x^\downarrow\mid x\in \mathfrak g\}$ are  two copies of $\mathfrak g$,
  \item [(3)]  $\text{proj-}\text{gr}~\mathcal C$ is the subcategory of finite generated projective modules of  $\text{gr}~\mathcal C$, \item [(4)]  $\mathcal C_{(-\omega_\mathbf u',~ \omega_{\mathbf u})}\cong \text{Vec}_ \Bbbk$ (the category of vector space over $\Bbbk$).\end{itemize}
The categorification functors for $\mathfrak g^\diamond$  will be denoted by  $E_i^\diamond$ and $F_i^\diamond$, where $i\in\mathbb I$ and  $\diamond \in\{\uparrow,\downarrow\}$.
\item[(TPC5)] For any  $i\in\mathbb I$, denote by   $E_i$ and $F_i$  the categorification functors for $\mathfrak g$. There is a compatibility  between the categorical $\mathfrak g$-action on $\mathcal C$ and the categorical $(\mathfrak g^\downarrow\oplus \mathfrak g^\uparrow)$-
    action on $\text{gr}\  \mathcal C$ in the sense that there are short exact sequences
    \begin{equation} \label{sss123} \begin{aligned}
    &0\rightarrow \Delta\circ E_i^\uparrow\rightarrow E_i\circ \Delta \rightarrow \Delta\circ E_i^\downarrow\rightarrow 0,\\
    &0\rightarrow \Delta\circ F_i^\downarrow\rightarrow F_i\circ \Delta \rightarrow \Delta\circ F_i^\uparrow\rightarrow 0,
    \end{aligned} \end{equation}
where $\Delta=\bigoplus_{(\lambda,\mu)\in P\times P} \Delta_{(\lambda,\mu)}$ and $\Delta_{(\lambda,\mu)}: \mathcal C_{\preceq(\lambda,\mu) }\rightarrow \mathcal C_{(\lambda,\mu) }$ is the corresponding  standardization functor of the stratified category $\mathcal C$.
\end{itemize}
\end{Defn}

 \textbf{ Proof of   Theorem~\ref{main}}:  Thanks to Definition~\ref{defioftpc}, we need to   verify that  $ A$-lfdmod  admits the structures in  (TPC1)--(TPC5).

By  Corollary~\ref{irr} and Lemma~\ref{dddddcate}, we have results on the classification of irreducible objects in $A$-lfdmod. Now, (TPC1) follows from Theorem~\ref{categoricalactofa} which gives a bijection between labellings.
 (TPC3) follows from  Theorem~\ref{COBMW11} which says that  $ A$-lfdmod is an  upper finite fully stratified category in the sense of Brundan and Stroppel\cite{BS} with respect to the stratification function in \eqref{nsjwhfun}.
   (TPC4) follows from Lemma~\ref{dddddcate} and the well-known categorical $\mathfrak g$-action on  the representations of    degenerate cyclotomic Hecke algebras~\cite[Theorem~ 4.18]{BK1}~\cite[Theorem~4.25]{Ro1}.
The short exact sequences in \eqref{sss123}       follow from \eqref{keyses} in Theorem ~\ref{usuactifuc1} and hence  (TPC5) follows. So, it remains to verify (TPC2). Thanks to  \cite[Theorem~4.27]{BD}, it suffices  to check the following conditions in  \cite[Definition~4.25]{BD}:
\begin{itemize}
\item[(1)] A weight decomposition of the category $A$-lfdmod$=\bigoplus_{\sigma\in P\times P} A\text{-lfdmod}_\sigma$.
\item[(2)] Biadjoint endofunctors $E=\bigoplus_{i\in\mathbb I}E_i$ and $F=\bigoplus_{i\in \mathbb I}F_i$ such that $(E_i, F_i)$ are biadjoint functors.
\item[(3)]For all   $i,j\in\mathbb I$, there
 are natural transformations
$$\begin{tikzpicture}[baseline=-.5mm]
\draw[->,thick,darkred] (0,-.3) to (0,.3);
     \draw (0,0) \rdot;
      \node at (0,-.4) {$\color{darkblue}\scriptstyle i$};
\end{tikzpicture}: E_i\rightarrow E_i, \text{ and }   \begin{tikzpicture}[baseline = 2.5mm]
	\draw[->,thick,darkred] (0.28,0) to[out=90,in=-90] (-0.28,.6);
 \node at (-0.29,-.1) {$\color{darkblue}\scriptstyle i$};\node at (0.29,-.1) {$\color{darkblue}\scriptstyle j$};
	\draw[->,thick,darkred] (-0.28,0) to[out=90,in=-90] (0.28,.6);
\end{tikzpicture}: E_i\circ E_j\rightarrow E_j\circ E_i,$$
 which induce  a strict monoidal functor from $\mathcal {QH}$ to $\END( A\text{-lfdmod})$.
\item[(4)] The endomorphisms $[E_i]$ and $[F_i]$ make $\mathbb C\otimes_\mathbb Z K_0( A\text{-pmod})$ into a well-defined $\mathfrak g$-module with $\sigma$-weight space $\mathbb C\otimes_\mathbb Z K_0(A\text{-pmod}_\sigma)$.
\item[(5)] For any $i\in \mathbb I$ and any finitely generated left $A$-module $M$, the endomorphism $$(\begin{tikzpicture}[baseline=-.5mm]
\draw[->,thick,darkred] (0,-.3) to (0,.3);
   \draw (0,0) \rdot;
      \node at (0,-.4) {$\color{darkblue}\scriptstyle i$};
\end{tikzpicture})_M: E_iM\rightarrow E_iM$$  is nilpotent.
\end{itemize}

In fact, (1) follows from  the partial result on the blocks of $A$-lfdmod in  Corollary~\ref{hlink}. In this case, we  set
$A$-lfdmod$_\sigma$ to be the Serre subcategory of $A$-lfdmod generated by $L(\lambda)$ such that $\text{wt}_\uparrow(\lambda)+\text{wt}_\downarrow(\lambda)=\sigma $.
The biadjoint functors in (2) are given in Lemma~\ref{bimoudisomah} and \eqref{EFi}. (3) follows from  Lemma~\ref{quvier hecke action}.
 Thanks to Corollary~\ref{ijxxexeu}, $\mathbb C\otimes_\mathbb Z K_0( A\text{-pmod})$ can be identified  with $\mathbb C\otimes _\mathbb Z K_0(A\text{-mod}^\Delta)$ and hence   (4) follows  Theorem~\ref{categoricalactofa}.
Finally, by \eqref{isohecke}, under the isomorphism  in \eqref{isowideaff}  we have
$$ \left( \begin{tikzpicture}[baseline=-.5mm]
\draw[->,thick,darkred] (0,-.3) to (0,.3);
     \draw (0,0) \rdot;
      \node at (0,-.4) {$\color{darkblue}\scriptstyle i$};
\end{tikzpicture}\right)_M=\left((\begin{tikzpicture}[baseline=-.5mm]
\draw[<-,thick,darkblue] (0,-.3) to (0,.3);
    \draw (0,0) \bdot;
\end{tikzpicture})_M-i \text{Id} \right)_{E_iM}.$$
Therefore, it  suffices to show that there is bound on the Jordan block sizes of $$x_\downarrow\otimes \text{Id}=(\begin{tikzpicture}[baseline=-.5mm]
\draw[<-,thick,darkblue] (0,-.3) to (0,.3);
     \draw (0,0) \bdot;
\end{tikzpicture})_M: A_\downarrow\otimes_A M\rightarrow A_\downarrow\otimes_A M$$
for any finitely generated left $A$-module $M$.
Since $E$ is a sweet functor (e.g., Lemma~\ref{bijiont})   of $A$-lfdmod in the sense of \cite[Definition~2.10]{BD},
  by \cite[Theorem~2.11]{BD},
  $E$ preserve the set of finitely generated   modules.
  Therefore, $EM=A_\downarrow\otimes_A M$ is finitely generated. Since $A$ is locally finite dimensional, $\End_{A}(EM)$ is finite dimensional and hence $\left(\begin{tikzpicture}[baseline=-.5mm]
\draw[<-,thick,darkblue] (0,-.3) to (0,.3);
    \draw (0,0) \bdot;
\end{tikzpicture}\right)_M$
    has a minimal polynomial. So, there is bound on the Jordan block sizes. This completes the proof of (5).\qed

\end{document}